\documentclass[a4paper,10pt]{amsart}
	
  \usepackage[top=1in,bottom=1in,left=1.2in,right=1.2in]{geometry}	
  \usepackage{mathdots}

 \usepackage{amsmath}		   			                
  \usepackage{amsthm}					                
  \usepackage{amssymb}  				    	        
  \usepackage{mathrsfs}					                
  \usepackage{color}					    	        
  \usepackage[all]{xy}  	  		                    
  \usepackage{fancyhdr}                             	

  \usepackage{hyperref}
	
  \theoremstyle{remark}         \newtheorem{remark}{Remark}[section]
  \theoremstyle{remark} 		\newtheorem{example}{Example}[section]
  \theoremstyle{definition}		\newtheorem{definition}{Definition}[section]
  \theoremstyle{plain}			\newtheorem{theorem}{Theorem}[section]
  \theoremstyle{plain}			\newtheorem{lemma}[theorem]{Lemma}
  \theoremstyle{plain}			\newtheorem{corollary}[theorem]{Corollary}
  \theoremstyle{plain}			\newtheorem{proposition}[theorem]{Proposition}

  \DeclareMathOperator{\im}{im}				        	    
  \DeclareMathOperator{\Hom}{Hom}				            
  \DeclareMathOperator{\Ext}{Ext}				            
  \DeclareMathOperator{\id}{id}				                

  \newcommand{\varphantom}[1]{\mathrel{\phantom{#1}}}
  \newcommand{\pl}{\partial}
  
  \newcommand{\To}{\longrightarrow}
  \newcommand{\ot}{\otimes}

  \newcommand{\al}{\alpha}
  \newcommand{\be}{\beta}
  \newcommand{\ga}{\gamma}
  \newcommand{\Si}{\Sigma}
  \newcommand{\si}{\sigma}
  \newcommand{\om}{\omega}
  \newcommand{\Om}{\Omega}
  
  \newcommand{\de}{\delta}
  \newcommand{\De}{\Delta}
  \newcommand{\lam}{\lambda}

  \newcommand{\mbb}[1]{\mathbb{#1}}			    	
  \newcommand{\mbf}[1]{\mathbf{#1}}			    	
  \newcommand{\scr}[1]{\mathscr{#1}} 				
  \newcommand{\mc}[1]{\mathcal{#1}}			    	
  \newcommand{\mf}[1]{\mathfrak{#1}}				
  \newcommand{\nan}{\mathbb{N}}			    	    
  \newcommand{\inn}{\mathbb{Z}} 				    
  \newcommand{\DD}{\mathbf{D}}                    
  \usepackage{cases}

  \newcommand{\simp}{\mathrm{simp}}
  \newcommand{\GS}{\mathrm{GS}}
  
  \newcommand{\Proj}{\operatorname{Proj}}
  \newcommand{\bou}{\pmb u}
  \newcommand{\bov}{\pmb v}
  \newcommand{\bow}{\pmb w}
  \newcommand{\Tot}{\operatorname{Tot}}
  \newcommand{\ppl}{{}^\circ\pl}
  \newcommand{\pmul}{{}^\circ\!\mu}
  \newcommand{\Qch}{\mathsf{Qch}}
  \newcommand{\Def}{\mathrm{Def}}
  \newcommand{\tw}{\mathrm{tw}}
  \newcommand{\ab}{\mathrm{ab}}
  \newcommand{\Hoch}{\mathrm{Hoch}}

  \allowdisplaybreaks
  \numberwithin{equation}{section}

\renewcommand{\lim}{\mathrm{lim}}

\newcommand{\CC}{\mathbf{C}}

\newcommand{\aaa}{\ensuremath{\mathcal{A}}}

\newcommand{\hhh}{\ensuremath{\mathcal{H}}}

\newcommand{\ttt}{\ensuremath{\mathcal{T}}}

\title{Hochschild cohomology of projective hypersurfaces}

\author{Liyu Liu}
\address[Liyu Liu]{Departement Wiskunde-Informatica, Universiteit Antwerpen, Middelheimcampus,
Middelheimlaan 1,
2020 Antwerp, Belgium}
\curraddr{School of Mathematical Science, Yangzhou University, No.\ 180 Siwangting Road, 225002 Yangzhou, Jiangsu, China}
\email{lyliu@yzu.edu.cn}

\author{Wendy Lowen}
\address[Wendy Lowen]{Departement Wiskunde-Informatica, Universiteit Antwerpen, Middelheimcampus,
Middelheimlaan 1,
2020 Antwerp, Belgium}
\email{wendy.lowen@uantwerpen.be}

\thanks{The authors acknowledge the support of the European Union for the ERC grant No 257004-HHNcdMir and the support of the Research Foundation Flanders (FWO) under Grant No.~G.0112.13N}

\begin{document}
\maketitle

\setlength{\baselineskip}{3.5ex plus 0.3ex minus 0.1ex}     

\begin{abstract}
We compute Hochschild cohomology of projective hypersurfaces starting from the Gerstenhaber-Schack complex of the (restricted) structure sheaf. We are particularly interested in the second cohomology group and its relation with deformations. We show that a projective hypersurface is smooth if and only if the classical HKR decomposition holds for this group. In general, the first Hodge component describing scheme deformations has an interesting inner structure corresponding to the various ways in which first order deformations can be realized: deforming local multiplications, deforming restriction maps, or deforming both. We make our computations precise in the case of quartic hypersurfaces, and compute explicit dimensions in many examples.
\end{abstract}

\section{Introduction}

Hochschild cohomology originated as a cohomology theory for associative algebras, which is known to be closely related to deformation theory since the work of Gerstenhaber. Meanwhile, both the cohomology and the deformation side of the picture have been developed for a variety of mathematical objects, ranging from schemes \cite{swan} \cite{kontsevich2} to abelian \cite{lowenvandenberghab}, \cite{Lowen-VandenBergh:hoch} and differential graded \cite{kellerdih}, \cite{lowenvandenbergh:curv} categories. One of the first generalizations considered after the algebra case was the case of presheaves of algebras, as thoroughly investigated by Gerstenhaber and Schack \cite{gerstenhaberschack}, \cite{Gerstenhaber-Schack:survey}, \cite{gerstenhaberschack2}. For a presheaf $\aaa$, Hochschild cohomology is defined as an $\Ext$ of bimodules $\Ext_{\aaa \textrm{-} \aaa}(\aaa, \aaa)$ in analogy with the algebra case. An important tool in the study of this cohomology is the (normalized, reduced) Gerstenhaber-Schack double complex $\CC(\aaa)$. We denote its associated total complex by $\mbf C_{\GS}(\mc A)$, and the cohomology of this complex by $H^n_{\mathrm{GS}}(\aaa) = H^n\CC_{\mathrm{GS}}(\aaa)$. We have $H^n_{\mathrm{GS}}(\aaa) \cong \Ext^n_{\aaa \textrm{-} \aaa}(\aaa, \aaa)$.
Unlike what the parallel result for associative algebras may lead one to expect, in general $H^2_{\GS}(\mc A)$ is not identified with the family of first order deformations of the presheaf $\mc A$.  A correct interpretation of $H^2_{\GS}(\mc A)$ is as the family of first order deformations of $\mc A$ as a \emph{twisted} presheaf, and an explicit isomorphism
\begin{equation}\label{defint}
H^2_{\GS}(\mc A)\To\Def_{\tw}(\mc A)
\end{equation}
is given in \cite[Thm.\ 2.21]{DLL:defo-qch}. Moreover, in loc.\ cit.,  if $\aaa$ is quasi-compact semi-separated, the existence of a bijective correspondence between the first order deformations of $\mc A$ as a twisted presheaf and the abelian deformations of the category $\Qch(\mc A)$ of quasi-coherent sheaves is proven. Hence in this case there are isomorphisms $H^2_{\GS}(\mc A)\cong\Def_{\tw}(\mc A)\cong\Def_{\ab}(\Qch(\mc A))$.

Throughout, let $k$ be an algebraically closed field of characteristic zero. The situation when $\mc A$ is a presheaf of commutative $k$-algebras over $\mf V$ is very interesting. As discussed in \cite{gerstenhaberschack}, in this case the complex ${\mbf C}_{\GS}(\mc A)$ admits the Hodge decomposition of complexes
\begin{equation}
{\mbf C}_{\GS}(\mc A) = \bigoplus_{r\in\nan}{\mbf C}_{\GS}(\mc A)_r,
\end{equation}
which induces the Hodge decomposition of the cohomology groups $H^n_{\GS}(\mc A)$ in terms of the cohomology groups $H^n_{\mathrm{GS}}(\aaa)_r = H^n{\mbf C}_{\GS}(\mc A)_r$:
\begin{equation}\label{Hodgedecintro}
H^n_{\GS}(\mc A) = \bigoplus_{r\in\nan}H^n_{\mathrm{GS}}(\aaa)_r.
\end{equation}

 The zero-th Hodge complex ${\mbf C}_{\GS}(\mc A)_0$ is nothing but the simplicial cohomology complex of $\aaa$, and the first Hodge complex ${\mbf C}_{\GS}(\mc A)_1$, which is called the asimplicial Harrison complex in \cite{gerstenhaberschack}, classifies first order deformations of $\aaa$ as a commutative presheaf. Hence, in this case the map \eqref{defint} naturally restricts to
\begin{equation}
H^2_{\GS}(\mc A)_1\To\Def_{\mathrm{cpre}}(\mc A).
\end{equation}
Recall that a GS $n$-cochain has $n+1$ components coming from the double complex $\CC(\aaa)$, in particular
\begin{equation}\label{eqdecomp}
\CC^2_{\mathrm{GS}}(\aaa) = \CC^{0,2}(\aaa) \oplus \CC^{1,1}(\aaa) \oplus \CC^{2,0}(\aaa).
\end{equation}
Following \cite{DLL:defo-qch}, we usually write a GS $2$-cochain as $(m,f,c)$ corresponding to the decomposition \eqref{eqdecomp}, and we call a $2$-cocycle $(m,f,c)$ \emph{untwined} (we used the terminology \emph{decomposable} in \cite{DLL:defo-qch}) if $(m, 0, 0)$, $(0, f, 0)$ and $(0,0, c)$ are $2$-cocycles as well.
A GS $2$-class $\gamma \in H^2_{\mathrm{GS}}(\aaa)$ is called \emph{intertwined} if there is no untwined representative $(m, f, c)$ with $\gamma = [(m, f, c)]$.
Since $\aaa$ is a presheaf of commutative algebras, for a $2$-cocycle $(m, f, c)$ we automatically have that $(0,0,c)$ is a cocycle so $(m, f, c)$ is untwined if and only if $(m, 0, 0)$ is a cocycle if and only if $(0, f, 0)$ is a cocycle. Under the Hodge decomposition $$\CC^2_{\GS}(\mc A)=\CC^2_{\GS}(\mc A)_2\oplus \CC^2_{\GS}(\mc A)_1\oplus \CC^2_{\GS}(\mc A)_0,$$ any $2$-cocycle $(m,f,c)$ factors as the sum
\begin{equation}\label{decompab}
(m,f,c)=(m-m^{\ab},0,0)+(m^{\ab},f,0)+(0,0,c)
\end{equation}
of $2$-cocycles.
Locally, $m^{\ab}$ is symmetric, sending $(a,b)$ to $m(a,b)/2+m(b,a)/2$, and $m-m^{\ab}$ is anti-symmetric. Hence, $(m, f, c)$ is untwined if and only if $m^{\ab}$ is a $2$-cocycle.

There exist various presheaves $\mc A$ such that every GS $2$-class admits a representative $(m,f,c)$ with $m^{\ab}=0$, which is thus untwined. This happens if $\mc A(V)$ is smooth for all $V\in\mf V$ (see \cite[\S3.3]{DLL:defo-qch}). In this case, let $\mc T$ be the associated tangent presheaf, and thus $m$, $f$, $c$ represent classes in $H^0_{\simp}(\mf V,\wedge^2\mc T)$, $H^1_{\simp}(\mf V,\mc T)$, $H^2_{\simp}(\mf V,\mc A)$, respectively. The Hodge decomposition gives rise to isomorphisms
\[
H^n_{\GS}(\mc A)\cong \bigoplus_{p+q=n}H^p_{\simp}(\mf V, \wedge^q\mc T)
\]
in terms of the simplicial cohomology for all $n$.  A typical example of such a presheaf is obtained from a quasi-compact separated, smooth scheme. Let $(X,\mc O_X)$ be a quasi-compact separated scheme with an affine open covering $\mf V$ which is closed under intersection, and let $\mc A=\mc O_X|_{\mf V}$ be the restriction of $\mc O_X$ to the covering $\mf V$. The cohomology $H^\bullet_{\GS}(\mc A)$ turns out to be isomorphic to the Hochschild cohomology
\[
HH^\bullet(X):=\Ext^\bullet_{X\times X}(\De_*\mc O_X, \De_*\mc O_X)
\]
of the scheme $X$ where $\De\colon X\to X\times X$ is the diagonal map \cite{Lowen-VandenBergh:hoch}. If furthermore, $X$ is smooth, then the Hodge decomposition corresponds to the HKR decomposition and we obtain the familiar formula
\begin{equation}\label{eq:HKR-decomp-introduction}
HH^n(X)\cong \bigoplus_{p+q=n}H^p(X,\wedge^q\mc T_X)
\end{equation}
where $\mc T_X$ is the tangent sheaf of $X$. This formula has been proved in various different contexts and ways \cite{Gerstenhaber-Schack:survey}, \cite{kontsevich2}, \cite{swan}, \cite{yekutieli}, \cite{DLL:defo-qch}.

If $\mc A(V)$ is not smooth for some $V\in\mf V$ (for instance, $X$ has singularities), then whereas we still have $H^2_{\mathrm{GS}}(\aaa)_0 \cong H^2_{\mathrm{simp}}(\mf V, \aaa)$ and $H^2_{\mathrm{GS}}(\aaa)_2 \cong H^2_{\mathrm{simp}}(\mf V, \wedge^2 \ttt)$, the situation for the first Hodge component $H^2_{\mathrm{GS}}(\aaa)_1$ now becomes more interesting. In the decomposition \eqref{decompab} of a $2$-cocycle $(m, f, c)$, the contribution of $m^{\ab}$ is not always zero, and $[(m^{\ab}, f, 0)]$ is intertwined in general. Although cocycles of the form $(m^{\ab},f,0)$ are not as nice as $(0,f,0)$, in some situations we can simplify them in the alternative manner, that is by getting rid of $f$ rather than $m^{\ab}$. Consider the following two subgroups of
$H^2_{\mathrm{GS}}(\aaa)_1$:
\begin{itemize}
\item the subgroup $E_{\mathrm{res}}$ of $2$-classes of the form $[(0,f,0)]$;
\item the subgroup $E_{\mathrm{mult}}$ of $2$-classes of the form $[(m, 0, 0)]$.
\end{itemize}
We are interested in computing $H^2_{\mathrm{GS}}(\aaa)_1$, $E_{\mathrm{mult}}$, $E_{\mathrm{res}}$, as well as understanding the relations between those three groups, possibly depending on the scheme $X$ (with $\mc A=\mc O_X|_{\mf V}$). Note that $E_{\mathrm{res}} \cong H^1_{\simp}(\mf V,\mc T)$.

The decomposition \eqref{eq:HKR-decomp-introduction} has been generalized to the not necessarily smooth
case by Buchweitz and Flenner in \cite{Buchweitz-Flenner:decomp-Atiyah}, using the Atyiah-Chern character. In terms of the relative cotangent complex $\mbb L_{X/k}$, the generalization is given by
\begin{equation}\label{eq:Buch-Flen-decomp-introduction}
HH^n(X)\cong \bigoplus_{p+q=n}\Ext^p_X(\wedge^q_{\vphantom{X}}\mbb L_{X/k},\mc O_X)
\end{equation}
where $\wedge^q$ should be understood as derived exterior product. Their arguments are mostly established in the derived category $\DD(X)$, and an interpretation of cohomology classes in terms of GS-representatives is not immediate.

Since we need GS-representatives in order to use the deformation interpretation from \eqref{defint}, in \S\ref{subsec:quasi-iso-GS} we construct a smaller complex $\mc H^\bullet$ than ${\mbf C}_{\GS}(\mc A)$ and we give an explicit quasi-isomorphism $\mc H^\bullet\to \mbf C_{\GS}(\mc A)$. Our construction of $\hhh^\bullet$ builds on \cite{BACH} and \cite{Michler:hypersurface}, in both of which the Hochschild (co)homology of affine hypersurfaces is computed. Following their methods, in \S \ref{sec:HHaffine} we describe the Hodge components of the affine Hochschild cohomology groups in terms of the cotangent complex. The other key ingredient in our approach to the projective case is the use of a mixed complex associated to a pair of orthogonal sequences in a commutative ring, which is developed in the self-contained section \S\ref{sec:mixedcomp}.

In \S\ref{subsec:cotangent-complex} we present the cotangent complex $\mbb L_{X/k}$ in terms of twisted structure sheaves $\mc O_X(l)$, and we verify that the cohomology of $\mc H^\bullet$ agrees with \eqref{eq:Buch-Flen-decomp-introduction}, and $\mc H^\bullet$ can be considered to be a natural enhancement of \eqref{eq:Buch-Flen-decomp-introduction}.

In \S\ref{subsec:compute-cohomology}, we compute the cohomology groups of $\mc H^\bullet$ in terms of two easier complexes $\mc C^\bullet(\bou;S)$ and $\mc K^\bullet(\bov;R)$ of graded modules. Our main theorem is the following:

\begin{theorem} \label{thmmain}
Let $X\subset\mbb P^n$ be a projective hypersurface of degree $d$. Denote by $P^i$ the $i$-th cohomology group of $\mc C^\bullet(\bou;S)$ and by $Q^i$ the $i$-th cocycle group. Denote by $Z^i$ the $i$-th cocycle group of $\mc K^\bullet(\bov;R)$. Then the cohomology of $\mc H^\bullet$ is given by
\begin{enumerate}
  \item when $d>n+1$,
  \[
    H^i(\mc H^{\bullet})\cong\bigoplus_{r<i}P^{i-2r}_{r+(i-r)(d-1)}\oplus Q^{-i}_{i}\oplus \scr S(Z^{-i+n-1}_{d-i-2});
  \]
  \item when $d=n+1$,
  \[
    H^i(\mc H^{\bullet})\cong
    \begin{cases}
    \displaystyle \bigoplus_{r<i}P^{i-2r}_{r+n(i-r)}\oplus Q^{-i}_{i}, & i\neq n-1, n,\\
    \displaystyle \bigoplus_{r<i}P^{i-2r}_{r+n(i-r)}\oplus Q^{-i}_{i}\oplus k^n, & i=n-1,\\
    \displaystyle \bigoplus_{r\leq i}P^{i-2r}_{r+n(i-r)}, & i= n;
    \end{cases}
  \]
  \item when $d<n+1$,
  \[
    H^i(\mc H^{\bullet})\cong \bigoplus_{r<i}P^{i-2r}_{r+(i-r)(d-1)}\oplus Q^{-i}_{i}.
  \]
\end{enumerate}
In the above formulas, $\scr S$ is a linear map defined in \eqref{eq:scriptS}, and the subscripts of $P^\bullet$, $Q^\bullet$, $Z^\bullet$ stand for the degrees of homogeneous elements in $P^\bullet$, $Q^\bullet$, $Z^\bullet$.
\end{theorem}

As an application of Theorem \ref{thmmain}, we give a cohomological characterization of smoothness for projective hypersurfaces in \S \ref{subsec:char-smooth}. Recall that an affine hypersurface $\mathrm{Spec}(A)$ is smooth if and only if the first Hodge component $H^2_{(1)}(A,A)$ vanishes (Remark \ref{remsmooth}). In deformation theoretic terms, this corresponds to the fact that $A$ has only trivial commutative deformations. For a projective hypersurface $X$ with restricted structure sheaf $\mc A=\mc O_X|_{\mf V}$, the parallel statement is that $X$ is smooth if and only if the first Hodge component $H^2_{\mathrm{GS}}(\aaa)_1$ coincides with its subgroup $E_{\mathrm{res}} \cong H^1(X, \ttt_X)$ which describes locally trivial scheme deformations of $X$. In other words, $X$ is smooth if and only if the classical HKR decomposition \eqref{eq:HKR-decomp-introduction} holds for the second Hochschild cohomology group of $X$ (Theorem \ref{thm:smooth-equiv}). 

In \S\ref{subsec:char-smooth}, we show that for $\mc A=\mc O_X|_{\mf V}$ with $X$ a projective hypersurface of dimension $\geq 2$, we have
\begin{equation}\label{eqplus}
H^2_{\mathrm{GS}}(\aaa)_1 = E_{\mathrm{mult}} + E_{\mathrm{res}},
\end{equation}
whence we can choose a complement $E$ of $E_{\mathrm{res}}$ inside $H^2_{\mathrm{GS}}(\aaa)_1$ such that $E \subseteq E_{\mathrm{mult}}$. Intuitively, we visualize the situation with the aid of the following diagram:
\[
\xymatrix{
\text{Hodge components:} & H^2_{\GS}(\mc A)_2 \ar@{~}[d] & & H^2_{\GS}(\mc A)_1 \ar@{~}[d]\ar@{~}[dl] & H^2_{\GS}(\mc A)_0 \ar@{~}[d] \\
\text{HKR components:} & H^0_{\simp}(\mf V,\wedge^2\mc T) \ar[d] & E \ar[dl] & H^1_{\simp}(\mf V,\mc T) \ar[d] & H^2_{\simp}(\mf V,\mc A) \ar[d] \\
\text{representatives:} & (m,0,0) & & (0,f,0) & (0,0,c)
}
\]

Remarkably, based upon the results from \S\ref{subsec:compute-cohomology}, an intertwined 2-class (that is, a class in $H^2_{\mathrm{GS}}(\aaa)_1 \setminus E_{\mathrm{mult}} + E_{\mathrm{res}}$) can only exist for a non-smooth projective curve in $\mbb P^2$ of degree $\geq 5$, and we give concrete examples of such curves of degree $\geq 6$ in \S\ref{subsec:intertwined}. We also leave the existence of  intertwined 2-classes for degree 5 curves as an open question.

In \S\ref{subsec:quartic-surface}, we study the case when $X$ is a quartic surface in $\mbb P^3$ in some detail. We show that the dimension of $H^2_{\GS}(\mc A)_1$ lies between $20$ and $32$, reaching all possible values except 30 and 31. The minimal value $H^2_{\GS}(\mc A)_1 = 20$ is reached in the smooth case, in which $X$ is a K3 surface and $H^2_{\GS}(\mc A)_1 \cong H^1(X, \ttt_X)$, as well as in some non-smooth examples like the Kummer surfaces. From our general results, we know the dimension of $E_{\mathrm{mult}}$ to be one less than the dimension of $H^2_{\GS}(\mc A)_1$.
In the smooth case we have $E_{\mathrm{mult}} \subseteq E_{\mathrm{res}}$ and for the Fermat quartic, we give a concrete description of the (19 dimensional) $E_{\mathrm{mult}}$ both by representatives of the form $(m, 0, 0)$ (commutatively deforming the affine pieces) and by equivalent representatives of the form $(0, f, 0)$ (classical picture of a smooth scheme deformation arising from glueing trivial affine deformations). For non-smooth schemes, we encounter examples in which $E_{\mathrm{res}}$ is one-dimensional (and hence $H^2_{\mathrm{GS}}(\aaa)_1 = E_{\mathrm{mult}} \oplus E_{\mathrm{res}}$) as well as examples in which $E_{\mathrm{mult}} \cap E_{\mathrm{res}} \neq 0$.

Finally, let us mention that the zero-th Hodge component $H^2_{\GS}(\mc A)_0$ is invariably one dimensional, and we know that the dimension of the second Hodge component $H^2_{\GS}(\mc A)_2$ is at least one. Although our results allow us to compute the dimension of $H^2_{\GS}(\mc A)_2$ in concrete examples, so far we have not determined the precise range of this dimension.\\

\noindent \emph{Acknowledgement}: The authors are grateful to Pieter Belmans for his interesting comments and questions concerning an earlier version of the paper, which led to the discovery of an error in \S \ref{subsec:intertwined} that is corrected in the current version.

\section{Mixed complexes associated to orthogonal sequences}\label{sec:mixedcomp}
This section is self-contained. In order to make preparations for future computations, we construct several complexes which are related to Koszul complexes, as well as quasi-isomorphisms between them.

Let $R$ be a commutative ring, and let $\bou=(u_0\ldots,u_n)$, $\bov=(v_0,\ldots,v_n)$ be two sequences in $R$. We call $(\bou, \bov)$ a \emph{pair of orthogonal sequences of length $n$ (an $n$-POS)} if
\[
\sum_{i=0}^nu_iv_i=0
\]
holds in $R$. Let $(\mc K^\bullet(\bou; R),\pl_{\bou})$ be the Koszul cochain complex determined by $\bou$, namely, $\mc K^\bullet(\bou; R)$ is the DG $R$-algebra $\wedge^\bullet (Re_0\oplus\cdots \oplus Re_n)$ with $|e_i|=-1$ and $\pl_{\bou}(e_i)=u_i$. Similarly, let $(\mc K_\bullet(\bov; R),\pl^{\bov})=\wedge^\bullet (Rf_0\oplus\cdots \oplus Rf_n)$ be the Koszul chain complex determined by $\bov$. Applying $\Hom_R(-,R)$ to $\mc K_\bullet(\bov; R)$, we obtain a cochain complex $\Hom^\bullet_R(\mc K_\bullet(\bov; R),R)$ whose terms are
\[
\Hom^{-p}_R(\mc K_\bullet(\bov; R),R)=\Hom_R(\mc K_p(\bov; R),R)=\bigoplus_{0\leq i_1<\cdots< i_p\leq n}R(f_{i_1}\wedge\cdots\wedge f_{i_p})^*
\]
and whose differentials are
\begin{align*}
(\pl^{\bov})^*\colon\Hom^{-p}_R(\mc K_\bullet(\bov; R),R)&\To \Hom^{-p-1}_R(\mc K_\bullet(\bov; R),R)\\
(f_{i_1}\wedge\cdots\wedge f_{i_p})^*&\longmapsto \sum_{j=0}^nv_j(f_j\wedge f_{i_1}\wedge\cdots\wedge f_{i_p})^*.
\end{align*}

Since for each $p$, the correspondence $e_{i_1}\wedge\cdots\wedge e_{i_p}\longleftrightarrow (f_{i_1}\wedge\cdots\wedge f_{i_p})^*$ establishes an isomorphism between $\mc K^{-p}(\bou; R)$ and $\Hom^{-p}_R(\mc K_\bullet(\bov; R),R)$  in a natural way. The differentials $(\pl^{\bov})^*$ induce another complex structure on $\mc K^{\bullet}(\bou; R)$ given by
\begin{align*}
\pl_{\bov}\colon\mc K^{-p}(\bou; R)&\To \mc K^{-p-1}(\bou; R)\\
e_{i_1}\wedge\cdots\wedge e_{i_p}&\longmapsto \sum_{j=0}^nv_je_j\wedge e_{i_1}\wedge\cdots\wedge e_{i_p}.
\end{align*}

\begin{remark}\label{rmk:mixed-v}
$(\mc K^\bullet(\bou;R),\pl_{\bov})$ is isomorphic to the Koszul complex determined by the sequence $\bov^\star=(v_0, -v_1,\ldots, (-1)^nv_n)$.
\end{remark}

\begin{lemma}
  $\mc K^{\bullet}(\bou,\bov; R)=(\mc K^{\bullet}(\bou; R),\pl_{\bou},\pl_{\bov})$ is a mixed complex.
\end{lemma}

\begin{proof}
  Let us verify the equality $\pl_{\bou}\pl_{\bov}+\pl_{\bov}\pl_{\bou}=0$. On one hand,
  \begin{align*}
    \pl_{\bou}\pl_{\bov}(e_{i_1}\wedge\cdots\wedge e_{i_p})
    &=\sum_{j=0}^n\pl_{\bou}(v_je_j\wedge e_{i_1}\wedge\cdots\wedge e_{i_p})\\
    &=\sum_{j=0}^nv_j\biggl(u_{j}e_{i_1}\wedge\cdots \wedge e_{i_p}+\sum_{k=1}^{p}(-1)^{k}u_{i_k}e_j\wedge e_{i_1}\wedge\cdots \wedge \widehat{e_{i_{k}}}\wedge\cdots\wedge e_{i_p}\biggr)\\
    &=\sum_{j=0}^nu_{j}v_je_{i_1}\wedge\cdots \wedge e_{i_p}+\sum_{j=0}^n\sum_{k=1}^{p}(-1)^{k}u_{i_k}v_j e_j\wedge e_{i_1}\wedge\cdots \wedge \widehat{e_{i_{k}}}\wedge\cdots\wedge e_{i_p}\\
    &=\sum_{j=0}^n\sum_{k=1}^{p}(-1)^{k}u_{i_k}v_je_j\wedge e_{i_1}\wedge\cdots \wedge \widehat{e_{i_{k}}}\wedge\cdots\wedge \cdots\wedge e_{i_p}.
  \end{align*}
  On the other hand,
  \begin{align*}
    \pl_{\bov}\pl_{\bou}(e_{i_1}\wedge\cdots\wedge e_{i_p})&=\pl_{\bov}\biggl(\sum_{k=1}^p(-1)^{k-1}u_{i_k}e_{i_1}\wedge\cdots\wedge \widehat{e_{i_k}}\wedge\cdots\wedge e_{i_p}\biggr)\\
    &=\sum_{k=1}^p(-1)^{k-1}u_{i_k}\sum_{j=0}^nv_je_j\wedge e_{i_1}\wedge\cdots\wedge \widehat{e_{i_k}}\wedge\cdots\wedge e_{i_p}\\
    &=\sum_{j=0}^n\sum_{k=1}^p(-1)^{k-1}u_{i_k}v_je_j\wedge e_{i_1}\wedge\cdots\wedge \widehat{e_{i_k}}\wedge\cdots\wedge e_{i_p}.
  \end{align*}
  Thus $\pl_{\bou}\pl_{\bov}+\pl_{\bov}\pl_{\bou}=0$ is true.
\end{proof}

This mixed complex gives rise to a double complex $\mc K^{\bullet,\bullet}(\bou,\bov; R)$ in the first quadrant as in Figure \ref{fig:double-comp-K}. For $r\in\nan$, define $\tau^{r}\mc K^{\bullet,\bullet}(\bou,\bov;R)$ to be the quotient double complex of $\mc K^{\bullet,\bullet}(\bou,\bov;R)$ consisting of all entries whose coordinates satisfy $0\leq q\leq r$.
\begin{figure}[!htp]
\[
\xymatrix@C=8mm@C=7mm{
\vdots & \vdots & \vdots & \vdots & \iddots \\
\mc K^{-3}(\bou,\bov; R) \ar[r]^-{\pl_{\bou}}\ar[u]_-{\pl_{\bov}} & \mc K^{-2}(\bou,\bov; R) \ar[r]^-{\pl_{\bou}}\ar[u]_-{\pl_{\bov}} & \mc K^{-1}(\bou,\bov; R) \ar[r]^-{\pl_{\bou}}\ar[u]_-{\pl_{\bov}} & \mc K^{0}(\bou,\bov; R) \ar[u]_-{\pl_{\bov}} \\
\mc K^{-2}(\bou,\bov; R) \ar[r]^-{\pl_{\bou}}\ar[u]_-{\pl_{\bov}} & \mc K^{-1}(\bou,\bov; R) \ar[r]^-{\pl_{\bou}}\ar[u]_-{\pl_{\bov}} & \mc K^{0}(\bou,\bov; R) \ar[u]_-{\pl_{\bov}} \\
\mc K^{-1}(\bou,\bov; R) \ar[r]^-{\pl_{\bou}}\ar[u]_-{\pl_{\bov}} & \mc K^{0}(\bou,\bov; R) \ar[u]_-{\pl_{\bov}} \\
\mc K^{0}(\bou,\bov; R) \ar[u]_-{\pl_{\bov}}
}
\]
\caption{Double complex $\mc K^{\bullet,\bullet}(\bou,\bov; R)$}
\label{fig:double-comp-K}
\end{figure}

Suppose that $v_t$ is invertible for some $t\in \{0,1,\ldots,n\}$. Let $\bow=(u_0,\ldots,\widehat{u_t},\ldots,u_n)$, and $(\mc K^\bullet(\bow; R),\pl_{\bow})$ be the corresponding Koszul complex. Define $\iota\colon \mc K^\bullet(\bow; R)\to \mc K^{\bullet}(\bou; R)$ to be the canonical embedding morphism, and define $\pi\colon \mc K^{\bullet}(\bou; R)\to \mc K^\bullet(\bow; R)$ by
\[
\pi(e_{i_1}\wedge\cdots\wedge e_{i_p})=
\begin{cases}
  e_{i_1}\wedge\cdots\wedge e_{i_p}, & \text{if none of $i_j$ is $t$,}\\
  {}\\
  \displaystyle -\sum_{k\neq t}v_kv^{-1}_te_{i_1}\wedge\cdots\wedge e_{i_{j-1}}\wedge e_k\wedge e_{i_{j+1}}\wedge\cdots\wedge e_{i_p}, & \text{if $t=i_j$ for some $j$.}
\end{cases}
\]
for each $p$.

\begin{lemma}\label{lem:koszul-1}
  $\pi\colon \mc K^{\bullet}(\bou; R)\to \mc K^\bullet(\bow; R)$ is a morphism of complexes.
\end{lemma}

\begin{proof}
  It suffices to prove $\pl_{\bow}\pi(e_{i_1}\wedge\cdots\wedge e_{i_p})=\pi\pl_{\bou}(e_{i_1}\wedge\cdots\wedge e_{i_p})$ when $t=i_j$ for some $j$.

  We have
  \begin{align*}
    \pl_{\bow}\pi(e_{i_1}\wedge\cdots\wedge e_{i_p})    &=-\pl_{\bow}\biggl(\sum_{k\neq t}v_kv_t^{-1}e_{i_1}\wedge\cdots\wedge e_{i_{j-1}}\wedge e_k\wedge e_{i_{j+1}}\wedge\cdots\wedge e_{i_p}\biggr)\\
    &=-\sum_{k\neq t}v_kv_t^{-1}\pl_{\bow}(e_{i_1}\wedge\cdots\wedge e_{i_{j-1}}\wedge e_k\wedge e_{i_{j+1}}\wedge\cdots\wedge e_{i_p})\\
    &=-\sum_{k\neq t}v_kv_t^{-1}\biggl(\sum_{l\neq j}(-1)^{l-1}u_{i_l} e_{i_1}\wedge\cdots\wedge e_k\wedge \cdots\wedge \widehat{e_{i_{l}}}\wedge \cdots\wedge e_{i_p}\\
    &\varphantom{=}{}+(-1)^{j-1}u_k e_{i_1}\wedge\cdots\wedge \widehat{e_{k}}\wedge \cdots\wedge e_{i_p}\biggr)\\
    &=\sum_{k\neq t}\sum_{l\neq j}(-1)^{l}u_{i_l}v_kv_t^{-1} e_{i_1}\wedge\cdots\wedge e_k\wedge \cdots\wedge \widehat{e_{i_{l}}}\wedge \cdots\wedge e_{i_p}\\
    &\varphantom{=}{}+\sum_{k\neq t}(-1)^{j}u_kv_kv_t^{-1} e_{i_1}\wedge\cdots\wedge \widehat{e_{k}}\wedge \cdots\wedge e_{i_p}\\
    &=\sum_{k\neq t}\sum_{l\neq j}(-1)^{l}u_{i_l}v_kv_t^{-1} e_{i_1}\wedge\cdots\wedge e_k\wedge \cdots\wedge \widehat{e_{i_{l}}}\wedge \cdots\wedge e_{i_p}\\
    &\varphantom{=}{}-(-1)^{j}u_t e_{i_1}\wedge\cdots\wedge \widehat{e_{t}}\wedge \cdots\wedge e_{i_p}.
  \end{align*}

  On the other hand, we have
  \begin{align*}
    \pi\pl_{\bou}(e_{i_1}\wedge\cdots\wedge e_{i_p})&=\pi\biggl(\sum_{l=1}^p(-1)^{l-1}u_{i_l}e_{i_1}\wedge\cdots\wedge \widehat{e_{i_l}}\wedge\cdots\wedge e_{i_p}\biggr)\\
    &=\sum_{l\neq j}(-1)^{l}u_{i_l}\sum_{k\neq t}v_kv_t^{-1}e_{i_1}\wedge\cdots\wedge e_k\wedge \cdots\wedge \widehat{e_{i_l}}\wedge\cdots\wedge e_{i_p}\\
    &\varphantom{=}{}+(-1)^{j-1}u_{t}e_{i_1}\wedge\cdots\wedge \widehat{e_{i_j}}\wedge\cdots\wedge e_{i_p}.
  \end{align*}
  This finishes the proof that $\pl_{\bow}\pi=\pi\pl_{\bou}$.
\end{proof}

\begin{lemma}\label{lem:koszul-2}
  For all $p$, the sequence
  \[
  0\To \mc K^{-p+1}(\bow;R)\xrightarrow[\quad]{\pl_{\bov}\iota}\mc K^{-p}(\bou;R)\xrightarrow[\quad]{\pi}\mc K^{-p}(\bow;R)\To 0
  \]
  is split exact.
\end{lemma}

\begin{proof}
  First of all, let us check that this is indeed a complex, namely, $\pi\pl_{\bov}\iota=0$. Consider the base element $e_{i_1}\wedge\cdots\wedge e_{i_{p-1}}$ in $\mc K^{-p+1}(\bow;R)$. Suppose $i_{j-1}<t<i_j$. We have
  \begin{align*}
    &\varphantom{=}\pi\pl_{\bov}\iota(e_{i_1}\wedge\cdots\wedge e_{i_{p-1}})\\
    &=\pi\biggl(\sum_{k\neq t}v_ke_k\wedge e_{i_1}\wedge\cdots\wedge e_{i_{p-1}}+v_t e_t\wedge e_{i_1}\wedge\cdots\wedge e_{i_{p-1}}\biggr)\\
    &=\sum_{k\neq t}v_ke_k\wedge e_{i_1}\wedge\cdots\wedge e_{i_{p-1}}-\sum_{k\neq t}v_tv_kv_t^{-1}e_k \wedge e_{i_1}\wedge\cdots\wedge e_{i_{p-1}}\\
    &=0.
  \end{align*}

  Next, we consider the map $\id-\iota\pi$. By the definition of $\pi$, if none of $i_j$ is $t$, then
  \[
  (\id-\iota\pi)(e_{i_1}\wedge\cdots\wedge e_{i_{p}})=0;
  \]
  if $t=i_j$, then
  \begin{align*}
    (\id-\iota\pi)(e_{i_1}\wedge\cdots\wedge e_{i_{p}})&=e_{i_1}\wedge\cdots\wedge e_{i_{p}}+\sum_{k\neq t}v_kv_t^{-1}e_{i_1}\wedge\cdots\wedge e_{i_{j-1}}\wedge e_k\wedge e_{i_{j+1}}\wedge\cdots\wedge e_{i_p}\\
    &=\sum_{k=0}^nv_kv_t^{-1}e_{i_1}\wedge\cdots\wedge e_{i_{j-1}}\wedge e_k\wedge e_{i_{j+1}}\wedge\cdots\wedge e_{i_p}\\
    &=\sum_{k=0}^n(-1)^{j-1}v_kv_t^{-1} e_k\wedge e_{i_1}\wedge\cdots\wedge \widehat{e_{i_{j}}} \wedge\cdots\wedge e_{i_p}\\
    &=\pl_{\bov}((-1)^{j-1}v_t^{-1}e_{i_1}\wedge\cdots\wedge \widehat{e_{i_{j}}} \wedge\cdots\wedge e_{i_p}).
  \end{align*}
  It follows that there exists a map $\zeta\colon \mc K^{-p}(\bou;R)\to\mc K^{-p+1}(\bow;R)$ given by
  \[
  \zeta(e_{i_1}\wedge\cdots\wedge e_{i_{p}})=
  \begin{cases}
    0, & \text{if none of $i_j$ is $t$,}\\
  (-1)^{j-1}v_t^{-1}e_{i_1}\wedge\cdots\wedge \widehat{e_{i_{j}}} \wedge\cdots\wedge e_{i_p}, & \text{if $t=i_j$ for some $j$,}
  \end{cases}
  \]
  which satisfies $\pl_{\bov}\iota\zeta+\iota\pi=\id$. Moreover, $\pi\iota=\id$, $\zeta\pl_{\bov}\iota=\id$. These facts indicate that the complex is split exact.
\end{proof}

Let $\tau^{\geq r}$ be the stupid truncation functor. Since the top row of $\tau^{r}\mc K^{\bullet,\bullet}(\bou,\bov;R)$ is the same as $\tau^{\geq 0}(\mc K^{\bullet}(\bou;R)[-r])$, we define the morphism $\iota_{t,(r)}$ associated to $t$ as the composition of
\[
\tau^{\geq r}(\mc K^{\bullet}(\bow;R)[-2r])\xrightarrow[\quad]{\iota}\tau^{\geq r}(\mc K^{\bullet}(\bou;R)[-2r])\hookrightarrow \Tot(\tau^{r}\mc K^{\bullet,\bullet}(\bou,\bov;R)).
\]
Sometimes we suppress the subscript $t$ in $\iota_{t,(r)}$ if no confusion arises.

\begin{proposition}\label{prop:koszul-quasi}
  For any $r\geq 0$, $\iota_{(r)}\colon \tau^{\geq r}(\mc K^{\bullet}(\bow;R)[-2r])\to \Tot(\tau^{r}\mc K^{\bullet,\bullet}(\bou,\bov;R))$ is a quasi-isomorphism with a quasi-inverse $\pi_{(r)}$ induced by $\pi$.
\end{proposition}

\begin{proof}
  By Lemmas \ref{lem:koszul-1}, \ref{lem:koszul-2}, the sequence
  \[
  0\To \mc K^{\bullet}(\bow;R)[1-2r]\xrightarrow[\quad]{(-1)^{\bullet}\pl_{\bov}\iota}\mc K^{\bullet}(\bou;R)[-2r]\xrightarrow[\quad]{\pi}\mc K^{\bullet}(\bow;R)[-2r]\To 0
  \]
  of cochain complexes is exact. After shifting degrees, we have another exact sequence
  \[
  0\To \mc K^{\bullet}(\bow;R)[2-2r]\xrightarrow[\quad]{(-1)^{\bullet-1}\pl_{\bov}\iota}\mc K^{\bullet}(\bou;R)[1-2r]\xrightarrow[\quad]{\pi}\mc K^{\bullet}(\bow;R)[1-2r]\To 0.
  \]
  Since $\pl_{\bov}\iota\zeta+\iota\pi=\id$ (see the proof of Lemma \ref{lem:koszul-2}), we have $(-1)^{\bullet}\pl_{\bov}\iota\pi=(-1)^{\bullet}\pl_{\bov}(\id-\pl_{\bov}\iota\zeta)=(-1)^{\bullet}\pl_{\bov}$. So the above two exact sequences are combined into a new one
  \[
  0\To \mc K^{\bullet}(\bow;R)[2-2r]\xrightarrow[\quad]{(-1)^{\bullet-1}\pl_{\bov}\iota}\mc K^{\bullet}(\bou;R)[1-2r]\xrightarrow[\quad]{(-1)^{\bullet}\pl_{\bov}}\mc K^{\bullet}(\bou;R)\xrightarrow[\quad]{\pi}\mc K^{\bullet}(\bow;R)[-2r]\To 0.
  \]
  Continuing the procedure, we obtain a long exact sequence
  \[
  \cdots\To \mc K^{\bullet}(\bou;R)[2-2r]\xrightarrow[\quad]{(-1)^{\bullet-1}\pl_{\bov}}\mc K^{\bullet}(\bou;R)[1-2r]\xrightarrow[\quad]{(-1)^{\bullet}\pl_{\bov}}\mc K^{\bullet}(\bou;R)[-2r]\overset{\pi}{\twoheadrightarrow}\mc K^{\bullet}(\bow;R)[-2r].
  \]

  Let the functor $\tau^{\geq r}$ act on the long sequence, and then by using the sign trick, we make all the terms except the last one (i.e.\ $\tau^{\geq r}(\mc K^{\bullet}(\bow;R)[-2r])$) into a double complex. It is obvious that the resulting double complex is nothing but $\tau^{r}\mc K^{\bullet,\bullet}(\bou,\bov;R)$. Therefore, $\pi$ induces a quasi-isomorphism
  \[
  \pi_{(r)}\colon \Tot(\tau^{r}\mc K^{\bullet,\bullet}(\bou,\bov;R))\To \tau^{\geq r}(\mc K^{\bullet}(\bow;R)[-2r])
  \]
  which is quasi-inverse to $\iota_{(r)}$.
\end{proof}

\begin{definition}
An $n$-POS $(\bou,\bov)$ is said to be \emph{proportional} to another one $(\bou',\bov')$ if there exist invertible $\lam$, $\mu\in R$ such that $(\bou',\bov')=(\lam\bou,\mu\bov)$.
\end{definition}

Notice that the $(p,q)$-entry of $\tau^{r}\mc K^{\bullet,\bullet}(\bou,\bov;R)$ (resp.\ $\tau^{r}\mc K^{\bullet,\bullet}(\bou',\bov';R)$) is $\mc K^{p-q}(\bou,\bov;R)$ (resp.\ $\mc{K}^{p-q}(\bou',\bov';R)$), and that $\mc K^{p-q}(\bou,\bov;R)$ and $\mc K^{p-q}(\bou',\bov';R)$ share the same rank as free $R$-modules. There are isomorphisms
\[
\lam^{p}\mu^{q}\colon \mc K^{p-q}(\bou,\bov;R)\To \mc K^{p-q}(\bou',\bov';R)
\]
given by the multiplication by $\lam^{p}\mu^{q}$ for all $p$, $q$, and they constitute an isomorphism
\begin{equation}\label{eq:iso-double}
\xi_{(r)}\colon \tau^{r}\mc K^{\bullet,\bullet}(\bou,\bov;R)\To \tau^{r}\mc K^{\bullet,\bullet}(\bou',\bov';R)
\end{equation}
of double complexes. The induced isomorphism between their total complexes is denoted by $\xi_{(r)}^{\Tot}$.

\section{Hochschild cohomology of affine hypersurfaces}\label{sec:HHaffine}

Let $A=k[y_1,\ldots,y_n]/(G)$ be the quotient of the polynomial algebra $k[y_1,\ldots,y_n]$ by a unique relation $G$. There are several papers concerning the Hochschild and cyclic (co)homology of $A$, the treatment of the topic dating back to Wolffhardt's work on Hochschild homology of (analytic) complete intersections \cite{wolffhardt}. We base our exposition on the more recent papers \cite{BACH}, \cite{Michler:hypersurface}. In \cite{Michler:hypersurface}, Michler describes the Hochschild homology groups of $A$  as well as their Hodge decompositions when $G$ is reduced, based on the cotangent complex of $A$. The Hochschild cohomology groups are not treated in \cite{Michler:hypersurface}. In \cite{BACH}, the authors from BACH construct a nice finitely generated free resolution $\mc R_\bullet(A)$ of $A$ over $A^e$ under an additional condition on $G$. For the normalized bar resolution $\bar{C}^\mathrm{bar}_\bullet(A)$, the authors give comparison maps
\begin{equation}\label{alpha}
\al\colon \bar{C}^\mathrm{bar}_\bullet(A)\to\mc R_\bullet(A)
\end{equation}
and $\al'\colon \mc R_\bullet(A)\to \bar{C}^\mathrm{bar}_\bullet(A)$ satisfying $\al\al'=\id$. By virtue of the smaller resolution $\mc R_\bullet(A)$, the authors compute the Hochschild homology and cohomology of $A$.

From now on, we assume that $G=G(y_1,\ldots,y_n)$ has leading term $y_1^d$  with respect to the lexicographic ordering $y_1>\cdots> y_n$. Under this assumption, we are able to use the resolution $\mc R_\bullet(A)$ from \cite{BACH} and obtain the Hochschild cohomology groups as $H^p(A,A)=H^p(\mc L^\bullet(A))$ where $\mc L^\bullet(A)=\Hom_{A^e}(\mc R_{\bullet}(A),A)$. In this section, we first make the complex $\mc L^\bullet(A)$ explicit according to \cite{BACH}. Next we restate $\mc L^\bullet(A)$ in terms of the cotangent complex, inspired by \cite{Michler:hypersurface}. Finally, Hochschild cohomology of localizations of $A$ is considered.

By the construction of \cite{BACH}, $\mc L^\bullet_|(A)=\wedge^\bullet (A\mf e_1\oplus\cdots\oplus A\mf e_n)$ and $\mc L^\bullet(A)$ is the algebra of divided powers over $\mc L^\bullet_|(A)$ in one variable $\mf s$. Put $|\mf e_i|=1$ and $|\mf s^{(j)}|=2j$, $\mc L^\bullet(A)$ is made into a DG $A$-algebra whose differential is given by $\mf e_i\mapsto (\pl G/\pl y_i)\mf s^{(1)}$ and $\mf s^{(1)}\mapsto 0$. By writing $\mf e_{i_1\ldots i_l}$ instead of the product $\mf e_{i_1}\wedge\cdots\wedge\mf e_{i_l}$, we have
\[
\mc L^p(A)=\bigoplus_{\substack{0\leq j\leq p/2\\1\leq i_1<\cdots<i_{p-2j}\leq n}}A\mf e_{i_1\ldots i_{p-2j}}\mf s^{(j)},
\]
and the differential $\mc L^p(A)\to\mc L^{p+1}(A)$ is given by
\[
\mf e_{i_1\ldots i_{p-2j}}\mf s^{(j)}\longmapsto\sum_{l=1}^{p-2j}(-1)^{l-1}\frac{\pl G}{\pl y_{i_l}}\mf e_{i_1\ldots \widehat{i_l}\ldots i_{p-2j}}\mf s^{(j+1)}.
\]
It immediately follows that the $A$-module complex $\mc L^\bullet(A)$ admits a decomposition $\mc L^\bullet(A) = \oplus_{r\in\nan}\mc L^\bullet(A)_r$ with
\begin{equation}\label{eq:L(A)r}
\mc L^\bullet(A)_r=\tau^{\geq r}(\mc K^{\bullet}((\pl G/\pl y_i)_{1\leq i\leq n};A)[-2r]).
\end{equation}

As stated in \cite{Michler:hypersurface} (also see \cite[Ch.\ III, Prop.\ 3.3.6]{Illusie:cot-comp-1}), the cotangent complex $\mbb{L}_{A/k}$ of $A$, which is unique up to homotopy equivalence, is given by
\[
0\To Adz\xrightarrow[\quad]{\de} \bigoplus_{i=1}^nAd y_i\To 0
\]
where the two nonzero terms sit in degrees $-1$ and $0$ respectively, $dz$ and $dy_i$ are base elements and
\[
\de(dz)=\sum_{i=1}^{n}\frac{\pl G}{\pl y_i} dy_i.
\]
Note that by \cite{gerstenhaberschackhodge}, $H^p(A, A)$ has the Hodge decomposition $\oplus_{r\in\nan}H^p_{(r)}(A,A)$, and the component
\[
H^p_{(r)}(A,A)\cong \Ext^{p-r}_A(\wedge^r\mbb L_{A/k}, A).
\]
By \cite[Ch.\ VIII, Cor.\ 2.1.2.2]{Illusie:cot-comp-2}, $\wedge^r\mbb L_{A/k}$ is isomorphic to a complex determined by $\de$ in the derived category $\DD^b(A)$, more explicitly,
\[
\wedge^r\mbb L_{A/k}\cong \bigoplus_{i+j=r}\wedge^i(Ady_1\oplus\cdots\oplus Ady_n)\ot_A \varGamma^j(Adz)
\]
where $\varGamma^j(-)$ is the degree $j$ component of the divided power functor over $A$.\footnote{Upright $\Gamma(X,-)$ will denote the global section functor on a scheme $X$ in \S\ref{sec:HHproj}.} It follows that $\Ext^{p-r}_A(\wedge^r\mbb L_{A/k}, A)$ is the $(p-r)$-th cohomology group of
\begin{equation}\label{eq:dual-cotangent-compl}
\Hom_A(\wedge^r\mbb L_{A/k}, A)\cong \bigoplus_{i+j=r}\wedge^i(A(dy_1)^*\oplus\cdots\oplus A(dy_n)^*)\ot_A \varGamma^j(A(dz)^*)
\end{equation}
Notice that the $j$-th term of the right-hand side of \eqref{eq:dual-cotangent-compl} is free of rank $\binom{n}{r-j}$ which is the same as $\tau^{\geq 0}(\mc K^{\bullet}((\pl G/\pl y_i)_{1\leq i\leq n};A)[-r])$ for all $0\leq j\leq r$. By taking into account the differentials, we can construct an isomorphism $\Hom_A(\wedge^r\mbb L_{A/k}, A)\cong \tau^{\geq 0}(\mc K^{\bullet}((\pl G/\pl y_i)_{1\leq i\leq n};A)[-r])$. Equivalently, $\Hom_A(\wedge^r\mbb L_{A/k}, A)[-r]\cong \mc L^\bullet(A)_r$ by \eqref{eq:L(A)r} and consequently the isomorphism
\[
\mc L^\bullet(A)\cong \bigoplus_{r\in\nan}\Hom_A(\wedge^r\mbb L_{A/k}, A)[-r]
\]
holds true in $\DD^b(A)$. Therefore, $H^p_{(r)}(A,A)\cong  H^{p}(\Hom_A(\wedge^r\mbb L_{A/k}, A)[-r])\cong H^p(\mc L^\bullet(A)_r)$, and the decomposition of $H^p(\mc L^\bullet(A))$ deduced from \cite{BACH} actually corresponds to the Hodge decomposition of $H^p(A,A)$.

Observe that the comparison map $\al$ from \eqref{alpha} gives rise to a quasi-isomorphism $\be\colon \mc L^\bullet(A)\to \bar{C}^\bullet(A,A)$ landing in the normalized Hochschild complex of $A$. The morphism $\be$, whose explicit expression will be given later on, induces the isomorphism $H^p_{(r)}(A,A)\cong  H^{p}(\mc L^\bullet(A)_r)$. For our purpose, we first introduce some cochains. Note that the algebra $A$ has a basis
\[
\mc B_A=\{y_1^{p_1}y_2^{p_2}\cdots y_n^{p_n}\mid 0\leq p_1\leq d-1,\, p_2,\ldots,p_n\in\nan\}.
\]
We define for $1\leq l\leq n$ a normalized 1-cochain $\ppl/\pl y_l$ by
\begin{equation}\label{eqpart}
\mc B_A\ni y_1^{p_1}y_2^{p_2}\cdots y_n^{p_n}=f\longmapsto\frac{\ppl f}{\pl y_l}=p_ly_1^{p_1}\cdots y_{l-1}^{p_{l-1}}y_l^{p_l-1}y_{l+1}^{p_{l+1}}\cdots y_n^{p_n}
\end{equation}
and a normalized 2-cochain $\pmul$ by
\begin{equation}\label{eqmu}
\pmul(f,g)=
\begin{cases}
0, & p_1+q_1<d,\\
y_1^{p_1+q_1-d}y_2^{p_2+q_2}\cdots y_n^{p_n+q_n}, & p_1+q_1\geq d.
\end{cases}
\end{equation}
for an additional $g=y_1^{q_1}y_2^{q_2}\cdots y_n^{q_n}\in\mc B_A$.
One can easily check that $\pmul$ is a $2$-cocycle.

Now we give the expression of $\be =\sum_r\be_{(r)} \colon \mc L^\bullet(A)\to \bar{C}^\bullet(A,A)$:
\begin{equation}\label{eq:qiso-affine}
\be_{(p-j)}(\mf e_{i_1\ldots i_{p-2j}}\mf s^{(j)})=(-1)^{\binom{p-2j}{2}}\frac{\ppl}{\pl y_{i_1}}\cup\cdots\cup \frac{\ppl}{\pl y_{i_{p-2j}}}\cup\pmul^{\cup j}.
\end{equation}
The notation $\cup$, not to be confused with the well-known cup product, is defined as
   \[
   P_1\cup P_2\cup\cdots \cup P_m=\frac{1}{m!}\sum_{\si\in S_m}(-1)^{c}\mu\circ\bigl(P_{\si^{-1}(1)}\ot P_{\si^{-1}(2)}\ot\cdots\ot P_{\si^{-1}(m)}\bigr)
   \]
   where $P_i\in\bar{C}^\bullet(A,A)$, $\mu$ is the multiplication map (or rather its unique extension by associativity to an $m$-ary multiplication map) and
   \[
   c=\#\{(i,j)\mid i<j,\, \si^{-1}(i)>\si^{-1}(j),\, P_i, P_j \text{ have odd degrees}\}.
   \]
Thus, the operation $\cup$ becomes supercommutative. For example,
  \begin{gather*}
  \frac{\ppl}{\pl y_i}\cup\pmul=\frac{1}{2}\mu\circ\biggl(\frac{\ppl}{\pl y_i}\ot\pmul+\pmul\ot\frac{\ppl}{\pl y_i}\biggr)=\pmul\cup\frac{\ppl}{\pl y_i},\\
  \frac{\ppl}{\pl y_i}\cup\frac{\ppl}{\pl y_j}=-\frac{\ppl}{\pl y_j}\cup\frac{\ppl}{\pl y_i}.
  \end{gather*}

\begin{remark}
Since $\be_{(r)}(\mc L^\bullet(A)_r)\subseteq \bar{C}^\bullet(A,A)_r$, we also call $\mc L^\bullet(A)=\oplus_{r\in\nan}\mc L^\bullet(A)_r$ the \emph{Hodge decomposition}.
\end{remark}

\begin{remark}\label{remsmooth}
Recall that the vanishing of the groups $H^2_{(1)}(A,M)$ for all $A$-modules $M$ characterizes smoothness of $A$. By \cite[Thm.\ 5.3]{Knudson:smooth-affine}, the condition is in turn equivalent to the vanishing of the single group $H^2_{(1)}(A,A)$, i.e.\ $H^2(\mc L^\bullet(A)_1)=0$. It follows that $A$ is smooth if and only if the ideal $(\pl G/\pl y_1,\ldots,\pl G/\pl y_n)$ is equal to $A$ itself.
\end{remark}

Let $\bar{A}$ be the localization of $A$ at a multiplicatively closed set generated by $y_{t_1},\ldots,y_{t_h}$ where $2\leq t_1<\cdots< t_h\leq n$. Let $\si\colon \bar{A}\to B$ be a morphism of commutative algebras such that $B$ is a flat $\bar{A}$-module via $\si$. Then $\bar{A}$ has a basis
\[
\mc B_{\bar{A}}=\{y_1^{p_1}y_2^{p_2}\cdots y_n^{p_n}\mid 0\leq p_1\leq d-1, p_{t_1},\ldots,p_{t_h}\in\inn, \text{ other }p_i\in\nan\}.
\]
As above, cochains $\ppl/\pl y_l\in\bar{C}^1(\bar{A},\bar{A})$ and $\pmul\in\bar{C}^2(\bar{A},\bar{A})$ can be defined similarly. After composing them with $\si$, we obtain cochains in $\bar{C}^1(\bar{A},B)$, $\bar{C}^2(\bar{A},B)$. Furthermore, one can easily check that there is a quasi-isomorphism $\be\colon B\ot_A \mc L^\bullet(A)\to \bar{C}^\bullet(\bar{A},B)$ whose expression is similar to the one shown in \eqref{eq:qiso-affine}.

\section{Hochschild cohomology of projective hypersurfaces}\label{sec:HHproj}

For any morphism $X\to Y$ of schemes or analytic spaces, Buchweitz and Flenner introduce the Hochschild complex $\mbb H_{X/Y}$ of $X$ over $Y$ \cite{Buchweitz-Flenner:global-Hochschild}, and they deduce an isomorphism $\mbb H_{X/Y}\cong \mbb S(\mbb L_{X/Y}[1])$ in the derived category $\DD(X)$ where $\mbb L_{X/Y}$ denotes the cotangent complex of $X$ over $Y$ and $\mbb S(\mbb L_{X/Y}[1])$ is the derived symmetric algebra \cite{Buchweitz-Flenner:decomp-Atiyah}. As a consequence, there is a decomposition of Hochschild cohomology in terms of the derived exterior powers of the cotangent complex
\begin{equation}\label{eq:generalized-HKR}
HH^i(X/Y)\cong\bigoplus_{p+q=i}\Ext^p_X(\wedge^q_{\vphantom{X}}\mbb L_{X/Y}, \mc O_X)
\end{equation}
which generalizes the HKR decomposition in the smooth case. Around the same time, Schuhmacher also deduced the decomposition \eqref{eq:generalized-HKR} using a different method \cite{Schuhmacher:decomp-Noetherian-scheme}.


In general, it may be hard to compute the right hand side of \eqref{eq:generalized-HKR}, but in some special situations, $\mbb{L}_{X/Y}$ has a very nice expression. For example, in \cite[Expose VIII]{SGA6} Berthelot defines $\mbb L_{X/Y}$ as a complex concentrated in two degrees when $X\to Y$ factors as a closed immersion $X\to X'$ followed by a smooth morphism $X'\to Y$. In particular, Berthelot's definition can be applied to the case when $Y=\mathrm{Spec}\,k$ and $X$ is a projective hypersurface over $k$. Although $\mbb{L}_{X/k}$ admits a very simple expression in this case, we do not use it for our computation.  As a sequel to \cite{DLL:defo-qch}, \cite{Lowen-VandenBergh:hoch}, we compute $HH^i(X)$ starting from the Gerstenhaber-Schack complex, since a deformation interpretation of Gerstenhaber-Schack $2$-cocycles is at hand \cite{DLL:defo-qch}. In \S \ref{subsec:quasi-iso-GS} we construct a series of complexes of $\mc O_X$-modules, as well as morphisms from their associated \v{C}ech complexes to the respective components of the normalized reduced Gerstenhaber-Schack complex $\bar{\mbf{C}}'_{\GS}(\mc O_X|_{\mf V})$ (for a chosen covering $\mf V$). Using the technique from \S\ref{sec:mixedcomp}, we prove that these maps are quasi-isomorphisms.

Due to the theoretical significance of the cotangent complex, we give an expression of $\mbb L_{X/k}$ in terms of twisted structure sheaves $\mc O_X(l)$ in \S\ref{subsec:cotangent-complex} when $X$ is a projective hypersurface. This allows us to explain directly how our results agree with Buchweitz and Flenner's.

In \S \ref{subsec:compute-cohomology}, we prove our main theorem Theorem \ref{thmmain}, providing a computation of the Hochschild cohomology groups of a projective hypersurface of degree $d$ in ${\mbb P}^n$ in terms of easier complexes. The result makes a basic distinction between the case $d > n+1$, the harder case $d = n+1$ and the easier case $d \leq n$.
 
Based upon our computations in \S \ref{subsec:compute-cohomology}, we prove in \S \ref{subsec:char-smooth} that a projective hypersurface is smooth if and only if the HKR decomposition of the second Hochschild cohomology group \eqref{eq:HKR-decomp-introduction} holds (Theorem \ref{thm:smooth-equiv}). This can be seen as an analogue of the characterization of smoothness of affine hypersurfaces (Remark \ref{remsmooth}).

Recall that by definition of the GS complex, we have $$\CC^2_{\mathrm{GS}}(\aaa) = \CC^{0,2}(\aaa) \oplus \CC^{1,1}(\aaa) \oplus \CC^{2,0}(\aaa).$$
We call a $2$-cocycle $(m,f,c) \in \CC^2_{\mathrm{GS}}(\aaa)$ \emph{untwined} ({decomposable} in \cite{DLL:defo-qch}) if $(m, 0, 0)$, $(0, f, 0)$ and $(0,0, c)$ are all $2$-cocycles.
A GS $2$-class is called \emph{intertwined} if it has no untwined representative $(m, f, c)$.
In \S \ref{subsec:intertwined}, based upon the results from \S \ref{subsec:compute-cohomology} we show that for a projective hypersurface as above if either $n \neq 2$ or $n = 2$ and $d \leq 4$, no intertwined $2$-class exists. We give a family of concrete examples of intertwined $2$-class for $n = 2$ and $d \geq 5$.

Finally, in \S \ref{subsec:quartic-surface} we pay special attention to the case of quartic surfaces. We show that the dimension of $H^2_{\GS}(\mc A)_1$ lies between $20$ and $32$, reaching all possible values except 30 and 31. The minimal value $H^2_{\GS}(\mc A)_1 = 20$ is reached in the smooth K3 case. We also present an analysis of how $H^2_{\GS}(\mc A)_1$ is built up from $2$-classes of type $[(m', 0, 0)]$ and $2$-classes of type $[(0, f', 0)]$, giving explicit computations in concrete examples.

\subsection{Construction of quasi-isomorphisms}\label{subsec:quasi-iso-GS}
Let $R=k[x_0,\ldots,x_n]$ and $F=F(x_0,\ldots,x_n)$ be a homogeneous polynomial of degree $d\geq 2$. Let $S=R/(F)$ and $X=\Proj S\subseteq \mbb{P}^n$. Suppose that $F$ has a summand $x_0^d$ when $F$ is uniquely expressed as a sum of nonzero monomials. In this way, $X$ can be covered by
\[
\mf U=\{U_i=X\cap\{x_i\neq 0\}\mid 1\leq i\leq n\}.
\]
Let $\mf V=\{V_{i_1\ldots i_s}=U_{i_1}\cap\cdots\cap U_{i_s}\mid 1\leq i_1<\cdots< i_s\leq n\}$ be the associated covering closed under intersections. For any a $p$-simplex $\si\in\mc N_p(\mf V)$, denote its domain and codomain by ${}_\diamond\si$ and $\si_\diamond$ respectively. Let $\mbf C^{\bullet,\bullet}(\mc A)$ be the Gerstenhaber-Schack double complex where $\mc A=\mc O_X|_{\mf V}$, namely,
\[
\mbf C^{p,q}(\mc A)=\prod_{\si\in\mc N_p(\mf V)}\Hom_k(\mc A(\si_\diamond)^{\ot q}, \mc A({}_\diamond\si))
\]
endowed with the (vertical) product Hochschild differential $d_{\Hoch}$ and the (horizontal) simplicial differential $d_{\simp}$. Recall that a cochain $f = (f_{\si})\in \mbf C^{p,q}(\mc A)$ is called normalized if for any $p$-simplex $\si$, $f_{\si}$ is normalized, and it is called reduced if $f_{\si}=0$ whenever $\si$ is degenerate. Let $\bar{\mbf C}^{\prime \bullet,\bullet}(\mc A)$ be the normalized reduced sub-double complex of $\mbf C^{\bullet,\bullet}(\mc A)$ and $\bar{\mbf C}_{\GS}^{\prime \bullet}(\mc A)$ be the associated total complex.

Observe that for $1\leq i\leq n$,  $A_i=\mc A(U_i)=k[y_0,\ldots,\widehat{y_i},\ldots, y_n]/(G_i)$ where
\[
G_i=F(y_0,\ldots,y_{i-1}, 1, y_{i+1},\ldots,y_n)=y_0^d+\cdots
\]
is monic. Here we assign an ordering $y_0>\cdots>y_{i-1}> y_{i+1}>\cdots> y_n$. So we have complexes $\mc L^\bullet(A_i)$ as given in \S \ref{sec:HHaffine}. Denote by $\bow_i$ the sequence
\[
\biggl(\frac{\pl G_i}{\pl y_0},\ldots,\frac{\pl G_i}{\pl y_{i-1}}, \frac{\pl G_i}{\pl y_{i+1}},\ldots, \frac{\pl G_i}{\pl y_n}\biggr).
\]
Then $\mc L^\bullet(A_i)_r=\tau^{\geq r}(\mc K^\bullet(\bow_i;A_i)[-2r])$.

For any $V\in \mf V$ let $\Phi(V)=\{t\in\{1,\ldots,n\} \mid V\subseteq U_t\}$. If $t\in\Phi(V)=\{t_1,\ldots,t_m\}$, we express $\mc A(V)$ in term of generators and relations as
\[
\mc A(V,t)=k[y_0,\ldots,\widehat{y_t},\ldots,y_n,y_{t_1}^{-1},\ldots,\widehat{y_{t}^{-1}},\ldots,y_{t_m}^{-1}]/(G_t, y_{t_1}y_{t_1}^{-1}-1,\ldots, y_{t_m}y_{t_m}^{-1}-1).
\]
Since $\mc A(V,t)$ is a localization of $A_t$, there is a quasi-isomorphism
\[
\be\colon B\ot_{A_t}\mc L^\bullet(A_t)\To \bar{C}^\bullet(\mc A(V,t),B)
\]
for any flat morphism $\mc A(V,t)\to B$ by the last paragraph of \S\ref{sec:HHaffine}. If $s$ also belongs to $\Phi(V)$, the canonical isomorphism $\mc A(V,t)\to \mc A(V,s)$ is denoted by $\zeta_{t,s}$. Unfortunately, $\zeta_{t,s}$ is not compatible with the differentials of $\mc L^\bullet(A_t)$ and $\mc L^\bullet(A_s)$, namely, the square
\[
\xymatrix{
B\ot_{A_t}\mc L^\bullet(A_t)\ar[r]^-{\be}\ar[d]_-{\zeta_{t,s}} & \bar{C}^\bullet(\mc A(V,t),B)  \\
B\ot_{A_s}\mc L^\bullet(A_s)\ar[r]^-{\be} & \bar{C}^\bullet(\mc A(V,s),B) \ar[u]_-{\zeta_{t,s}^*}
}
\]
fails to be commutative. So one does not expect that the complexes $\mc L^\bullet(\mc A(V))$ for all affine pieces $V$ can be made into a complex $\mc L^\bullet$ of sheaves on $X$ equipped with nice restriction maps. The reason is that the $\mc L^\bullet(\mc A(V))$'s are too small. In order to study $\mc A$ globally, we have to put on their weight. Their ``food'' should be convenient for computation in principle.

It follows from Euler's formula
\[
\sum_{i=0}^{n}\frac{\pl F}{\pl x_i}\cdot x_i=d\cdot F
\]
that
\[
\bou=\biggl(\frac{\pl F}{\pl x_0},\frac{\pl F}{\pl x_1},\ldots, \frac{\pl F}{\pl x_n}\biggr)
\text{ and }
\bov=(x_0,x_1,\ldots,x_n)
\]
make up an $n$-POS in $S$. Also, there is an $n$-POS $(\bou_i,\bov_i)$ in $A_i$:
\[
\bou_i=\biggl(\frac{\pl G_i}{\pl y_0},\ldots,\frac{\pl G_i}{\pl y_{i-1}}, H_i, \frac{\pl G_i}{\pl y_{i+1}},\ldots, \frac{\pl G_i}{\pl y_n}\biggr)
\text{ and }
\bov_i=(y_0,\ldots,y_{i-1},1,y_{i+1},\ldots, y_n)
\]
where
\[
H_i=\frac{\pl F}{\pl x_i}(y_0, y_1,\ldots, y_{i-1},1,y_{i+1},\ldots, y_n).
\]
Since $\bow_i$ is the subsequence of $\bou_i$ by deleting $H_i$, the results from \S\ref{sec:mixedcomp} apply. As before we get the mixed complex $\mc K^{\bullet}(\bou,\bov; S)$ and the double complex $\mc K^{\bullet,\bullet}(\bou,\bov; S)$.

Let $r\geq 0$ and let us consider $\tau^r\mc K^{\bullet,\bullet}(\bou,\bov; S)$. We twist the degrees of its entries as in Figure \ref{fig:truncation-S}
\begin{figure}[!htp]
\[
\xymatrix@C=8mm@R=8mm{
S(r)^{\binom{n+1}{r}} \ar[r]^-{\pl_{\bou}}  & S(r+d-1)^{\binom{n+1}{r-1}} \ar[r]^-{\pl_{\bou}}  & \cdots \ar[r]^-{\pl_{\bou}} & S(rd-d+1)^{\binom{n+1}{1}} \ar[r]^-{\pl_{\bou}} & S(rd) \\
S(r-1)^{\binom{n+1}{r-1}} \ar[r]^-{\pl_{\bou}}\ar[u]_-{\pl_{\bov}} & S(r+d-2)^{\binom{n+1}{r-2}} \ar[r]^-{\pl_{\bou}}\ar[u]_-{\pl_{\bov}} & \cdots \ar[r]^-{\pl_{\bou}}  & S(rd-d) \ar[u]_-{\pl_{\bov}} \\
\vdots \ar[u]_-{\pl_{\bov}} & \vdots \ar[u]_-{\pl_{\bov}} & \iddots  \\
S(1)^{\binom{n+1}{1}} \ar[r]^-{\pl_{\bou}}\ar[u]_-{\pl_{\bov}} & S(d) \ar[u]_-{\pl_{\bov}} \\
S \ar[u]_-{\pl_{\bov}}
}
\]
\caption{Double complex $\tau^r\mc K^{\bullet,\bullet}(\bou,\bov; S)$}
\label{fig:truncation-S}
\end{figure}
so that it is made into a double complex of graded $S$-modules. The associated total complex gives rise to a complex of sheaves
\[
\mc F_r^{\bullet}:\quad \mc O_X\To \mc O_X(1)^{n+1}\To\cdots \To \mc O_X(rd-d+1)^{n+1}\To \mc O_X(rd).
\]
We in turn have double complexes $\mc E_r^{\bullet,\bullet}$, $\mc G_r^{\bullet,\bullet}$ and $\mc H_r^{\bullet,\bullet}$ as follows:
\[
\mc{E}^{p,q}_r=\mbf{C}_{\simp}^{\prime p}(\mf V,\mc F_r^q|_{\mf V}),\quad
\mc{G}^{p,q}_r=\check{\mbf{C}}^{\prime p}(\mf V,\mc F_r^q),\quad
\mc{H}^{p,q}_r=\check{\mbf{C}}^{\prime p}(\mf U,\mc F_r^q).
\]
Their associated complexes are denoted by $\mc E_r^{\bullet}$, $\mc G_r^{\bullet}$ and $\mc H_r^{\bullet}$ respectively. Since $\mf V$ is a refinement of $\mf U$, we fix a map $\lam\colon \mf V\to \mf U$ such that $V\subseteq \lam(V)$ for all $V\in\mf V$. There is a quasi-isomorphism $\bar{\lam}\colon\mc H_r^{\bullet}\to\mc G_r^{\bullet}$ induced by $\lam$, defined by for all $f\in \mc H_r^{p,q}$,
\[
\bar{\lam}(f)_{V_{i_0}\ldots V_{i_p}}=f_{\lam(V_{i_0})\ldots\lam(V_{i_p})}.
\]
Composing $\bar{\lam}$ with the explicit quasi-isomorphism $\mc G_r^{\bullet}\to\mc E_r^{\bullet}$ given in \cite{DLL:defo-qch}, we obtain a quasi-isomorphism $\bar{\bar{\lam}}\colon \mc H_r^{\bullet}\to\mc E_r^{\bullet}$ given by
\begin{equation}\label{eq:bar-bar-lambda}
\bar{\bar{\lam}}(f)_{V_{j_0}\subseteq\cdots\subseteq V_{j_p}}=f_{\lam(V_{j_0})\ldots\lam(V_{j_p})}.
\end{equation}
Let $\bar{\mbf C}^{\prime \bullet,\bullet}(\mc A)=\oplus_{r\in\nan}\bar{\mbf C}_r^{\prime \bullet,\bullet}(\mc A)$ be the Hodge decomposition. Our goal is to construct a family of morphisms $\mc E_r^{\bullet,\bullet}\to \bar{\mbf C}_r^{\prime \bullet,\bullet}(\mc A)$ of double complexes for all $r$ that give rise to quasi-isomorphisms $\mc E_r^{\bullet}\to\bar{\mbf C}_{\GS}^{\prime\bullet}(\mc A)_r$. Since the cohomology of $\bar{\mbf C}_{\GS}^{\prime\bullet}(\mc A)$ turns out to be isomorphic to the Hochschild cohomology of $X$ (see \cite[Thm.\ 7.8.1]{Lowen-VandenBergh:hoch}), the cohomology $HH^\bullet(X)$ can be computed by $\mc H^\bullet:=\oplus_{r\in\nan}\mc H_r^\bullet$, namely, $HH^i(X)\cong H^i(\mc H^\bullet)$.

Let $\si\in\mc N_p(\mf V)$ be a $p$-simplex and consider $t$, $s\in\Phi(\si_\diamond)$. We have quasi-isomorphisms
\[
\be_t\colon  \bigoplus_{r\in\nan}\tau^{\geq r}(\mc K^\bullet(\bow_t;\mc A({}_\diamond\si,t))[-2r])\cong \mc A({}_\diamond\si,t)\ot_{A_t}\mc L^\bullet(A_t)\To \bar{C}^\bullet(\mc A(\si_\diamond,t),\mc A({}_\diamond\si,t))
\]
and $\be_s$, which is defined similarly. Let $\ppl_t/\pl y_i$, $\pmul_t$ and $\ppl_s/\pl y_i$, $\pmul_s$ be the resulting Hochschild cochains as defined in \eqref{eqpart} and \eqref{eqmu}. According to the generators and relations of $\mc A(\si_\diamond,t)$ and $\mc A({}_\diamond\si,t)$, we can regard $\ppl_t/\pl y_i$, $\pmul_t$ to be cochains in $\bar{C}^\bullet(\mc A(\si_\diamond,t),\mc A({}_\diamond\si,t))$ by abuse of notation, and similarly for $\ppl_s/\pl y_i$, $\pmul_s$.

\begin{lemma}\label{lem:pseudopartial-compatible}
  Let $\zeta'_{t,s}\colon \bar{C}^\bullet(\mc A(\si_\diamond,t),\mc A({}_\diamond\si,t))\to \bar{C}^\bullet(\mc A(\si_\diamond,s),\mc A({}_\diamond\si,s))$ be the isomorphism induced by $\zeta_{t,s}$. Then
  \begin{enumerate}
    \item $\zeta'_{t,s}(\ppl_t/\pl y_i)=y_t\cdot\ppl_s/\pl y_i$ if $i\neq t$, $s$.
    \item $\zeta'_{t,s}(\ppl_t/\pl y_s)=-\sum_{i\neq s}y_ty_i\cdot\ppl_s/\pl y_i$.
    \item $\zeta'_{t,s}(\pmul_t)=y_t^d\cdot\pmul_s$.
  \end{enumerate}
\end{lemma}

\begin{proof}
  (1) (2) Choose any $f=y_0^{p_0}\cdots y_{s-1}^{p_{s-1}}y_{s+1}^{p_{s+1}}\cdots y_n^{p_n}\in\mc B_{\mc A(\si_\diamond,s)}$ and let $|f|=\sum_{i\neq t,s}p_i$. We have
  \begin{align*}
    \zeta'_{t,s}\biggl(\frac{\ppl_t}{\pl y_i}\biggr)(f)&=\zeta_{t,s}\circ\frac{\ppl_t}{\pl y_i}\circ\zeta_{s,t}(f)\\
    &=\zeta_{t,s}\circ\frac{\ppl_t}{\pl y_i}(y_0^{p_0}\cdots y_{t-1}^{p_{t-1}}y_{t+1}^{p_{t+1}}\cdots y_s^{-|f|}\cdots y_n^{p_n})\\
    &=\zeta_{t,s}(p_iy_0^{p_0}\cdots y_i^{p_i-1}\cdots y_{t-1}^{p_{t-1}}y_{t+1}^{p_{t+1}}\cdots y_s^{-|f|}\cdots y_n^{p_n})\\
    &=p_iy_0^{p_0}\cdots y_i^{p_i-1}\cdots y_t^{p_t+1}\cdots y_{s-1}^{p_{s-1}}y_{s+1}^{p_{s+1}}\cdots y_n^{p_n}\\
    &=y_t\frac{\ppl_s}{\pl y_i}(f)
  \end{align*}
  for all $i\neq t$, $s$, and
  \begin{align*}
    \zeta'_{t,s}\biggl(\frac{\ppl_t}{\pl y_s}\biggr)(f)&=\zeta_{t,s}\circ\frac{\ppl_t}{\pl y_s}(y_0^{p_0}\cdots y_{t-1}^{p_{t-1}}y_{t+1}^{p_{t+1}}\cdots y_s^{-|f|}\cdots y_n^{p_n})\\
    &=\zeta_{t,s}(-|f|y_0^{p_0}\cdots y_{t-1}^{p_{t-1}}y_{t+1}^{p_{t+1}}\cdots y_s^{-|f|-1}\cdots y_n^{p_n})\\
    &=-|f|y_0^{p_0}\cdots  y_t^{p_t+1}\cdots y_{s-1}^{p_{s-1}}y_{s+1}^{p_{s+1}}\cdots y_n^{p_n}\\
    &=-y_t|f|f\\
    &=-\sum_{i\neq s}y_ty_i\frac{\ppl_s}{\pl y_i}(f).
  \end{align*}

  (3) Let $g=y_0^{q_0}\cdots y_{s-1}^{q_{s-1}}y_{s+1}^{q_{s+1}}\cdots y_n^{q_n}\in\mc B_{\mc A(\si_\diamond,s)}$ and $|g|=\sum_{i\neq t,s}q_i$. Assume $p_0+q_0\geq d$. Then
  \begin{align*}
    \zeta'_{t,s}(\pmul_t)(f,g)&=\zeta_{t,s}\circ\pmul_t(y_0^{p_0}\cdots y_{t-1}^{p_{t-1}}y_{t+1}^{p_{t+1}}\cdots y_s^{-|f|}\cdots y_n^{p_n}, y_0^{q_0}\cdots y_{t-1}^{q_{t-1}}y_{t+1}^{q_{t+1}}\cdots y_s^{-|g|}\cdots y_n^{q_n})\\
    &=\zeta_{t,s}(y_0^{p_0+q_0-d}\cdots y_{t-1}^{p_{t-1}+q_{t-1}}y_{t+1}^{p_{t+1}+q_{t+1}}\cdots y_s^{-|f|-|g|}\cdots y_n^{p_n+q_n})\\
    &=y_0^{p_0+q_0-d}\cdots y_t^{p_t+q_t+d}\cdots y_{s-1}^{p_{s-1}+q_{s-1}}y_{s+1}^{p_{s+1}+q_{s+1}}\cdots y_n^{p_n+q_n}\\
    &=y_t^d\cdot\pmul_s(f,g).
  \end{align*}
  On the other hand, $\zeta'_{t,s}(\pmul_t)(f,g)=0=y_t^d\cdot\pmul_s(f,g)$ trivially holds if $p_0+q_0< d$.
\end{proof}

There are proportional $n$-POS $(\zeta_{t,s}(\bou_t),\zeta_{t,s}(\bov_t))$, $(\bou_s,\bov_s)$ in $\mc A({}_\diamond\si,s)$ with $\bou_s=y_t^{d-1}\zeta_{t,s}(\bou_t)$ and $\bov_s=y_t\zeta_{t,s}(\bov_t)$. There is an isomorphism
\[
\xi_{t,s,(r)}^{\Tot}\colon \Tot(\tau^{r}\mc K^{\bullet,\bullet}(\zeta_{t,s}(\bou_t),\zeta_{t,s}(\bov_t); \mc A({}_\diamond\si,s)))\To \Tot(\tau^{r}\mc K^{\bullet,\bullet}(\bou_s,\bov_s; \mc A({}_\diamond\si,s)))
\]
as given in \eqref{eq:iso-double}. Since the $t$-th, $s$-th components of $\zeta_{t,s}(\bov_t)$ and $\bov_s$ are invertible, we have the diagram
\begin{equation}\label{eq:cd1}
\xymatrix{
\Tot(\tau^{r}\mc K^{\bullet,\bullet}(\zeta_{t,s}(\bou_t),\zeta_{t,s}(\bov_t); \mc A({}_\diamond\si,s))) \ar[dd]_-{\xi_{t,s,(r)}^{\Tot}}\ar[r]^-{\pi_{t,(r)}} & \tau^{\geq r}(\mc K^{\bullet}(\zeta_{t,s}(\bow_t); \mc A({}_\diamond\si,s))[-2r]) \ar[d]^-{\be'_{t,(r)}} \\
& \bar{C}^\bullet(\mc A(\si_\diamond,s),\mc A({}_\diamond\si,s)) \\
\Tot(\tau^{r}\mc K^{\bullet,\bullet}(\bou_s,\bov_s; \mc A({}_\diamond\si,s))) \ar[r]^-{\pi_{s,(r)}} & \tau^{\geq r}(\mc K^{\bullet}(\bow_s; \mc A({}_\diamond\si,s))[-2r]) \ar[u]_-{\be_{s,(r)}}
}
\end{equation}
where $\be'_{t,(r)}$ is induced by $\be_{t,(r)}$ and $\zeta_{t,s}$.

\begin{lemma}\label{lem:cd-compatible}
  The diagram \eqref{eq:cd1} is commutative.
\end{lemma}

\begin{proof}
  Choose any base element $E=e_{i_1}\wedge\cdots \wedge e_{i_p}\in \mc K^{p}(\zeta_{t,s}(\bou_t),\zeta_{t,s}(\bov_t); \mc A({}_\diamond\si,s))$. When viewed as a cochain in $\tau^{r}\mc K^{\bullet,\bullet}(\zeta_{t,s}(\bou_t),\zeta_{t,s}(\bov_t); \mc A({}_\diamond\si,s))$, $E$ locates in position $(r-p,r)$. So $\xi_{t,s,(r)}^{\Tot}(E)=(y_t^{d-1})^{r-p}y_t^rE=y_t^{r+(d-1)(r-p)}E$. Let us prove the lemma by a case-by-case argument.

  If $t$, $s\notin\{i_1,\ldots,i_p\}$, then
  \begin{align*}
  \be'_{t,(r)}\circ\pi_{t,(r)}(E)&=\be'_{t,(r)}(\mf e_{i_1\ldots i_p}\mf s^{(r-p)})\\
  &=\zeta_{t,s}\circ(-1)^{\binom{p}{2}}\biggl(\frac{\ppl_t}{\pl y_{i_1}}\cup\cdots\cup \frac{\ppl_t}{\pl y_{i_{p}}}\cup\pmul_t^{\cup (r-p)}\biggr)\circ(\zeta_{s,t})^{\ot (2r-p)}\\
  &=(-1)^{\binom{p}{2}}\zeta'_{t,s}\biggl(\frac{\ppl_t}{\pl y_{i_1}}\biggr)\cup\cdots\cup \zeta'_{t,s}\biggl(\frac{\ppl_t}{\pl y_{i_{p}}}\biggr)\cup \bigl(\zeta'_{t,s}(\pmul_t)\bigr)^{\cup (r-p)}\\
  &=(-1)^{\binom{p}{2}}y_t\frac{\ppl_s}{\pl y_{i_1}}\cup\cdots\cup y_t\frac{\ppl_s}{\pl y_{i_{p}}}\cup (y_t^d\cdot\pmul_s)^{\cup (r-p)}\\
  &=y_t^{r+(d-1)(r-p)}(-1)^{\binom{p}{2}}\frac{\ppl_s}{\pl y_{i_1}}\cup\cdots\cup \frac{\ppl_s}{\pl y_{i_{p}}}\cup \pmul_s^{\cup (r-p)}\\
  &=y_t^{r+(d-1)(r-p)}\be_{s,(r)}(\mf e_{i_1\ldots i_p}\mf s^{(r-p)})\\
  &=\be_{s,(r)}\circ\pi_{s,(r)}\circ\xi_{t,s,(r)}^{\Tot}(E).
  \end{align*}

  If $s\notin\{i_1,\ldots,i_p\}$ and $t=i_j$ for some $j$, then
  \begin{align*}
  \be'_{t,(r)}\circ\pi_{t,(r)}(E)&=-\sum_{m\neq t}y_my_t^{-1}\zeta_{t,s}\circ(-1)^{\binom{p}{2}}\biggl(\frac{\ppl_t}{\pl y_{i_1}}\cup\cdots\cup \frac{\ppl_t}{\pl y_{i_{j-1}}}\cup \frac{\ppl_t}{\pl y_{m}}\cup \frac{\ppl_t}{\pl y_{i_{j+1}}}\\
  &\varphantom{=}{}\cup\cdots\cup \frac{\ppl_t}{\pl y_{i_{p}}}\cup\pmul_t^{\cup (r-p)}\biggr)\circ(\zeta_{s,t})^{\ot (2r-p)}\\
  &=-\sum_{m\neq t}y_my_t^{-1}(-1)^{\binom{p}{2}}\zeta'_{t,s}\biggl(\frac{\ppl_t}{\pl y_{i_1}}\biggr)\cup\cdots\cup \zeta'_{t,s}\biggl(\frac{\ppl_t}{\pl y_{m}}\biggr)\cup\cdots\\
  &\varphantom{=}{}\cup \zeta'_{t,s}\biggl(\frac{\ppl_t}{\pl y_{i_p}}\biggr)\cup \bigl(\zeta'_{t,s}(\pmul_t)\bigr)^{\cup (r-p)}\\
  &=-\sum_{m\neq t,s}y_my_t^{r+(d-1)(r-p)-1}(-1)^{\binom{p}{2}}\frac{\ppl_s}{\pl y_{i_1}}\cup\cdots\cup \frac{\ppl_s}{\pl y_{m}}\cup\cdots\cup \frac{\ppl_s}{\pl y_{i_p}}\\
  &\varphantom{=}{}\cup (\pmul_s)^{\cup (r-p)}-y_t^{r+(d-1)(r-p)-2}(-1)^{\binom{p}{2}}\frac{\ppl_s}{\pl y_{i_1}}\cup\cdots\\
  &\varphantom{=}{}\cup \biggl(-\sum_{i\neq s}y_ty_i\frac{\ppl_s}{\pl y_{i}}\biggr)\cup\cdots\cup \frac{\ppl_s}{\pl y_{i_p}}\cup (\pmul_s)^{\cup (r-p)}\\
  &=y_t^{r+(d-1)(r-p)}(-1)^{\binom{p}{2}}\frac{\ppl_s}{\pl y_{i_1}}\cup\cdots\cup \frac{\ppl_s}{\pl y_{t}}\cup\cdots\cup \frac{\ppl_s}{\pl y_{i_p}}\cup (\pmul_s)^{\wedge (r-p)}\\
  &=y_t^{r+(d-1)(r-p)}\be_{s,(r)}(\mf e_{i_1\ldots i_p}\mf s^{(r-p)})\\
  &=\be_{s,(r)}\circ\pi_{s,(r)}\circ\xi_{t,s,(r)}^{\Tot}(E).
  \end{align*}

  If $t\notin\{i_1,\ldots,i_p\}$ and $s=i_l$ for some $l$, then
  \begin{align*}
  \be'_{t,(r)}\circ\pi_{t,(r)}(E)&=(-1)^{\binom{p}{2}}\zeta'_{t,s}\biggl(\frac{\ppl_t}{\pl y_{i_1}}\biggr)\cup\cdots\cup \zeta'_{t,s}\biggl(\frac{\ppl_t}{\pl y_{i_{p}}}\biggr)\cup \bigl(\zeta'_{t,s}(\pmul_t)\bigr)^{\cup (r-p)}\\
  &=y_t^{r+(d-1)(r-p)-1}(-1)^{\binom{p}{2}}\frac{\ppl_s}{\pl y_{i_1}}\cup\cdots\cup \biggl(-\sum_{m\neq s}y_ty_m\frac{\ppl_s}{\pl y_{m}}\biggr) \cup\cdots\\
  &\varphantom{=}{}\cup \frac{\ppl_s}{\pl y_{i_{p}}}\cup \pmul_s^{\wedge (r-p)}\\
  &=-y_t^{r+(d-1)(r-p)}\sum_{m\neq s}y_m(-1)^{\binom{p}{2}}\frac{\ppl_s}{\pl y_{i_1}}\cup\cdots\cup \frac{\ppl_s}{\pl y_{m}} \cup\cdots\cup \frac{\ppl_s}{\pl y_{i_{p}}}\\
  &\varphantom{=}{}\cup \pmul_s^{\cup (r-p)}\\
  &=-y_t^{r+(d-1)(r-p)}\sum_{m\neq s}y_m\be_{s,(r)}(\mf e_{i_1\ldots m\ldots i_p}\mf s^{(r-p)})\\
  &=\be_{s,(r)}\circ\pi_{s,(r)}\circ\xi_{t,s,(r)}^{\Tot}(E).
  \end{align*}

  If $t=i_j$ and $s=i_l$ for some $j$, $l$ then
  \begin{align*}
  \be'_{t,(r)}\circ\pi_{t,(r)}(E)&=-\sum_{m\neq t}y_my_t^{-1}\zeta_{t,s}\circ(-1)^{\binom{p}{2}}\biggl(\frac{\ppl_t}{\pl y_{i_1}}\cup\cdots\cup \frac{\ppl_t}{\pl y_{i_{j-1}}}\cup \frac{\ppl_t}{\pl y_{m}}\cup \frac{\ppl_t}{\pl y_{i_{j+1}}}\\
  &\varphantom{=}{}\cup\cdots\cup \frac{\ppl_t}{\pl y_{s}}\cup\cdots\cup \frac{\ppl_t}{\pl y_{i_{p}}}\cup\pmul_t^{\cup (r-p)}\biggr)\circ(\zeta_{s,t})^{\ot (2r-p)}\\
  &=-\sum_{m\neq t,s}y_my_t^{-1}(-1)^{\binom{p}{2}}\zeta'_{t,s}\biggl(\frac{\ppl_t}{\pl y_{i_1}}\biggr)\cup\cdots\cup \zeta'_{t,s}\biggl(\frac{\ppl_t}{\pl y_{m}}\biggr)\cup\cdots\\
  &\varphantom{=}{}\cup \zeta'_{t,s}\biggl(\frac{\ppl_t}{\pl y_{s}}\biggr)\cup \cdots\cup \zeta'_{t,s}\biggl(\frac{\ppl_t}{\pl y_{i_p}}\biggr)\cup \bigl(\zeta'_{t,s}(\pmul_t)\bigr)^{\cup (r-p)}\\
  &=-\sum_{m\neq t,s}y_my_t^{r+(d-1)(r-p)-2}(-1)^{\binom{p}{2}}\frac{\ppl_s}{\pl y_{i_1}}\cup\cdots\cup \frac{\ppl_s}{\pl y_{m}}\cup\cdots\\
  &\varphantom{=}{}\cup \biggl(-\sum_{i\neq s}y_ty_i\frac{\ppl_s}{\pl y_{i}}\biggr)\cup\cdots\cup \frac{\ppl_s}{\pl y_{i_p}}\cup (\pmul_s)^{\cup (r-p)}\\
  &=y_t^{r+(d-1)(r-p)-1}\sum_{m\neq t,s}\sum_{i\neq s}y_my_i(-1)^{\binom{p}{2}}\frac{\ppl_s}{\pl y_{i_1}}\cup\cdots\cup \frac{\ppl_s}{\pl y_{m}}\cup\cdots\\
  &\varphantom{=}{}\cup \frac{\ppl_s}{\pl y_{i}}\cup\cdots\cup \frac{\ppl_s}{\pl y_{i_p}}\cup (\pmul_s)^{\cup (r-p)}\\
  &=y_t^{r+(d-1)(r-p)-1}\sum_{m\neq t,s}y_my_t(-1)^{\binom{p}{2}}\frac{\ppl_s}{\pl y_{i_1}}\cup\cdots\cup \frac{\ppl_s}{\pl y_{m}}\cup\cdots\\
  &\varphantom{=}{}\cup \frac{\ppl_s}{\pl y_{t}}\cup\cdots\cup \frac{\ppl_s}{\pl y_{i_p}}\cup (\pmul_s)^{\cup (r-p)}\\
  &=-y_t^{r+(d-1)(r-p)}\sum_{m\neq s}y_m(-1)^{\binom{p}{2}}\frac{\ppl_s}{\pl y_{i_1}}\cup\cdots\cup \frac{\ppl_s}{\pl y_{t}}\cup\cdots\cup \frac{\ppl_s}{\pl y_{m}}\cup\\
  &\varphantom{=}{}\cdots\cup \frac{\ppl_s}{\pl y_{i_p}}\cup (\pmul_s)^{\cup (r-p)}\\
  &=-y_t^{r+(d-1)(r-p)}\sum_{m\neq s}y_m\be_{s,(r)}(\mf e_{i_1\ldots m\ldots i_p}\mf s^{(r-p)})\\
  &=\be_{s,(r)}\circ\pi_{s,(r)}\circ\xi_{t,s,(r)}^{\Tot}(E). \qedhere
  \end{align*}
\end{proof}

Therefore we obtain a commutative diagram
\[
\xymatrix@C=10mm{
\displaystyle\bigoplus_{r\in\nan}\Tot(\tau^{r}\mc K^{\bullet,\bullet}(\bou_t,\bov_t; \mc A({}_\diamond\si,t))) \ar[d]_-{\zeta_{t,s}}\ar[r]^-{\be_t\circ\pi_{t}} & \bar{C}^\bullet(\mc A(\si_\diamond,t),\mc A({}_\diamond\si,t)) \ar[d]^-{\zeta'_{t,s}} \\
\displaystyle\bigoplus_{r\in\nan}\Tot(\tau^{r}\mc K^{\bullet,\bullet}(\zeta_{t,s}(\bou_t),\zeta_{t,s}(\bov_t); \mc A({}_\diamond\si,s))) \ar[d]_-{\xi_{t,s}^{\Tot}}\ar[r]^-{\be'_t\circ\pi_{t}} & \bar{C}^\bullet(\mc A(\si_\diamond,s),\mc A({}_\diamond\si,s)) \ar@{=}[d] \\
\displaystyle\bigoplus_{r\in\nan}\Tot(\tau^{r}\mc K^{\bullet,\bullet}(\bou_s,\bov_s; \mc A({}_\diamond\si,s))) \ar[r]^-{\be_s\circ\pi_{s}} & \bar{C}^\bullet(\mc A(\si_\diamond,s),\mc A({}_\diamond\si,s))
}
\]
where the vertical morphisms are isomorphisms and the horizontal ones are quasi-isomorphisms. Let $\xi'_{t,s}=\xi_{t,s}^{\Tot}\circ\zeta_{t,s}$. The twisting number $r+(d-1)(r-d)$ of the $(r-p,r)$-entry in Figure \ref{fig:truncation-S} coincides with the exponent of $y_t$ in the proof of Lemma \ref{lem:cd-compatible}. This is equivalent to say that $\xi'_{t,s}$ is the canonical automorphism of $\mc F^\bullet({}_\diamond\si)$ if we write $\mc A({}_\diamond\si)$ in terms of different generators and relations. Moreover, it is easy to check the coherence conditions
\[
\xi'_{s,u}\circ\xi'_{t,s}=\xi'_{t,u},\quad \zeta'_{s,u}\circ\zeta'_{t,s}=\zeta'_{t,u}
\]
hold true for any additional $u\in\Phi(\si_\diamond)$. This gives rise to well-defined morphisms
\[
\ga_{\si}=\be\circ\pi\colon\mc F^\bullet({}_\diamond\si)\to \bar{C}^\bullet(\mc A(\si_\diamond), \mc A({}_\diamond\si))
\]
for all simplices $\si\in\mc N_\bullet(\mf V)$ which commute with simplicial differentials. Remember that $\be$ and $\pi$ preserve the Hodge decomposition. These facts are summarized as

\begin{theorem}\label{thm:q-iso-GS}
Let $\mc E^{\bullet,\bullet}=\oplus_{r\in\nan}\mc E^{\bullet,\bullet}_r$. The morphisms $\ga_{\si}\colon\mc F^\bullet({}_\diamond\si)\to \bar{C}^\bullet(\mc A(\si_\diamond), \mc A({}_\diamond\si))$ for all simplices $\si$ on $\mf V$ constitute a morphism $\ga\colon\mc E^{\bullet,\bullet}\to\bar{\mbf C}^{\prime\bullet,\bullet}(\mc A)$ of double complexes that gives rise to a quasi-isomorphism $\mc E^{\bullet}\to\bar{\mbf C}^{\prime\bullet}_{\GS}(\mc A)$. Moreover, $\ga$ preserves the Hodge decomposition.
\end{theorem}

Recall that $F$ is required to contain $x_0^d$ as a summand. We claim that this condition is not too restrictive. In fact, a homogeneous polynomial $F\in R$ of degree $d$ always has the expression
\[
F(x_0,x_1,\ldots,x_n)=\sum_{i_1,\ldots,i_n}\mu_{i_1,\ldots,i_n}x_0^{d-i_1-\cdots-i_n}x_1^{i_1}\cdots x_n^{i_n}
\]
where $\mu_{i_1,\ldots,i_n}\in k$. Let $\Si \colon R\to R$ be the automorphism of graded algebras determined by
\[
x_0\mapsto x_0,\quad x_j\mapsto x_j+\lam_jx_0\quad (1\leq j\leq n)
\]
where $\lam_j$ are undetermined coefficients. So
\begin{align*}
\Si(F)&=\sum_{i_1,\ldots,i_n}\mu_{i_1,\ldots,i_n}x_0^{d-i_1-\cdots-i_n}(x_1+\lam_1x_0)^{i_1}\cdots (x_n+\lam_nx_0)^{i_n}\\
&=\sum_{i_1,\ldots,i_n}\mu_{i_1,\ldots,i_n}\lam_1^{i_1}\cdots\lam_n^{i_n}x_0^d+\cdots.
\end{align*}
If $\mu_{0,\ldots,0}=0$, there exist at least an array $(i_1,\ldots, i_n)$ such that $\mu_{i_1,\ldots,i_n}\neq 0$ and one of $i_l>0$ for $1\leq l\leq n$. So one can choose proper $\lam_j$ making
\[
\sum_{i_1,\ldots,i_n}\mu_{i_1,\ldots,i_n}\lam_1^{i_1}\cdots\lam_n^{i_n}=1.
\]
Therefore, $\Proj R/(F)\cong \Proj R/(\Si(F))$ with $\Si(F)$ containing $x_0^d$ as a summand. With this isomorphism, most conclusions in \S \ref{sec:HHproj} remain true for an arbitrary homogeneous polynomial $F$. Throughout the paper, to facilitate the computations, we maintain the condition that $F$ contains $x_0^d$ as a summand.

\subsection{The cotangent complex of a hypersurface}\label{subsec:cotangent-complex}
As stated in the beginning of \S\ref{sec:HHproj}, an explicit expression of $\mbb L_{X/Y}$ is given by Berthelot, when $X\to Y$ factors as a closed immersion $X\to X'$ followed by a smooth morphism $X'\to Y$. Let us recall it in the special case when $Y=\mathrm{Spec}\,k$ and $X=\Proj S$.

Obviously, $X'$ can be chosen to be $\Proj R=\mbb P^n$ and so the factorization $X\xrightarrow{\imath}\mbb P^n\to \mathrm{Spec}\,k$ satisfies the condition. Let $\mc O=\mc O_{\mbb P^n}$, and let $\mc I\subset \mc O$ be the sheaf of ideals determined by the closed immersion $X\to \mbb P^n$. By definition, $\mbb L_{X/k}^0=\imath^*\Om_{\mbb P^n}$, $\mbb L^{-1}_{X/k}=\mc I/\mc I^2$, and other $\mbb L_{X/k}^j$ are all zero, the differential $\mc I/\mc I^2=\imath^*\mc I\to \imath^*\Om_{\mbb P^n}$ is induced by $\mc I\hookrightarrow \mc O \xrightarrow{\mathrm{d}}\Om_{\mbb P^n}$.

Note that there is a complex of graded modules
\[
0\To S(-d)\xrightarrow[\quad]{\pl_{\bou}} S(-1)^{n+1}\xrightarrow[\quad]{\pl_{\bov}} S\To 0
\]
concentrated in degrees $-1$, $0$ and $1$. It gives rise to a complex
\[
0\To\mc O_X(-d)\xrightarrow[\quad]{\pl_{\bou}} \mc O_X(-1)^{n+1} \xrightarrow[\quad]{\pl_{\bov}} \mc O_X\To 0
\]
of $\mc O_X$-modules. It is easy to check that  the complex is the same as $\mc F_1^{\bullet\vee}[-1]$ where $(-)^\vee=\mc{H}om_{\mc O_X}(-,\mc O_X)$.

\begin{proposition}\label{prop:cotangent-complex}
There is a quasi-isomorphism $\mbb L_{X/k}\to \mc F_1^{\bullet\vee}[-1]$.
\end{proposition}

\begin{proof}
First of all, the sheaf of ideals $\mc I$ corresponds to the principal ideal $(F)$. So as $\mc O_X$-modules, $\mc I/\mc I^2$ is isomorphic to $\mc O_X(-d)$ since $\deg F=d$.

Next, there is an exact sequence
\begin{equation}\label{eq:Omega-P}
0\To \Om_{\mbb P^n}\To \mc O(-1)^{n+1}\xrightarrow[\quad]{\pl_{\bov}} \mc O\To 0
\end{equation}
by \cite[Thm.\ 8.13]{Hartshorne:alg-geom}. Since $\imath^*$ is right exact, we have an exact sequence
\[
\imath^*\Om_{\mbb P^n}\To \mc O_X(-1)^{n+1}\xrightarrow[\quad]{\pl_{\bov}} \mc O_X\To 0.
\]
We claim that the map $\imath^*\Om_{\mbb P^n}\to \mc O_X(-1)^{n+1}$ is injective. In fact, for any point $x\in X$, by localizing \eqref{eq:Omega-P} at the point $\imath(x)$ we obtain a split exact sequence
\[
0\To \Om_{\mbb P^n,\imath(x)}\To \mc O(-1)^{n+1}_{\imath(x)}\To \mc O_{\imath(x)}\To 0
\]
since $\mc O_{\imath(x)}$ is a free module over itself. After tensoring it with $\mc O_{X,x}$ over $\mc O_{\imath(x)}$, we in turn have that
\[
0\To \imath^*\Om_{\mbb P^n,x}\To \mc O_X(-1)^{n+1}_x\To \mc O_{X,x}\To 0
\]
is split exact. It follows that
\begin{equation}\label{eq:Omega-iP}
0\To\imath^*\Om_{\mbb P^n}\To \mc O_X(-1)^{n+1}\xrightarrow[\quad]{\pl_{\bov}} \mc O_X\To 0
\end{equation}
is actually exact.

Thus there exists a morphism $\mbb L_{X/k}\to \mc F_1^{\vee\bullet}[-1]$ given by
\[
\xymatrix{
\cdots \ar[r] & 0\ar[r] & \mc O_X(-d) \ar[r]^-{\pl_{\bou}} & \mc O_X(-1)^{n+1} \ar[r]^-{\pl_{\bov}} & \mc O_X \ar[r] & 0 \ar[r] & \cdots \\
\cdots \ar[r] & 0\ar[r]\ar[u] & \mc I/\mc I^2 \ar[r]^{\mathrm{d}}\ar[u]^-{\cong} & \imath^*\Om_{\mbb P^n} \ar[r]\ar[u] & 0  \ar[r]\ar[u] & 0 \ar[r]\ar[u] & \cdots
}
\]
which is a quasi-isomorphism since \eqref{eq:Omega-iP} is exact.
\end{proof}

\begin{corollary}
In the derived category $\DD(X)$, $\wedge^r\mbb L_{X/k}\cong \mc F_r^{\bullet\vee}[-r]$ for any $r\in\nan$.
\end{corollary}

\begin{proof}
The derived exterior product $\wedge^rK^\bullet$ has been described by T.\ Saito in \cite[\S4]{Takeshi:exterior-prod-2-terms} when $K^\bullet$ is a complex of locally free $\mc O_X$-modules of finite rank and $K^j=0$ for all $j\neq -1$, $0$ on any scheme $X$. In particular, if $K^{-1}$ is invertible, $\wedge^rK^\bullet$ is given by
\[
(K^{-1})^{\ot r}\To K^0\ot(K^{-1})^{\ot r-1}\To\cdots\To\wedge^{r-s}K^0 \ot(K^{-1})^{\ot s} \To\cdots\To \wedge^rK^0
\]
with the differentials $d_{\wedge^r}$ defined by $d_{\wedge^r}^{-s}=(\wedge^{r-s}\id_{K^0})\wedge d_K^{-1}\ot(\id^{\ot s-1}_{K^{-1}})$.

In our situation, $\mc I/\mc I^2\cong \mc O_X(-d)$ is invertible, and $\imath^*\Om_{\mbb P^n}$ is locally free of rank $n$. So the above form can be applied to $K^\bullet=\mbb L_{X/k}$. The $(-s)$-th term of $\wedge^r\mbb L_{X/k}$ is
\[
\wedge^{r-s}\imath^*\Om_{\mbb P^{n}}\ot(\mc I/\mc I^2)^{\ot s}\cong \imath^*\Om_{\mbb P^{n}}^{r-s}(-sd).
\]
Recall the exact sequence \eqref{eq:Omega-P}. It can be generalized to the long exact sequence
\[
0\To\Om_{\mbb P^n}^l\To \mc O(-l)^{\binom{n+1}{l}}\xrightarrow[\quad]{\pl_{\bov}} \mc O(-l+1)^{\binom{n+1}{l-1}} \xrightarrow[\quad]{\pl_{\bov}}\cdots\xrightarrow[\quad]{\pl_{\bov}}  \mc O(-1)^{n+1}\xrightarrow[\quad]{\pl_{\bov}} \mc O\To 0
\]
for any $l\in\nan$.\footnote{The proof is parallel to the one of \cite[Thm.\ 8.13]{Hartshorne:alg-geom}.} Just like the proof of Proposition \ref{prop:cotangent-complex}, we use the localization and then deduce
\[
0\To\imath^*\Om_{\mbb P^n}^l\To \mc O_X(-l)^{\binom{n+1}{l}}\xrightarrow[\quad]{\pl_{\bov}} \mc O_X(-l+1)^{\binom{n+1}{l-1}} \xrightarrow[\quad]{\pl_{\bov}}\cdots\xrightarrow[\quad]{\pl_{\bov}}  \mc O_X(-1)^{n+1}\xrightarrow[\quad]{\pl_{\bov}} \mc O_X\To 0
\]
is also exact.

These sequences constitute the diagram as follows,
\[
\xymatrix@C=6mm{
& & & & \mc O_X \\
& & & \mc O_X(-d) \ar[r]^-{\pl_{\bou}} & \mc O_X(-1)^{n+1} \ar[u]_-{\pl_{\bov}} \\
& & \iddots & \vdots \ar[u]_-{\pl_{\bov}} & \vdots \ar[u]_-{\pl_{\bov}} \\
& \mc O_X(d-rd) \ar[r]^-{\pl_{\bou}} & \cdots \ar[r]^-{\pl_{\bou}} & \mc O_X(2-r-d)^{\binom{n+1}{r-2}} \ar[r]^-{\pl_{\bou}}\ar[u]_-{\pl_{\bov}} & \mc O_X(-r+1)^{\binom{n+1}{r-1}} \ar[u]_-{\pl_{\bov}} \\
\mc O_X(-rd) \ar[r]^-{\pl_{\bou}} & \mc O_X(d-rd-1)^{n+1} \ar[r]^-{\pl_{\bou}}\ar[u]_-{\pl_{\bov}} & \cdots \ar[r]^-{\pl_{\bou}} & \mc O_X(1-r-d)^{\binom{n+1}{r-1}} \ar[r]^-{\pl_{\bou}}\ar[u]_-{\pl_{\bov}} & \mc O_X(-r)^{\binom{n+1}{r}} \ar[u]_-{\pl_{\bov}} \\
\mc O_X(-rd) \ar[r]^-{d_{\wedge^r}^{-r}}\ar[u] & \imath^*\Om_{\mbb P^n}(d-rd) \ar[r]^-{d_{\wedge^r}^{-r+1}}\ar[u] & \cdots \ar[r]^-{d_{\wedge^r}^{-2}} & \imath^*\Om_{\mbb P^n}^{r-1}(-d) \ar[r]^-{d_{\wedge^r}^{-1}}\ar[u] & \imath^*\Om_{\mbb P^n}^r \ar[u]
}
\]
where each column is exact, and the maps $\pl_{\bou}$ lift the differentials $d_{\wedge^r}$ since $d_{\wedge^r}$ are induced by the sequence $\bou$ (we adapt the Koszul sign rule here). Note that the associated total complex of the double complex by deleting the bottom row is exactly $\mc F_r^{\bullet\vee}[-r]$. Hence the diagram gives rise to a quasi-isomorphisms $\wedge^r\mbb L_{X/k}\to \mc F_r^{\bullet\vee}[-r]$.
\end{proof}

Before closing this section, let us compare Buchweitz and Flenner's formula \eqref{eq:generalized-HKR} ($Y=\mathrm{Spec}\,k$) and ours (i.e.\ $HH^i(X)\cong H^i(\mc H^\bullet)$) via the isomorphisms $\wedge^r\mbb L_{X/k}\to \mc F_r^{\bullet\vee}[-r]$.

Choose an injective resolution $0\to \mc O_X\to \mc I^\bullet$. Recall that $\mc{H}om(-, \mc I^\bullet)$, $-\ot \mc I^\bullet\colon \mathbf{C}(X)\to \mathbf{C}(X)$ are the total Hom functor and the total tensor functor respectively. Then $\mc{H}om(\mc F_q^{\bullet\vee}, \mc I^\bullet)\cong\mc F_q^{\bullet\vee\vee}\ot \mc I^\bullet\cong \mc F_q^{\bullet}\ot \mc I^\bullet$ is a complex of injective sheaves since $\mc F_q^{\bullet}$ is bounded and for each $i$, $\mc F_q^{i}$ is locally free of finite rank. There is a quasi-isomorphism $\mc F_q^{\bullet}\to \mc{H}om(\mc F_q^{\bullet\vee}, \mc I^\bullet)$. Thus
\begin{align*}
\Ext_X^p(\wedge^q_{\vphantom{X}}\mbb L_{X/k},\mc O_X)&=\Hom_{\DD(X)}(\wedge^q_{\vphantom{X}}\mbb L_{X/k},\mc O_X[p])\\
&\cong H^0(\Hom(\mc F_q^{\bullet\vee}[-q],\mc I^\bullet[p]))\\
&\cong H^{p+q}(\Gamma(X, \mc{H}om(\mc F_q^{\bullet\vee}, \mc I^\bullet)))\\
&= \mathbb{H}^{p+q}(X, \mc{H}om(\mc F_q^{\bullet\vee}, \mc I^\bullet))\\
&\cong\mathbb{H}^{p+q}(X, \mc F_q^{\bullet})
\end{align*}
where the hypercohomology $\mathbb{H}^{p+q}(X, \mc F_q^{\bullet})$ can also be computed by the (total) \v{C}ech complex (see e.g.\ \cite[Ch.\ 1]{Brylinski:loop-space}), namely, $\mathbb{H}^{p+q}(X, \mc F_q^{\bullet})\cong H^{p+q}(\mc H_q^\bullet)$. So
\[
\bigoplus_{p+q=i}\Ext_X^p(\wedge^q_{\vphantom{X}}\mbb L_{X/k},\mc O_X)\cong \bigoplus_{p+q=i}H^{p+q}(\mc H_q^\bullet)=H^i(\mc H^\bullet).
\]
Both results agree.

\subsection{Proof of the main theorem}\label{subsec:compute-cohomology}
Let us associate some graded modules to $X=\Proj S$. Note that the $\pl_{\bov}$ constitute a morphism
\[
\xymatrix{
\cdots \ar[r] & \mc K^{-3}(\bou;S) \ar[r]^-{\pl_{\bou}} & \mc K^{-2}(\bou;S) \ar[r]^-{\pl_{\bou}} & \mc K^{-1}(\bou;S) \ar[r]^-{\pl_{\bou}} & \mc K^{0}(\bou;S) \\
\cdots \ar[r] & \mc K^{-2}(\bou;S) \ar[r]^-{\pl_{\bou}}\ar[u]_-{\pl_{\bov}} & \mc K^{-1}(\bou;S) \ar[r]^-{\pl_{\bou}}\ar[u]_-{\pl_{\bov}} & \mc K^{0}(\bou;S) \ar[r] \ar[u]_-{\pl_{\bov}} & 0 \ar[u]
}
\]
from which we obtain the cokernel complex $\mc C^\bullet(\bou;S)$:
\begin{equation}\label{eq:quotient-complex}
\cdots\To \mc K^{-3}(\bou;S)/\im\pl_{\bov} \xrightarrow[\quad]{\pl_{\bou}} \mc K^{-2}(\bou;S)/\im\pl_{\bov} \xrightarrow[\quad]{\pl_{\bou}} \mc K^{-1}(\bou;S)/\im\pl_{\bov} \xrightarrow[\quad]{\pl_{\bou}} \mc K^{0}(\bou;S).
\end{equation}
The $i$-th cohomology group of $\mc C^\bullet(\bou;S)$ is denoted by $P^i$ and the $i$-th cocycle group by $Q^i$. Clearly, the $S$-modules $P^i$, $Q^i$ are graded modules. Denote by $Z^i$ the $i$-th cocycle group of $\mc K^\bullet(\bov;R)$, which is a graded $R$-module.

Recall that we have quasi-isomorphisms $\mc H^{\bullet}\to\mc G^{\bullet}\to\mc E^{\bullet}\to \bar{\mbf C}^{\prime \bullet}_{\GS}(\mc A)$. From now on, let us compute $H^\bullet_{\GS}(\mc A):=H^\bullet\bar{\mbf C}^{\prime\bullet}_{\GS}(\mc A)$ by using $\mc H^{\bullet,\bullet}$. We need some lemmas.

\begin{lemma}\label{lem:cohomo-1}
  The cohomology groups of $\mc K^\bullet(\bov^\star;S)$ are $H^0=H^0_0= k$, $H^{-1}=H^{-1}_{d-1}= k\bou^\star$ where $\bou^\star=(\pl F/\pl x_0, -\pl F/\pl x_1,\ldots, (-1)^n\pl F/\pl x_n)$, and $H^i=0$ for all $i\neq 0$, $-1$.
\end{lemma}

\begin{proof}
  Recall Remark \ref{rmk:mixed-v} and note that $\bov^\star_0=(-x_1,x_2,\ldots,(-1)^nx_n)$ is a regular sequence in $S$. So $\mc K^\bullet(\bov^\star;S)$, which is the mapping cone of $\mc K^\bullet(\bov^\star_0;S)\xrightarrow{x_0}\mc K^\bullet(\bov^\star_0;S)$, is quasi-isomorphic to
  \[
  \cdots\To 0\To S/(\bov^\star_0)\xrightarrow[\quad]{x_0}S/(\bov^\star_0) \To 0.
  \]
  Since $S/(\bov^\star_0)\cong k[x_0]/(x_0^d)$, we have $H^0=k$ and $H^{-1}\cong k(1-d)$ as graded modules. To show $\bou^\star$ is a base element in $H^{-1}$, we will check that $\bou^\star$ never belongs to $\im\pl_{\bov^\star}$. This is clear since $\pl F/\pl x_0$ contains $dx^{d-1}_0$ as a summand.
\end{proof}

The following is well known:
\begin{lemma}\label{lem:cohomo-2}
  The cohomology groups of $\mc K^\bullet(\bov; R)$ are $H^0=H^0_0= k$, and $H^i=0$ for all $i\neq 0$.
\end{lemma}

\begin{lemma}\label{lem:cohomo-3}
  Let $\tau^{r}_0$ be the zeroth graded component of $\tau^{r}\mc K^{\bullet,\bullet}(\bou,\bov;S)$.
  \begin{enumerate}
    \item If $0\leq r\leq n$, then
     \[
     H^i(\Tot\tau^{r}_0)\cong
      \begin{cases}
        0, & 0\leq i< r,\\
        Q^{-r}_{r}, &  i=r,\\
        P^{i-2r}_{r+(i-r)(d-1)}, & r< i\leq 2r.
      \end{cases}
     \]
    \item If $r\geq n+1$ and $d=n+1$, then
     \[
     H^i(\Tot\tau^{r}_0)\cong
      \begin{cases}
        0, & 0\leq i\leq r,\, i\neq n,\\
        k, & i=n,\\
        P^{i-2r}_{r+(i-r)(d-1)}, & r< i\leq 2r.
      \end{cases}
      \]
     \item If $r\geq n+1$ and $d\neq n+1$, then
     \[
      H^i(\Tot\tau^{r}_0)\cong
           \begin{cases}
             0, & 0\leq i\leq r,\\
             P^{i-2r}_{r+(i-r)(d-1)}, & r< i\leq 2r.
           \end{cases}
      \]
  \end{enumerate}
\end{lemma}

\begin{proof}
  We prove the statements by computing the spectral sequence ${}^IE^{p,q}_a$ determined by $\tau^{r}_0$.

  (1) Let $0\leq r\leq n$. The $p$-th column of $\tau^{r}\mc K^{\bullet,\bullet}(\bou,\bov;S)$ is the truncation  $\tau^{\leq -(n+1-r+p)}\mc K^\bullet(\bov^\star;S)$ up to twist. Notice that $-(n+1-r+p)\leq -1$. By Lemma \ref{lem:cohomo-1}, $H^i(\tau^{\leq -(n+1-r+p)}\mc K^\bullet(\bov^\star;S))=0$ if $i\neq -(n+1-r+p)$. It follows that the $p$-th column of $\tau^{r}\mc K^{\bullet,\bullet}(\bou,\bov;S)$ is exact except in spot $(p,r)$. By considering the zeroth graded component, we have ${}^IE^{p,q}_1=0$ if $q\neq r$, and
  \[
  {}^IE^{p,r}_1=\bigl(S(r+p(d-1))^{\binom{n+1}{r-p}}/\im\pl_{\bov}\bigr)_0 =\bigl(\mc K^{-(r-p)}(\bou;S)/\im\pl_{\bov}\bigr)_{r+p(d-1)}.
  \]
  To compute ${}^IE^{p,r}_2$, it suffices to consider the complex
  \[
  \bigl(\mc K^{-r}(\bou;S)/\im\pl_{\bov}\bigr)_{r} \To \cdots \To \bigl(\mc K^{-1}(\bou;S)/\im\pl_{\bov}\bigr)_{r+(r-1)(d-1)}\xrightarrow[\quad]{\pl_{\bou}} \bigl(\mc K^{0}(\bou;S)\bigr)_{rd}.
  \]
  Comparing this complex with \eqref{eq:quotient-complex}, we have ${}^IE^{0,r}_2=Q^{-r}_r$, ${}^IE^{p,r}_2=P^{-(r-p)}_{r+p(d-1)}$ if $p\geq 1$. Hence $H^i(\Tot\tau^r_0)\cong{}^IE^{i-r,r}_{\infty}={}^IE^{i-r,r}_2=P^{i-2r}_{r+(i-r)(d-1)}$ when $r< i\leq 2r$, and $H^r(\Tot\tau^r_0)\cong{}^IE^{0,r}_{\infty}={}^IE^{0,r}_2=Q^{-r}_{r}$.

  (2) Let $r\geq n+1$ and $d=n+1$. Just like in the situation in (1), we have
  \[
    {}^IE^{p,r}_1=\bigl(\mc K^{-(r-p)}(\bou;S)/\im\pl_{\bov}\bigr)_{r+p(d-1)}.
  \]
  By Lemma \ref{lem:cohomo-1} and taking into consideration the degrees, we find one more nonzero ${}^IE^{p,q}_1$, namely, ${}^IE^{0,n}_1\cong k$. For ${}^IE^{p,q}_2$, as shown in (1), ${}^IE^{0,r}_2=Q_r^{-r}$ and ${}^IE^{p,r}_2=P^{-(r-p)}_{r+p(d-1)}$ for all $1\leq p\leq 2r$. Note that $Q^{-(n+1)}_{n+1}$ is a $k$-submodule of $(S/\im\pl_{\bov})_{n+1}=k_{n+1}=0$ and $Q^{-s}=0$ if $s\geq n+2$. Hence ${}^IE^{0,r}_2=Q_r^{-r}=0$ since $r\geq n+1$. We also have ${}^IE^{0,n}_2={}^IE^{0,n}_1\cong k$ and the rest ${}^IE^{p,q}_2$ being all zero. Assertion (2) follows.

  (3) The proof is completely similar to (2). The only difference is that ${}^IE^{0,n}_1$ is zero since $d\neq n+1$.
\end{proof}

The double complex $\mc H_r^{\bullet,\bullet}$ leads to a spectral sequence ${}^{I\!I}E^{p,q}_{r,a}$ by filtration by rows. We begin to calculate it.

\subsubsection{Case 1: $d>n+1$}

Suppose $m\geq 0$ and $\mc O=\mc O_{\mbb{P}^n}$. By the exact sequence
\[
0\To\mc O(m-d)\xrightarrow[\quad]{F}\mc O(m)\To\mc O_X(m)\To 0,
\]
we immediately conclude that $H^i(X,\mc O_X(m))=0$ if $i\neq 0$, $n-1$, and $H^{n-1}(X,\mc O_X(m))\cong H^{n}(\mbb{P}^n,\mc O(m-d))\cong H^0(\mbb{P}^n, \mc O(d-n-1-m))^*$. Obviously, $H^0(\mbb{P}^n, \mc O(d-n-1-m))$ has a basis
\[
\{x_0^{i_0}x_1^{i_1}\cdots x_n^{i_n}\in R \mid i_0+i_1+\cdots+ i_n=d-n-1-m,\, i_0,i_1,\ldots,i_n\geq 0\}.
\]
On the other hand, the \v{C}ech cohomology group $\check{H}^{n-1}(\mf U,\mc O_X(m))$ has a basis
\[
\{x_0^{j_0}x_1^{j_1}\cdots x_n^{j_n}\in S_{x_1\cdots x_n} \mid j_0+j_1+\cdots+ j_n=m,\, 0\leq j_0\leq d-1,\, j_1,\ldots,j_n\leq -1\}
\]
where $S_{x_1\cdots x_n}$ is the localization of $S$ at $x_1\cdots x_n$. Since both groups have finite dimension over $k$, the duality gives rise to the bijection
\begin{align}
\scr S\colon H^0(\mbb{P}^n, \mc O(d-n-1-m))&\To \check{H}^{n-1}(\mf U,\mc O_X(m)),\label{eq:scriptS}\\
x_0^{i_0}x_1^{i_1}\cdots x_n^{i_n}&\longmapsto x_0^{d-1-i_0}x_1^{-1-i_1}\cdots x_n^{-1-i_n}.\notag
\end{align}
The map $\scr S$ induces $H^0(\mbb{P}^n, \mc O(d-n-1-m)^r)\To \check{H}^{n-1}(\mf U,\mc O_X(m)^r)$ for any $r\in\nan$ which is also denoted by $\scr S$.

Since $\check{H}^{n-1}(\mf U,\mc O_X(m))=0$ if $m\geq d$, by the definition of $\mc H_r^{\bullet,\bullet}$, we have
\[
{}^{I\!I}E^{p,q}_{r,1}=\check{H}^q(\mf U,\mc F_r^p)=
\begin{cases}
\check{H}^{n-1}\bigl(\mf U,\mc O_X(p)^{\binom{n+1}{p}}\bigr), & 0\leq p\leq r,\, q=n-1,\\
(\Tot\tau_0^r)^p, & 0\leq p\leq 2r,\, q=0,\\
0, & \text{otherwise.}
\end{cases}
\]
Since $\check{H}^{n-1}\bigl(\mf U,\mc O_X(p)^{\binom{n+1}{p}}\bigr)= \bigl(R_{d-n-1-p}^{\binom{n+1}{p}}\bigr)^*=\mc K^{-p}(\bov;R)_{d-n-1-p}^*$, the complex
\[
{}^{I\!I}E^{0,n-1}_{r,1}\To {}^{I\!I}E^{1,n-1}_{r,1}\To \cdots \To {}^{I\!I}E^{r-1,n-1}_{r,1}\To {}^{I\!I}E^{r,n-1}_{r,1}
\]
is dual to
\begin{equation}\label{eq:complex-R}
\mc K^{0}(\bov;R)_{d-n-1}\longleftarrow\mc K^{-1}(\bov;R)_{d-n-2} \longleftarrow\cdots\longleftarrow \mc K^{-r+1}(\bov;R)_{d-n-r}\longleftarrow\mc K^{-r}(\bov;R)_{d-n-1-r}.
\end{equation}
By Lemma \ref{lem:cohomo-2}, the only non trivial cohomology of the complex $\mc K^\bullet(\bov;R)$ is $H^0(\mc K^\bullet(\bov;R))= k$. The zero-th cohomology group of \eqref{eq:complex-R} is zero since the $(d-n-1)$-st graded component in $k$ is zero.  The unique possible nonzero cohomology of \eqref{eq:complex-R} is $H^{-r}=Z^{-r}_{d-n-1-r}$, yielding ${}^{I\!I}E^{r,n-1}_{r,2}=\scr S(Z^{-r}_{d-n-1-r})$. Combining this with Lemma \ref{lem:cohomo-3}, we obtain that ${}^{I\!I}E^{p,q}_{r,2}$ is given by
\[
{}^{I\!I}E^{p,q}_{r,2}=
\begin{cases}
\scr S(Z^{-r}_{d-n-1-r}), & p= r,\, q=n-1,\\
P^{p-2r}_{r+(p-r)(d-1)}, & r< p\leq 2r,\, q=0,\\
Q^{-r}_{r}, & p= r,\, q=0,\\
0, & \text{otherwise.}
\end{cases}
\]
Immediately, ${}^{I\!I}E^{p,q}_{r,n}=\cdots ={}^{I\!I}E^{p,q}_{r,2}$. On one hand, ${}^{I\!I}E^{r+n,0}_{r,n}=0$ when $r\leq n-1$, since $r+n>2r$; on the other hand, in case $r\geq n$, we have $d-n-1-r\leq -1$ and so ${}^{I\!I}E^{r,n-1}_{r,n}=0$ since $R$ has only non-negative grading.  So in order to show ${}^{I\!I}E^{p,q}_{r,n+1}={}^{I\!I}E^{p,q}_{r,n}$ for any pair $(r,n)$, it is sufficient to prove the differential ${}^{I\!I}E^{n,n-1}_{n,n}\to {}^{I\!I}E^{2n,0}_{n,n}$ (i.e.\ the case $r=n$) is zero.

Since ${}^{I\!I}E^{n,n-1}_{n,n}$ is a sub-quotient of $\check{\mbf C}^{\prime n-1}(\mf U, \mc F_n^n)$, we choose a cocycle $c^{n-1,n}\in \check{\mbf C}^{\prime n-1}(\mf U, \mc F_n^n)$ for any class in ${}^{I\!I}E^{n,n-1}_{n,n}$. Performing a diagram chase, a cochain $(c^{0,2n-1},c^{1,2n-2},\ldots, c^{n-1,n})$ in $ \mc H^{\bullet}$ can be given. Notice that $c^{0,2n-1}\in \mc H^{0,2n-1}=\check{\mbf C}^{\prime 0}(\mf U, \mc F^{2n-1}_n)=\mc K^{-1}(\bou;S)_{nd-d+1}$, and so $d_{v,\mc H}(c^{0,2n-1})=\pl_{\bou}(c^{0,2n-1})$ is a coboudary in $\mc K^{0}(\bou;S)_{nd}$, i.e.\ $d_{v,\mc H}(c^{0,2n-1})$ represents the zero class in $P^0_{nd}={}^{I\!I}E^{2n,0}_{n,n}$. It follows that the differential ${}^{I\!I}E^{n,n-1}_{n,n}\to {}^{I\!I}E^{2n,0}_{n,n}$ is a zero map.  Therefore, ${}^{I\!I}E^{p,q}_{r,\infty}={}^{I\!I}E^{p,q}_{r,2}$, and
\[
H^i(\mc H^{\bullet})\cong\bigoplus_{r\in\nan}\bigoplus_{p+q=i}{}^{I\!I}E^{p,q}_{r,\infty}
= \bigoplus_{r<i}P^{i-2r}_{r+(i-r)(d-1)}\oplus Q^{-i}_{i}\oplus \scr S(Z^{-i+n-1}_{d-i-2}).
\]

\subsubsection{Case 2: $d=n+1$}
The formula
\[
{}^{I\!I}E^{p,q}_{r,1}=
\begin{cases}
\check{H}^{n-1}\bigl(\mf U,\mc O_X(p)^{\binom{n+1}{p}}\bigr), & 0\leq p\leq r,\, q=n-1,\\
(\Tot\tau_0^r)^p, & 0\leq p\leq 2r,\, q=0,\\
0, & \text{otherwise.}
\end{cases}
\]
remains valid in this case. Note that the complex \eqref{eq:complex-R} has only one nonzero term $\mc K^0(\bov;R)_{d-n-1}=R_0= k$. By applying Lemma \ref{lem:cohomo-3} again, we conclude that for $0\leq r\leq n$,
\[
{}^{I\!I}E^{p,q}_{r,2}=
\begin{cases}
k, & p= 0,\, q=n-1,\\
P^{p-2r}_{r+n(p-r)}, & r< p\leq 2r,\, q=0,\\
Q^{-r}_{r}, & p= r,\, q=0,\\
0, & \text{otherwise,}
\end{cases}
\]
and for $r\geq n+1$,
\[
{}^{I\!I}E^{p,q}_{r,2}=
\begin{cases}
k, & p= 0,\, q=n-1,\\
k, & p= n,\, q=0,\\
P^{p-2r}_{r+n(p-r)}, & r< p\leq 2r,\, q=0,\\
0. & \text{otherwise.}
\end{cases}
\]
It follows that ${}^{I\!I}E^{p,q}_{r,n}=\cdots ={}^{I\!I}E^{p,q}_{r,2}$.

Since for any $V_{i_1\ldots i_s}\in \mf V$, the algebra $A_{i_1\ldots i_s}=\mc O_X(V_{i_1\ldots i_s})$ is identified with the zero-th graded component of $S_{x_{i_1}\cdots x_{i_s}}$, the localization of $S$ with respect to the element $x_{i_1}\cdots x_{i_s}$, we conclude that the \v{C}ech complex $\check{\mbf C}^{\prime\bullet}(\mf U, \mc F_r^0)$ for any $r$ is the sub-complex of
\[
\prod_{i_1} S_{x_{i_1}}\To \prod_{i_1<i_2} S_{x_{i_1}x_{i_2}}\To\cdots\To \prod_{i_1<\cdots<i_{n-1}} S_{x_{i_1}\cdots x_{i_{n-1}}} \To S_{x_{1}\cdots x_{n}}
\]
consisting of all cochains of degree zero. Since ${}^{I\!I}E^{0,n-1}_{r,n}$ is a sub-quotient of $\check{\mbf C}^{\prime n-1}(\mf U, \mc F_r^0)$, it seems apt to choose $x_0^nx_1^{-1}\cdots x_n^{-1}\in \check{\mbf C}^{\prime n-1}(\mf U, \mc F_r^0)$ as a base element of ${}^{I\!I}E^{0,n-1}_{r,n}$. However, for the sake of easy computation, we use $x_1^{-1}\cdots x_n^{-1}\cdot\pl F/\pl x_0$ flexibly rather than $x_0^nx_1^{-1}\cdots x_n^{-1}$. Similar to the argument in the case $d> n+1$, one finds a cochain $(c^{0,n-1},c^{1,n-2},\ldots, c^{n-1,0})$ in $\mc H^{\bullet}$ with $c^{n-1,0}=x_1^{-1}\cdots x_n^{-1}\cdot\pl F/\pl x_0$. The differential ${}^{I\!I}E^{0,n}_{r,n-1}\to {}^{I\!I}E^{n,0}_{r,n}$ sends the class represented by $c^{n-1,0}$ to the one represented by $d_{v,\mc H}(c^{0,n-1})$.
\begin{enumerate}
\item If $0\leq r\leq n-1$, $d_{v,\mc H}(c^{0,n-1})$ belongs to $P^{n-2r}_{r+n(n-r)}$. Recall the shape and size of the triangle $\tau^r\mc K^{\bullet,\bullet}(\bou,\bov;S)$. The element $d_{v,\mc H}(c^{0,n-1})$ is zero itself if $r$ is very small, or is a sum $\pl_{\bou}(?)+\pl_{\bov}(?')$ if $r$ is larger. According to the construction of \eqref{eq:complex-R}, $\pl_{\bou}(?)+\pl_{\bov}(?')$ necessarily represents the zero class. In both cases, $c^{n-1,0}$ is killed by the differential ${}^{I\!I}E^{0,n}_{r,n-1}\to {}^{I\!I}E^{n,0}_{r,n}$.
\item If $r=n$, the diagram chase shows $d_{v,\mc H}(c^{0,n-1})=\bou^\star+\im\pl_{\bov}\in Q^{-n}_{n}=\ker\{S^{n+1}_{n}/\im\pl_{\bov}\to S^{n(n+1)/2}_{2n}\im\pl_{\bov}\}$. By the definition of $\mc C^\bullet(\bou;S)$, $\bou^\star+\im\pl_{\bov}$ happens to be a base element of $\im\{S_0/\im\pl_{\bov}\to S^{n+1}_{n}/\im\pl_{\bov}\}$. So ${}^{I\!I}E^{0,n}_{r,n-1}=k\to {}^{I\!I}E^{n,0}_{r,n}=Q^{-n}_{n}$ is injective and its cokernel is given by $Q_n^{-n}/(k\bou^\star+\im\pl_{\bov})=P_n^{-n}$.
\item If $r\geq n+1$, we claim that the differential ${}^{I\!I}E^{0,n}_{r,n-1}=k\to {}^{I\!I}E^{n,0}_{r,n}=k$ is an isomorphism. The assertion follows from Lemma \ref{lem:k-to-k} which will be proven later on.
\end{enumerate}

Summarizing, the spectral sequence
\begin{align*}
{}^{I\!I}E^{p,q}_{r,\infty}&={}^{I\!I}E^{p,q}_{r,n+1}=
\begin{cases}
k, & p= 0,\, q=n-1,\\
P^{p-2r}_{r+n(p-r)}, & r< p\leq 2r,\, q=0,\\
Q^{-r}_{r}, & p= r,\, q=0,\\
0, & \text{otherwise,}
\end{cases}
&&\text{if $0\leq r\leq n-1$,}\\
{}^{I\!I}E^{p,q}_{r,\infty}&={}^{I\!I}E^{p,q}_{r,n+1}=
\begin{cases}
P^{p-2r}_{r+n(p-r)}, & r\leq p\leq 2r,\, q=0,\\
0, & \text{otherwise,}
\end{cases}
&&\text{if $r=n$,}\\
{}^{I\!I}E^{p,q}_{r,\infty}&={}^{I\!I}E^{p,q}_{r,n+1}=
\begin{cases}
P^{p-2r}_{r+n(p-r)}, & r< p\leq 2r,\, q=0,\\
0, & \text{otherwise,}
\end{cases}
&&\text{if $r\geq n+1$.}
\end{align*}
Therefore,
\[
H^i(\mc H^{\bullet})\cong
\begin{cases}
\displaystyle \bigoplus_{r<i}P^{i-2r}_{r+n(i-r)}\oplus Q^{-i}_{i}, & i\neq n-1, n,\\
\displaystyle \bigoplus_{r<i}P^{i-2r}_{r+n(i-r)}\oplus Q^{-i}_{i}\oplus k^n, & i=n-1,\\
\displaystyle \bigoplus_{r\leq i}P^{i-2r}_{r+n(i-r)}, & i= n.
\end{cases}
\]

Note that $\mc F_r^q$ is a direct sum of some terms as given in Figure \ref{fig:truncation-S}, and hence $\mc H_r^{p,q}$ admits a decomposition
\[
\check{\mbf C}^{\prime p}(\mf U, \mc O_X(q)^{\binom{n+1}{q}})\oplus \check{\mbf C}^{\prime p}(\mf U, \mc O_X(q+d-2)^{\binom{n+1}{q-2}})\oplus \check{\mbf C}^{\prime p}(\mf U, \mc O_X(q+2d-4)^{\binom{n+1}{q-4}})\oplus \cdots
\]
when $q\leq r$. Intuitively, $\mc O_X(q)^{\binom{n+1}{q}}$ appearing in the first component corresponds to a graded module located at the leftmost edge in Figure \ref{fig:truncation-S}. We hence call a cochain in $\mc H_r^{p,q}$ \emph{left preferred} if it has possible nonzero component only in $\check{\mbf C}^{\prime p}(\mf U, \mc O_X(q)^{\binom{n+1}{q}})$.

\begin{lemma}\label{lem:k-to-k}
  Suppose $d=n+1$ and $r\geq n$. There exists a cochain $(c^{0,n-1},c^{1,n-2},\ldots,c^{n-1,0})$ in $\mc H^{n-1}_r$ such that each $c^{n-1-q,q}$ is left preferred in $\mc H_r^{n-1-q,q}$ and
  \[
  c^{n-1,0}=x_1^{-1}\cdots x_n^{-1}\frac{\pl F}{\pl x_0},\quad d_{\mc F}(c^{0,n-1},c^{1,n-2},\ldots,c^{n-1,0})=((-1)^{n-1}\bou^\star,0,\ldots,0).
  \]
\end{lemma}

\begin{proof}
  During the proof, we will frequently meet elements in $S_{x_{i_1}\cdots x_{i_m}}$. To avoid confusion, we underline denominators to distinguish between similar looking elements. For example, $\underline{x_1^{-1}}\in S_{x_1}$, $\underline{x_1^{-1}x_2^0}\in S_{x_1x_2}$, $\underline{x_1^{-1}x_2^0x_3^0}\in S_{x_1x_2x_3}$. The notations $\mf f_{j_1\ldots j_s}$ stand for formal bases elements.

  When the \v{C}ech indices $(i_1,\ldots, i_s)$ appear, the complements are denote by $(j_1,\ldots,j_{n-s})$, namely, the latter are obtained by deleting $i_1,\ldots, i_s$ from $(1,2,\ldots, n)$. The permutation
  \[
  \begin{pmatrix}
    1 & \ldots & s & s+1 & \ldots & n \\
    i_1 & \ldots & i_s & j_1 & \ldots & j_{n-s}
  \end{pmatrix}
  \]
  is a shuffle, whose parity $(n^2-s^2+n-s)/2-(j_1+\cdots+ j_{n-s})$ is denoted by $\wp(i_1,\ldots,i_s)$ or even by $\wp(\imath)$ if no confusion arises.

  Starting with $c^{n-1,0}=\underline{x_1^{-1}\cdots x_n^{-1}}\pl F/\pl x_0$, we have
  \begin{align*}
    d_{\mc F}(c^{n-1,0})&=(-1)^{n-1}\underline{x_1^{-1}\cdots x_n^{-1}}x_0\frac{\pl F}{\pl x_0}\mf f_0+(-1)^{n-1}\sum_{j=1}^n\underline{x_1^{-1}\cdots x_j^0\cdots x_n^{-1}}\frac{\pl F}{\pl x_0}\mf f_j\\
    &=(-1)^{n-1}\sum_{j=1}^n\underline{x_1^{-1}\cdots x_j^0\cdots x_n^{-1}}\biggl(\frac{\pl F}{\pl x_0}\mf f_j-\frac{\pl F}{\pl x_j}\mf f_0\biggr).
  \end{align*}
  Choose $c^{n-2,1}=(c^{n-2,1}_{i_1,\ldots,i_{n-1}})$ as
  \[
  c^{n-2,1}_{i_1,\ldots,i_{n-1}}=(-1)^{\wp(\imath)+1}\underline{x_{i_1}^{-1}\cdots x_{i_{n-1}}^{-1}}\biggl(\frac{\pl F}{\pl x_0}\mf f_{j_1}-\frac{\pl F}{\pl x_{j_1}}\mf f_0\biggr).
  \]
  One can easily show that $\pl_{\bou}(c^{n-2,1})=0$. Thus $d_{\mc F}(c^{n-2,1})=(-1)^{n-2}\pl_{\bov}(c^{n-2,1})$ whose components are
  \begin{align*}
    d_{\mc F}(c^{n-2,1}_{i_1,\ldots,i_{n-1}})&=(-1)^{j_1+1}\underline{x_{i_1}^{-1}\cdots x_{i_{n-1}}^{-1}}x_0\frac{\pl F}{\pl x_0}\mf f_{0j_1}+(-1)^{j_1+1}\sum_{l=1}^{n-1}\underline{x_{i_1}^{-1}\cdots x_{i_l}^0\cdots x_{i_{n-1}}^{-1}}\frac{\pl F}{\pl x_0}\mf f_{i_lj_1}\\
    &\varphantom{=}{}+(-1)^{j_1+1}\sum_{l=1}^{n-1}\underline{x_{i_1}^{-1}\cdots x_{i_l}^0\cdots x_{i_{n-1}}^{-1}}\frac{\pl F}{\pl x_{j_1}}\mf f_{0i_l}+(-1)^{j_1+1}\underline{x_{i_1}^{-1}\cdots x_{i_{n-1}}^{-1}}x_{j_1}\frac{\pl F}{\pl x_{j_1}}\mf f_{0j_1}\\
    &=(-1)^{j_1+1}\sum_{l=1}^{n-1}\underline{x_{i_1}^{-1}\cdots x_{i_l}^0\cdots x_{i_{n-1}}^{-1}}\biggl(\frac{\pl F}{\pl x_0}\mf f_{i_lj_1}+\frac{\pl F}{\pl x_{j_1}}\mf f_{0i_l}-\frac{\pl F}{\pl x_{i_l}}\mf f_{0j_1}\biggr).
  \end{align*}
  Choose $c^{n-3,2}=(c^{n-3,2}_{i_1,\ldots,i_{n-2}})$ as
  \[
  c^{n-3,2}_{i_1,\ldots,i_{n-2}}=(-1)^{\wp(\imath)}\underline{x_{i_1}^{-1}\cdots  x_{i_{n-2}}^{-1}}\biggl(\frac{\pl F}{\pl x_0}\mf f_{j_1j_2}-\frac{\pl F}{\pl x_{j_1}}\mf f_{0j_2}+\frac{\pl F}{\pl x_{j_2}}\mf f_{0j_1}\biggr)
  \]
  which is again in $\ker\pl_{\bou}$. Thus $d_{\mc F}(c^{n-3,2})=(-1)^{n-3}\pl_{\bov}(c^{n-3,2})$ whose components are
  \begin{align*}
    d_{\mc F}(c^{n-3,2}_{i_1,\ldots,i_{n-2}})&=(-1)^{n-j_1-j_2}\biggl(\underline{x_{i_1}^{-1}\cdots x_{i_{n-2}}^{-1}}x_0\frac{\pl F}{\pl x_0}\mf f_{0j_1j_2}+\sum_{l=1}^{n-2}\underline{x_{i_1}^{-1}\cdots x_{i_l}^0\cdots x_{i_{n-2}}^{-1}}\frac{\pl F}{\pl x_0}\mf f_{i_lj_1j_2}\\
    &\varphantom{=}{}+\sum_{l=1}^{n-2}\underline{x_{i_1}^{-1}\cdots x_{i_l}^0\cdots x_{i_{n-2}}^{-1}}\frac{\pl F}{\pl x_{j_1}}\mf f_{0i_lj_2}+\underline{x_{i_1}^{-1}\cdots x_{i_{n-2}}^{-1}}x_{j_1}\frac{\pl F}{\pl x_{j_1}}\mf f_{0j_1j_2}\\
    &\varphantom{=}{}+\sum_{l=1}^{n-2}\underline{x_{i_1}^{-1}\cdots x_{i_l}^0\cdots x_{i_{n-2}}^{-1}}\frac{\pl F}{\pl x_{j_2}}\mf f_{0j_1i_l}+\underline{x_{i_1}^{-1}\cdots x_{i_{n-2}}^{-1}}x_{j_2}\frac{\pl F}{\pl x_{j_2}}\mf f_{0j_1j_2}\\
    &=(-1)^{n-j_1-j_2}\sum_{l=1}^{n-2}\underline{x_{i_1}^{-1}\cdots x_{i_l}^0\cdots x_{i_{n-2}}^{-1}}\biggl(\frac{\pl F}{\pl x_0}\mf f_{i_lj_1j_2}+\frac{\pl F}{\pl x_{j_1}}\mf f_{0i_lj_2}\\
    &\varphantom{=}{}+\frac{\pl F}{\pl x_{j_2}}\mf f_{0j_1i_l}-\frac{\pl F}{\pl x_{i_l}}\mf f_{0j_1j_2}\biggr).
  \end{align*}
  Choose $c^{n-4,3}=(c^{n-4,3}_{i_1,\ldots,i_{n-3}})$ as
  \[
  c^{n-4,3}_{i_1,\ldots,i_{n-3}}=(-1)^{\wp(\imath)+1}\underline{x_{i_1}^{-1}\cdots  x_{i_{n-3}}^{-1}}\biggl(\frac{\pl F}{\pl x_0}\mf f_{j_1j_2 j_3}-\frac{\pl F}{\pl x_{j_1}}\mf f_{0j_2j_3}+\frac{\pl F}{\pl x_{j_2}}\mf f_{0j_1j_3}-\frac{\pl F}{\pl x_{j_3}}\mf f_{0j_1j_2}\biggr).
  \]

  Set $j_0=0$ by convention and continue the above procedure. We obtain
  \begin{equation}\label{eq:c}
  c^{s-1,n-s}_{i_1,\ldots,i_{s}}=(-1)^{\wp(\imath)+n-s}\underline{x_{i_1}^{-1}\cdots  x_{i_{s}}^{-1}}\sum_{m=0}^{n-s}(-1)^m\frac{\pl F}{\pl x_{j_m}}\mf f_{j_0\ldots \widehat{j_m}\ldots j_{n-s}}
  \end{equation}
  successively, which is obviously left preferred. In particular, when $s=1$,
  \[
  c^{0,n-1}_{i_1}=(-1)^{n-i_1}\underline{x_{i_1}^{-1}}\sum_{m=0}^{n-1}(-1)^m\frac{\pl F}{\pl x_{j_m}}\mf f_{j_0\ldots \widehat{j_m}\ldots j_{n-1}}
  \]
  and hence
  \begin{align*}
  d_{\mc F}(c^{0,n-1}_{i_1})&=(-1)^{n-i_1}\biggl(\underline{x_{i_1}^{0}}\sum_{m=0}^{n-1}(-1)^m\frac{\pl F}{\pl x_{j_m}}\mf f_{i_1j_0\ldots \widehat{j_m}\ldots j_{n-1}}+\underline{x_{i_1}^{-1}}\sum_{m=0}^{n-1}x_{j_m}\frac{\pl F}{\pl x_{j_m}}\mf f_{j_0\ldots  j_{n-1}}\biggr)\\
  &=(-1)^{n-i_1}\biggl(\underline{x_{i_1}^{0}}\sum_{m=0}^{n-1}(-1)^m\frac{\pl F}{\pl x_{j_m}}\mf f_{i_1j_0\ldots \widehat{j_m}\ldots j_{n-1}}-\underline{x_{i_1}^{0}}\frac{\pl F}{\pl x_{i_1}}\mf f_{j_0\ldots  j_{n-1}}\biggr)\\
  &=(-1)^{n-1}\underline{x_{i_1}^{0}}\biggl(\sum_{j_m<i_1}(-1)^m\frac{\pl F}{\pl x_{j_m}}\mf f_{j_0\ldots \widehat{j_m}\ldots i_1\ldots j_{n-1}}+\sum_{j_m>i_1}(-1)^{m+1}\frac{\pl F}{\pl x_{j_m}}\mf f_{j_0\ldots i_1\ldots \widehat{j_m}\ldots j_{n-1}}\\
  &\varphantom{=}{}+(-1)^{i_1}\underline{x_{i_1}^{0}}\frac{\pl F}{\pl x_{i_1}}\mf f_{j_0\ldots \widehat{i_1}\ldots j_{n-1}}\biggr)\\
  &=(-1)^{n-1}\underline{x_{i_1}^{0}}\sum_{m=0}^{n}(-1)^m \frac{\pl F}{\pl x_{j_m}}\mf f_{j_0\ldots \widehat{j_m}\ldots j_{n}}.
  \end{align*}
  So $d_{\mc F}(c^{0,n-1}_{i_1})$ is actually the restriction of the global section $(-1)^{n-1}\bou^\star$ to affine $V_{i_1}$. Hence the result follows.
\end{proof}

With minor modification, the proof of Lemma \ref{lem:k-to-k} is valid if the hypothesis $r\geq n$ is changed to $r<n$. Thus we obtain one more lemma as follows.

\begin{lemma}\label{lem:1-dim-tangent-sheaf}
  Suppose $d=n+1$ and $0\leq r\leq n-1$. There exists a cocycle
  \[
  (0,\ldots,0,c^{n-1-r,r},c^{n-r,r-1},\ldots,c^{n-1,0})
  \]
  in $\mc H^{n-1}_r$ where the components $c^{n-1-q,q}$ are given in \eqref{eq:c}. Each $c^{n-1-q,q}$ is left preferred in $\mc H_r^{n-1-q,q}$.
\end{lemma}

Note that there are $n$ copies of $k$ in the expression of $H^{n-1}(\mc H^{\bullet})$. They respectively come from $\check{\mbf{C}}^{\prime n-1}(\mf U, \mc F_r^0)$ for $0\leq r\leq n-1$. The class represented by the cocycle given in Lemma \ref{lem:1-dim-tangent-sheaf} is nontrivial since $c^{n-1,0}$ represents a nontrivial class. Consider the quasi-isomorphisms $\bar{\bar{\lam}}$ given in \eqref{eq:bar-bar-lambda} and $\ga$ given in Theorem \ref{thm:q-iso-GS}. The quasi-isomorphic image by $\ga\bar{\bar{\lam}}\colon \mc H^{\bullet}_r\to\bar{\mbf{C}}^{\prime\bullet}_{\GS}(\mc A)_r$ is a collection of local sections of the sheaf $\wedge^r \mc T_X$. More precisely, we summarize the fact as

\begin{proposition}\label{prop:1-dim-tangent}
Suppose $d=n+1$. For every $0\leq r\leq n-1$, there is a one-dimensional $k$-submodule of $H^{n-1-r}(X,\wedge^r\mc T_X)$, and consequently $H^{n-1-r}(X,\wedge^r\mc T_X)\neq 0$.
\end{proposition}

\subsubsection{Case 3: $d< n+1$}
This is an easy case, since the complex \eqref{eq:complex-R} is zero. The results are
\[
{}^{I\!I}E^{p,q}_{r,\infty}={}^{I\!I}E^{p,q}_{r,2}=
\begin{cases}
P^{p-2r}_{r+(p-r)(d-1)}, & r< p\leq 2r,\, q=0,\\
Q^{-r}_{r}, & p= r,\, q=0,\\
0, & \text{otherwise,}
\end{cases}
\]
and
\[
H^i(\mc H^{\bullet})\cong
\bigoplus_{r<i}P^{i-2r}_{r+(i-r)(d-1)}\oplus Q^{-i}_{i}.
\]

\subsection{Characterization of smoothness}\label{subsec:char-smooth}

In this section, we give a necessary and sufficient condition under which a hypersurface is smooth.

In the proof (not in the statement) of Theorem \ref{thm:smooth-equiv}, we make use of the following subgroups of $H^2_{\mathrm{GS}}(\aaa)_1$:
\begin{itemize}
\item the subgroup $E_{\mathrm{res}}$ of $2$-classes of the form $[(0,f,0)]$;
\item the subgroup $E_{\mathrm{mult}}$ of $2$-classes of the form $[(m, 0, 0)]$.
\end{itemize}

First of all, based upon the expression of $H^2(\mc H^\bullet_1)$ from \S\ref{subsec:compute-cohomology}, we obtain that $H^2_{\GS}(\mc A)_1$ contains $P^0_d$ as a summand for any $n$ and $d$. Every element $t\in P^0_d$ corresponds to a class in $E_\mathrm{mult}$. Let us consider when $t$ also belongs to $E_\mathrm{res}$.

Since $t\in P^0_d=(S/(\im\pl_{\bou}))_d$, $t$ lifts to an element $\bar{t}$ in $S_d$. We then identify $\bar{t}$ to a global section of $\mc O_X(d)$. For any $V\in\mf V$, $\bar{t}|_V\in \mc A(V)$ determines the left multiplication by $\bar{t}|_V$ on $\mc A(V)$, and so $\bar{t}|_V\circ\pmul$ represents a class in $H^2_{(1)}(\mc A(V),\mc A(V))$  which is independent of the choice of $\bar{t}$. Hence $t\in H^2_{\GS}(\mc A)_1$ is represented by the GS $2$-cocycle $(\bar{t}\circ\pmul,0,0):=((\bar{t}|_V\circ\pmul)_V,0,0)$ which only deforms the local multiplications of $\mc A$. If $\bar{t}|_V\circ\pmul$ happens to be a coboundary for all $V$, we have cochains $s_V\in C^1(\mc A(V),\mc A(V))$ such that $d_{\Hoch}(s_V)=\bar{t}|_V\circ\pmul$. Let $s=(s_V)_V\in\bar{\mbf C}^{\prime 0,1}(\mc A)$ and so $(\bar{t}\circ\pmul,0,0)-(0,-d_{\simp}(s),0)=d_{\GS}(s,0)$. Thus $t=[(\bar{t}\circ\pmul,0,0)]=[(0,-d_{\simp}(s),0)]$ belongs to $E_{\mathrm{mult}} \cap E_{\mathrm{res}}$. In the other direction, if $t\in E_\mathrm{mult}$ is also in $E_\mathrm{res}$, then we assume its representation is $(0,f,0)$. The difference $(\bar{t}\circ\pmul,0,0)-(0,f,0)$ has to be a GS coboundary, say $d_{\GS}(s,0)$. It follows that $\bar{t}|_V\circ\pmul=d_{\Hoch}(s_V)$ for all $V\in\mf V$.

Summarizing, we have $t\in E_{\mathrm{mult}} \cap E_{\mathrm{res}}$ if and only if $\bar{t}|_V\circ\pmul$ is a Hochschild $2$-coboundary for every $V\in \mf V$. Note that $\mc A(V)$ is a localization of $\mc A(U)$ if $V\subseteq U$. It follows that  $\bar{t}|_V\circ\pmul$ is a coboundary of $\mc A(V)$ provided that $\bar{t}|_U\circ\pmul$ is a coboundary of $\mc A(U)$. So this condition is again equivalent to the fact that $\bar{t}|_{U_i}\circ\pmul$ is a coboundary of $A_i$ for all $1\leq i\leq n$. By \S\ref{sec:HHaffine},
 \[
 H^2_{(1)}(A_i,A_i)= A_i\bigg/\biggl(\frac{\pl G_i}{\pl y_0},\ldots,\frac{\pl G_i}{\pl y_{i-1}},\frac{\pl G_i}{\pl y_{i+1}},\ldots,\frac{\pl G_i}{\pl y_3}\biggr)
 \]
and $\bar{t}|_{U_i}\circ\pmul$ is a coboundary if and only if $\bar{t}|_{U_i}$ is sent to zero by the projection $A_i\to H^2_{(1)}(A_i,A_i)$. Since $A_i=k[y_0\ldots,y_{i-1},y_{i+1},\ldots,y_n]/(G_i)$ and
\[
\sum_{j\neq i}y_j\frac{\pl G_i}{\pl y_j}+H_i=d\cdot G_i,
\]
we have
 \[
 H^2_{(1)}(A_i,A_i)= k[y_0,\ldots,y_{i-1},y_{i+1},\ldots,y_n]\bigg/\biggl(\frac{\pl G_i}{\pl y_0},\ldots,\frac{\pl G_i}{\pl y_{i-1}},H_i,\frac{\pl G_i}{\pl y_{i+1}},\ldots,\frac{\pl G_i}{\pl y_n}\biggr).
 \]
Recall the definition of $H_i$ given in \S\ref{subsec:quasi-iso-GS}. There is an algebra map $P^0\to H^2_{(1)}(A_i,A_i)$ defined by $x_j\mapsto y_j$ if $j\neq i$ and $x_i\mapsto 1$, whose kernel is $(x_i-1)P^0$. Thus $t\in E_{\mathrm{res}}$ if and only if $t\in\cap_{i=1}^n(x_i-1)P^0$. Notice that $t$ is homogeneous. If $t=(1-x_i)T_i$ for some $T_i\in P^0$, by comparing the homogeneous components, we conclude that $t$ is annihilated by a power of $x_i$ and $T_i=\sum_{m=0}^{\infty} tx_i^m$ which is actually a finite sum. In the opposite direction, if $t$ is annihilated by a power of $x_i$, then $t=(1-x_i)\sum_{m=0}^{\infty} tx_i^m\in (x_i-1)P^0$. Consequently, we have proven

\begin{lemma}\label{lem:mult-res}
 Let $t\in P^0_d$. Then $t\in E_{\mathrm{res}}$ if and only if $x_i\in\sqrt{\mathrm{ann}_{P^0}^{}(t)}$ for all $1\leq i\leq n$.
\end{lemma}

Next let us recall the work \cite{Gerstenhaber-Schack:survey} by Gerstenhaber and Schack. Starting from their Hodge decomposition for presheaves of commutative algebras
\begin{equation}\label{Hodgedec}
H^i_{\GS}(\mc A) = \bigoplus_{r\in\nan}H^i_{\mathrm{GS}}(\aaa)_r,
\end{equation}
they prove the existence of the HKR type decomposition
\[
H^i_{\GS}(\mc A)\cong\bigoplus_{p+q=i}H^p_{\mathrm{simp}}(\mf V,\wedge^q\mc T)
\]
for any smooth complex projective variety $X$, where $\mc A=\mc O_X|_{\mf V}$ (resp.\ $\mc T=\mc T_X|_{\mf V}$) is the restriction of the structure sheaf (resp.\ tangent sheaf) to an affine open covering $\mf V$ closed under intersection. In particular,
\[
H^2_{\GS}(\mc A)\cong H^0_{\mathrm{simp}}(\mf V,\wedge^2\mc T)\oplus H^1_{\mathrm{simp}}(\mf V,\mc T) \oplus H^2_{\mathrm{simp}}(\mf V,\mc A).
\]
The roles played by the three summands in the deformation of $\mc A$ (viewed as a twisted presheaf) are explained in \cite{DLL:defo-qch}. More concretely, elements in the three summands respectively deform the (local) multiplications, the restriction maps, and the twisting elements of $\mc A$. If $X$ is not necessarily smooth, Gerstenhaber and Schack's result remains partially correct: $H^i_{\GS}(\mc A)_r\cong H^{i-r}_{\mathrm{simp}}(\mf V,\wedge^r\mc T)$ if $r=0$ or $r =i$, and in general $H^i_{\GS}(\mc A)_{i-1}$ contains $H^{1}_{\mathrm{simp}}(\mf V,\wedge^{i-1}\mc T)$ as a $k$-submodule. For $i = 2$, we more precisely have
\begin{equation}
H^1_{\mathrm{simp}}(\mf V,\mc T) \cong E_{\mathrm{res}} \subseteq H^2_{\GS}(\mc A)_1.
\end{equation}
In particular, \eqref{Hodgedec} now yields
\begin{equation}\label{eq:HKR-decomp-second}
H^2_{\GS}(\mc A)\cong H^0_{\mathrm{simp}}(\mf V,\wedge^2\mc T)\oplus H^1_{\mathrm{simp}}(\mf V,\mc T) \oplus H^2_{\mathrm{simp}}(\mf V,\mc A)\oplus E.
\end{equation}
where $E$ is a complement of $E_{\mathrm{res}}$ in $H^2_{\GS}(\mc A)_1$.

When $X$ is a projective hypersurface, the isomorphism $H^p(X,\wedge^q\mc T_X)\cong H^p_{\simp}(\mf V, \wedge^q\mc T)$ holds for all $p$, $q$. The decomposition \eqref{eq:HKR-decomp-second} is equivalent to
\[
HH^2(X)\cong H^0(X,\wedge^2\mc T_X)\oplus H^1(X,\mc T_X) \oplus H^2(X,\mc O_X)\oplus E.
\]

We have thus proven:

\begin{proposition}\label{propHKRdef}
Let $X$ be a projective hypersurface. The following are equivalent:
\begin{enumerate}
\item The HKR decomposition holds for the second cohomology, i.e.
      \[
        HH^2(X)\cong H^0(X,\wedge^2\mc T_X)\oplus H^1(X,\mc T_X) \oplus H^2(X,\mc O_X).
      \]
\item We have $H^1(X,\mc T_X) \cong E_{\mathrm{res}} = H^2_{\GS}(\mc A)_1$.
\end{enumerate}
\end{proposition}

\begin{remark}
In deformation theoretic terms, Proposition \ref{propHKRdef} states that for a projective hypersurface $X$, the HKR decomposition holds for $HH^2(X)$ if and only if every (commutative) scheme deformation of $X$ can be realized by only deforming restriction maps while trivially deforming individual algebras on an affine cover. This is the classical deformation picture for smooth schemes.
\end{remark}

We have the following converse of the HKR theorem for projective hypersurfaces:


\begin{theorem}\label{thm:smooth-equiv}
  Let $X$ be a projective hypersurface. The following are equivalent:
  \begin{enumerate}
    \item $X$ is smooth.
    \item The HKR decomposition holds for all cohomology groups, i.e.
      \[
        HH^i(X)\cong \bigoplus_{p+q=i}H^p(X,\wedge^q\mc T_X),\quad \forall\; i\in\nan.
      \]
    \item The HKR decomposition holds for the second cohomology, i.e.
      \[
        HH^2(X)\cong H^0(X,\wedge^2\mc T_X)\oplus H^1(X,\mc T_X) \oplus H^2(X,\mc O_X).
      \]
  \end{enumerate}
\end{theorem}

\begin{proof}
It remains to prove (3)\,$\Rightarrow$\,(1). Assume $X$ is a hypersurface of degree $d$ in $\mbb P^n$ which is not smooth. According to Proposition \ref{propHKRdef}, it suffices to produce a class in $H^2_{\GS}(\mc A)_1 \setminus E_\mathrm{res}$.
 At least one of the algebras $A_i$ is not smooth, say $A_n$. By Remark \ref{remsmooth}, it follows that $H^2_{(1)}(A_n,A_n)\neq 0$. As before, we know
  \begin{align*}
    H^2_{(1)}(A_n,A_n)&=k[y_0,\ldots,y_{n-1}]\bigg/\biggl(\frac{\pl G_n}{\pl y_0},\cdots,\frac{\pl G_n}{\pl y_{n-1}}, H_n\biggr)\\
    &\cong R\bigg/\biggl(x_n-1,\frac{\pl F}{\pl x_0},\cdots,\frac{\pl F}{\pl x_{n-1}},\frac{\pl F}{\pl x_n}\biggr)\\
    &=P^0/(x_n-1).
  \end{align*}
  Since $P^0/(x_n-1)\neq 0$ this implies that $0\neq x_n^m\in P^0$ for any $m\in\nan$. In particular, $x_n^d\in P^0_d$ presents a non-trivial class in $H^2_{\GS}(\mc A)_1$, and $x_n\notin\sqrt{\mathrm{ann}^{}_{P^0}(x_n^d)}$. By Lemma \ref{lem:mult-res}, $x_n^d\notin E_\mathrm{res}$, which finishes the proof. \end{proof}

\subsection{Examples of intertwined classes}\label{subsec:intertwined}

We are particularly interested in $HH^2(X)$ since it parameterizes the equivalence classes of first order deformations of $X$. We retain the notations used before. On one hand, we have the decomposition \eqref{eq:HKR-decomp-second}. On the other hand, any GS $2$-cocycle
\[
(m,f,c)\in \bar{\mbf C}^{\prime 0,2}(\mc A)\oplus \bar{\mbf C}^{\prime 1,1}(\mc A)\oplus \bar{\mbf C}^{\prime 2,0}(\mc A)
\]
factors as $(m-m^{\ab},0,0)+(m^{\ab},f,0)+(0,0,c)$ under the Hodge decomposition where $m^{\ab}$ depends only on $m$. Since $E\subseteq H^2_{\GS}(\mc A)_1$, the elements in $E$ admit representatives of the form $(m,f,0)$. Normally, neither $(m,0,0)$ nor $(0,f,0)$ is a cocycle. The cocycle is called \emph{untwined} if $(m,0,0)$ or, equivalently $(0,f,0)$ is a cocycle. A $2$-class is called \emph{intertwined} if it has no untwined representative. 

In this section, we will given examples of such intertwined $2$-classes. By the decomposition of $\mc H^\bullet$ and by Theorem \ref{thm:q-iso-GS}, classes in $H^2(\mc H^\bullet_0)$ and $H^2(\mc H^\bullet_2)$ have untwined representatives of the form $(0,0,c)$ and $(m,0,0)$ respectively. It is sufficient to consider $H^2(\mc H^\bullet_1)$. Moreover, we exclude the case $n=1$ since this is affine case.

First of all, by the discussion in \S\ref{subsec:compute-cohomology}, $H^2(\mc H^\bullet_1)$ is the direct sum of $P^0_d$ and $Q^{-2}_2$ if $d<n+1$. Via the quasi-isomorphisms $\mc H^\bullet\to\mc G^\bullet\to \mc E^\bullet\to\bar{\mbf C}_{\GS}^{\prime\bullet}(\mc A)$, any element in $P^0_d$ or $Q^{-2}_2$ gives rise to a GS 2-class of the form $[(m,0,0)]\in H^2_{\GS}(\mc A)$. So intertwined $2$-class never exists if $d<n+1$.

Next, besides $P^0_d$ and $Q^{-2}_2$, $H^2(\mc H^\bullet_1)$ contains $k$ as a direct summand if $d=n+1$. By Proposition \ref{prop:1-dim-tangent}, any nonzero element in $k$ corresponds to a nonzero class in $H^1(X,\mc T_X)$ which clearly admits a representative of the form $(0,f,0)$.

Thus an intertwined class exists only possibly in $\scr S(Z^{n-3}_{d-4})$ in the case $d>n+1$. Necessarily, $n\leq 3$ since $Z^{n-3}=0$ for all $n>3$. Since $n=3$ implies $\scr S(Z^0_{d-4})\subseteq H^2(\mc H^\bullet_0)$, $n=2$ is the unique choice, and so $d>3$. Moreover, by the definition of $Z^{-1}_{d-4}$, the short sequence
\begin{equation}\label{eq:intertwine}
0\To Z^{-1}_{d-4}\To R^3_{d-4}\xrightarrow[\quad]{\pl_{\bov}} R^{}_{d-3}\To 0
\end{equation}
is exact. It follows that $Z^{-1}_{d-4}\neq 0$ only if $d>4$.

We have proven:

\begin{proposition}\label{propint}
	Suppose either $n \neq 2$ or $n = 2$ and $d \leq 4$. Then $H^2_{\mathrm{GS}}(\aaa)$ does not contain an intertwined cohomology class.
\end{proposition}

Now let $d\geq 6$ and $F=x_0^d+x_1^{d-1}x_2$. The map $\pl_{\bov}\colon R^3_1\to R_2^{}$ in \eqref{eq:intertwine} sends $(r_0,r_1,r_2)$ to $r_0x_0+r_1x_1+r_2x_2$, whose kernel is 3-dimensional with a basis $\{(-x_1,x_0,0), (-x_2,0,x_0), (0,-x_2,x_1)\}$. Since $\scr S(Z^{-1}_1)$ arises from $\mc H_1^\bullet$, we consider the double complex
\[
\xymatrix{
	\underline{S_{x_1}\oplus S_{x_2}} \ar[r] & S_{x_1x_2} \\
	S_{x_1}^3\oplus S_{x_2}^3 \ar[r]\ar[u]^-{\pl_{\bou}} & \underline{S_{x_1x_2}^3} \ar[u]^-{\pl_{\bou}} \\
	S_{x_1}\oplus S_{x_2} \ar[r]\ar[u]^-{\pl_{\bov}} & S_{x_1x_2} \ar[u]^-{\pl_{\bov}}\ar[r] & \underline{0}
}
\]
with three entries corresponding to $\mc H_1^2$ underlined. We choose the basis element $(0,-x_2,x_1)$, and so
\[
\scr S(0,-x_2,x_1)=(0,-x_0^4x_1^{-1}x_2^{-2}, x_0^4x_1^{-2}x_2^{-1})\in S^3_{x_1x_2}.
\]
Since $\bou=(dx_0^{d-1}, (d-1)x_1^{d-2}x_2, x_1^{d-1})$, $\pl_{\bou}(\scr S(0,-x_2,x_1))$ is equal to
\[
(d-1)x_1^{d-2}x_2\cdot(-x_0^4x_1^{-1}x_2^{-2})+x_1^{d-1}\cdot x_0^4x_1^{-2}x_2^{-1}=-(d-2)x_0^4x_1^{d-3}x_2^{-1}.
\]
Choose $(0, (d-2)x_0^4x_1^{d-3}x_2^{-1})\in S_{x_1}\oplus S_{x_2}$, and thus $((0, (d-2)x_0^4x_1^{d-3}x_2^{-1}),\scr S(0,-x_2,x_1),0)$ is a 2-cocycle in $\mc H_1^\bullet$.

Let us prove that the class $\mf{c}:=[((0, (d-2)x_0^4x_1^{d-3}x_2^{-1}),\scr S(0,-x_2,x_1),0)]$ is intertwined. Assume it can be written as $[(m',0,0)]+[(0,f',0)]$, then $m':=(m'_1,m'_2)\in\ker\{S_{x_1}\oplus S_{x_2}\to S_{x_1x_2}\}$. Note that $S$, $S_{x_1}$ and $S_{x_2}$ can be regarded as $k$-submodules of $S_{x_1x_1}$ since $S$ is a domain, and that $S_{x_1}\cap S_{x_2}=S$. We then have $m'_1=m'_2$ and so $m'_2\in S$.  It follows that $m'_2+(d-2)x_0^4x_1^{d-3}x_2^{-1}\in\im\{\pl_{\bou}\colon S_{x_2}^3\to S_{x_2}\}$, say
\begin{equation}\label{eq:ex-intertwined}
m'_2+(d-2)x_0^4x_1^{d-3}x_2^{-1}=dx_0^{d-1}a_1+(d-1)x_1^{d-2}x_2a_2+x_1^{d-1}a_3 
\end{equation}
for some $a_1$, $a_2$, $a_3\in S_{x_2}$. By considering their degrees, we have
\[
a_1=\sum_{\substack{0\leq i_0< d\\i_1\geq 0}}\lambda_1^{i_0i_1}x_0^{i_0}x_1^{i_1}x_2^{1-i_0-i_1}
\]
and similarly for $a_2$, $a_3$. The right-hand side of \eqref{eq:ex-intertwined} is
\begin{align*}
	&\sum_{i_1\geq 0}d\lambda_1^{0i_1}x_0^{d-1}x_1^{i_1}x_2^{1-i_1}-\sum_{\substack{1\leq i_0<d\\i_1\geq 0}}d\lambda_1^{i_0i_1}x_0^{i_0-1}x_1^{d+i_1-1}x_2^{2-i_0-i_1}\\
	&+\sum_{\substack{0\leq i_0<d\\i_1\geq 0}}(d-1)\lambda_2^{i_0i_1}x_0^{i_0}x_1^{d+i_1-2}x_2^{2-i_0-i_1}+\sum_{\substack{0\leq i_0<d\\i_1\geq 0}}\lambda_3^{i_0i_1}x_0^{i_0}x_1^{d+i_1-1}x_2^{1-i_0-i_1}.
\end{align*}
Observe that the basis element $x_0^4x_1^{d-3}x_2^{-1}$ never appears in any term of the right-hand side, since $d\geq 6$ and $i_1\geq 0$. Together with the fact $m'_2\in S$, we get a contradiction. Thus $\mf{c}$ is indeed an intertwined class.

We remind the reader that the projective curve $x_0^d+x_1^{d-1}x_2$ has a unique singularity $(0:0:1)$.

Next let us describe how the class deforms $\mc A$ in the case $d=6$. We have $\mf U=\{U_1,U_2\}$ and $\mf V=\{V_1,V_2, V_{12}\}$, and define $\lam\colon \mf V\to \mf U$ by
\[
V_1\mapsto U_1,\quad V_2\mapsto U_2,\quad V_{12}\mapsto U_2.
\]
The algebras $A_1$, $A_2$, $A_{12}$ are expressed as $k[y_0,y_2]/(y_0^6+y_2)$, $k[y_0,y_1]/(y_0^6+y_1^5)$, $k[y_0,y_1,y_1^{-1}]/(y_0^6+y_1^5)$ respectively. By the formula \eqref{eq:bar-bar-lambda}, we obtain a 2-cocycle $(e^0,e^1,0)$ in $\mc E_1^\bullet$ given by
\begin{align*}
	& e^0_{V_1}=0,\\
	& e^0_{V_2}=-4x_0^5x_1^3x_2^{-1}|_{V_2}=-4y_0^5y_1^3\in A_2,\\
	& e^0_{V_{12}}=-4y_0^5y_1^3\in A_{12},\\
	& e^1_{V_{12}\subset V_1}=-(0,-x_0^4x_1^{-1}x_2^{-2}, x_0^4x_1^{-2}x_2^{-1})|_{V_{12}}=(0,y_0^4y_1^{-1},y_0^4y_1^{-2})\in A_{12}^3,\\
	& e^1_{V_{12}\subset V_2}=0.
\end{align*}
So by Theorem \ref{thm:q-iso-GS}, the intertwined cocycle $(m,f,0)$ is given by
\begin{align*}
	& m^{}_{V_2}=-4y_0^5y_1^3\pmul^{}_{A_2},\\
	& m^{}_{V_{12}}=-4y_0^5y_1^3\pmul^{}_{A_{12}},\\
	& f^{}_{V_{12}\subset V_1}=\biggl(-y^4_0y_1^{-2}\frac{\ppl}{\ppl y_0}+(y_0^4y_1^{-1}-y_0^4y_1^{-2})\frac{\ppl}{\ppl y_1}\biggr)\circ\rho^{V_1}_{V_{12}}
\end{align*}
and other components equal to zero, where $\rho^{V_1}_{V_{12}}\colon A_1\to A_{12}$ is the restriction map.

Unfortunately, the authors have not found any intertwined class in the case $d=5$. So we pose the following open question:

\textbf{Question:} Does an intertwined 2-class exist for a degree 5 curve in $\mbb{P}^2$?

\subsection{The second cohomology groups of quartic surfaces}\label{subsec:quartic-surface}

As we exhibited in \S\ref{subsec:intertwined}, intertwined 2-classes exist for some non-smooth curves. In contrast, by Proposition \ref{propint} such classes do not exist for higher dimensional hypersurfaces, whence for these it suffices to study $2$-cocycles of the form $(m,0,0)$, $(0,f,0)$ and $(0,0,c)$ separately. Among projective hypersurfaces, we are particularly interested in quartic surfaces in $\mbb P^3$.

From now on, let $X$ be a projective quartic surface in $\mbb P^3$, i.e.\ $n=3$ and $d=4$. By the discussion in \S\ref{subsec:compute-cohomology},
\begin{align*}
  H^2_{\GS}(\mc A)_0&\cong k;\\
  H^2_{\GS}(\mc A)_1&\cong k\oplus P^0_4;\\
  H^2_{\GS}(\mc A)_2&\cong k\oplus Q^{-2}_2.
\end{align*}

Now let us make the three deformations arising from the three components ``$k$'' above explicit, following Lemma \ref{lem:1-dim-tangent-sheaf} and formula \eqref{eq:c}. A direct computation shows that
\begin{align*}
  c^{2,0}_{123}&=x_1^{-1}x_2^{-1}x_3^{-1}\frac{\pl F}{\pl x_0}, \\
  c^{1,1}_{12}&=x_1^{-1}x_2^{-1}\biggl(\frac{\pl F}{\pl x_3}\mf f_0-\frac{\pl F}{\pl x_0}\mf f_3\biggr), \\
  c^{1,1}_{13}&=x_1^{-1}x_3^{-1}\biggl(-\frac{\pl F}{\pl x_2}\mf f_0+\frac{\pl F}{\pl x_0}\mf f_2\biggr), \\
  c^{1,1}_{23}&=x_2^{-1}x_3^{-1}\biggl(\frac{\pl F}{\pl x_1}\mf f_0-\frac{\pl F}{\pl x_1}\mf f_3\biggr), \\
  c^{0,2}_{1}&=x_1^{-1}\biggl(\frac{\pl F}{\pl x_3}\mf f_{02}-\frac{\pl F}{\pl x_2}\mf f_{03}+\frac{\pl F}{\pl x_0}\mf f_{23}\biggr), \\
  c^{0,2}_{2}&=x_2^{-1}\biggl(-\frac{\pl F}{\pl x_3}\mf f_{01}+\frac{\pl F}{\pl x_2}\mf f_{03}-\frac{\pl F}{\pl x_0}\mf f_{23}\biggr), \\
  c^{0,2}_{3}&=x_3^{-1}\biggl(\frac{\pl F}{\pl x_2}\mf f_{01}-\frac{\pl F}{\pl x_1}\mf f_{02}+\frac{\pl F}{\pl x_0}\mf f_{12}\biggr).
\end{align*}
We choose a map $\lam\colon \mc V\to \mc U$ by $\lam(V_{j_1\ldots j_r})=U_{j_r}$ if $j_1<\cdots< j_r$, and the algebra $\mc A(V_{j_1\ldots j_r})$ is expressed as $k[y_0,\ldots,y_{j_r-1},y_{j_r+1},\ldots,y_3,y_{j_1}^{-1},\ldots,y_{j_{r-1}}^{-1}]/(G_{j_r})$. By \eqref{eq:bar-bar-lambda}, $c^{2,0}$ gives rise to a $2$-cocycle $(0,0,e^2)$ in $\mc E_0$ by
\[
e^{2}_{V_{123}\subset V_{12}\subset V_1}=-x_1^{-1}x_2^{-1}x_3^{-1}\frac{\pl F}{\pl x_0}\bigg|_{V_{123}}=-y_1^{-1}y_2^{-1}\frac{\pl G_3}{\pl y_0}.
\]
This in turn gives rise to the GS cocycle $(0,0,c)$ by
\[
c^{}_{V_{123}\subset V_{12}\subset V_1}=-y_1^{-1}y_2^{-1}\frac{\pl G_3}{\pl y_0}.
\]
Using \eqref{eq:bar-bar-lambda} again, we obtain a $2$-cocycle $(0,e^1,e^2)$ in $\mc E_1$ from $(0,c^{1,1},c^{2,0})$ with $e^2$ as above and $e^1$ given by
\begin{alignat*}{2}
  e^1_{V_{12}\subset V_1}&=y_1^{-1}\biggl(-\frac{\pl G_2}{\pl y_3}\mf f_0+\frac{\pl G_2}{\pl y_0}\mf f_3\biggr), &\quad e^1_{V_{123}\subset V_1}&=y_1^{-1}\biggl(-\frac{\pl G_3}{\pl y_2}\mf f_0+\frac{\pl G_3}{\pl y_0}\mf f_2\biggr), \\
  e^1_{V_{13}\subset V_1}&=y_1^{-1}\biggl(\frac{\pl G_3}{\pl y_2}\mf f_0-\frac{\pl G_3}{\pl y_0}\mf f_2\biggr), &\quad e^1_{V_{123}\subset V_2}&=y_2^{-1}\biggl(-\frac{\pl G_3}{\pl y_1}\mf f_0+\frac{\pl G_3}{\pl y_1}\mf f_3\biggr), \\
  e^1_{V_{23}\subset V_2}&=y_2^{-1}\biggl(-\frac{\pl G_3}{\pl y_1}\mf f_0+\frac{\pl G_3}{\pl y_1}\mf f_3\biggr), &\quad e^1_{V_{123}\subset V_{12}}&=y_2^{-1}\biggl(-\frac{\pl G_3}{\pl y_1}\mf f_0+\frac{\pl G_3}{\pl y_1}\mf f_3\biggr).
\end{alignat*}
Then we can deduce a GS cocycle $(0,f,0)$ from $(0,e^1,e^2)$. Notice that the expression of $m$ is independent of $e^2$. To have the expression explicitly, by the discussion in \S\ref{sec:mixedcomp}, we only have to replace the formal base element $\mf f_i$ by $\ppl/\ppl y_i$, then compose with the restriction map. For example,
\[
m^{}_{V_{12}\subset V_1}=y_1^{-1}\biggl(-\frac{\pl G_2}{\pl y_3}\frac{\ppl}{\ppl y_0}+\frac{\pl G_2}{\pl y_0}\frac{\ppl}{\ppl y_3}\biggr)\circ\rho^{V_1}_{V_{12}},
\]
and so on. Likewise, we conclude that the cocycle $(e^0,e^1,e^2)$ in $\mc E_2$ induced by $(c^{0,2},c^{1,1},c^{2,0})$ has the form
\begin{align*}
  e^{0}_{V_1}&=\frac{\pl G_1}{\pl y_3}\mf f_{02}-\frac{\pl G_1}{\pl y_2}\mf f_{03}+\frac{\pl G_1}{\pl y_0}\mf f_{23}, \\
  e^{0}_{V_2}&=-\frac{\pl G_2}{\pl y_3}\mf f_{01}+\frac{\pl G_2}{\pl y_2}\mf f_{03}-\frac{\pl G_2}{\pl y_0}\mf f_{13}, \\
  e^{0}_{V_3}&=\frac{\pl G_3}{\pl y_2}\mf f_{01}-\frac{\pl G_3}{\pl y_1}\mf f_{02}+\frac{\pl G_3}{\pl y_0}\mf f_{12}.
\end{align*}
Thus $(e^0,e^1,e^2)$ induces the GS cocycle $(m,0,0)$ given by
\begin{align}
m^{}_{V_1}&=\frac{\pl G_1}{\pl y_3}\cdot\frac{\ppl }{\pl y_0}\cup\frac{\ppl}{\pl y_2}-\frac{\pl G_1}{\pl y_2}\cdot\frac{\ppl }{\pl y_0}\cup\frac{\ppl}{\pl y_3}+\frac{\pl G_1}{\pl y_0}\cdot\frac{\ppl }{\pl y_2}\cup\frac{\ppl}{\pl y_3},\notag\\
m^{}_{V_2}&=-\frac{\pl G_2}{\pl y_3}\cdot\frac{\ppl }{\pl y_0}\cup\frac{\ppl}{\pl y_1}+\frac{\pl G_2}{\pl y_2}\cdot\frac{\ppl }{\pl y_0}\cup\frac{\ppl}{\pl y_3}+\frac{\pl G_2}{\pl y_2}\cdot\frac{\ppl }{\pl y_1}\cup\frac{\ppl}{\pl y_3},\label{eq:unique-biderivation}\\
m^{}_{V_3}&=\frac{\pl G_3}{\pl y_2}\cdot\frac{\ppl }{\pl y_0}\cup\frac{\ppl}{\pl y_1}-\frac{\pl G_3}{\pl y_1}\cdot\frac{\ppl }{\pl y_0}\cup\frac{\ppl}{\pl y_2}+\frac{\pl G_3}{\pl y_0}\cdot\frac{\ppl }{\pl y_1}\cup\frac{\ppl}{\pl y_2}.\notag
\end{align}

Let us look into the dimensions of $H^2_{\GS}(\mc A)_r$ for $r=0$, $1$, $2$. Obviously, $\dim H^2_{\GS}(\mc A)_0=1$. Since $P^0_4=(S/(\im\pl_{\bou}))_4=(R/(\im\pl_{\bou}))_4=R_4/\sum_{i,j=0}^3kx_i\cdot\pl F/\pl x_j$, we have the following inequality
\[
\dim P^0_4=\dim R_4-\dim\sum_{i,j=0}^3x_i\frac{\pl F}{\pl x_j}=35-\dim\sum_{i,j=0}^3kx_i\frac{\pl F}{\pl x_j}\geq 35-16=19.
\]
Next we investigate the upper bound of $\dim P^0_4$. Obviously, $\{x_i\cdot\pl F/\pl x_j\}_{0\leq i\leq 3}$ is $k$-linearly independent provided that $\pl F/\pl x_j\neq 0$. In particular, $\dim\sum_{i=0}^3kx_i\cdot\pl F/\pl x_0=4$ and hence $\dim P^0_4\leq 31$. Interestingly, there is a gap between 31 and other possible dimensions. Let us prove

\begin{lemma}
  If $\dim P^0_4\neq 31$, then $19\leq\dim P^0_4\leq 28$.
\end{lemma}

\begin{proof}
  Suppose $F=x_0^4+f_1x_0^3+f_2x_0^2+f_3x_0+f_4$ where $f_t\in k[x_1,x_2,x_3]$ are homogeneous of degree $t$.

  First of all, let us reduce the lemma to the case $f_1=0$. In fact, $\dim P^0_4=\dim H^2_{\GS}(\mc A)_1-1$ is invariant under isomorphism of surfaces. By an argument similar to the argument presented in the paragraph after Theorem \ref{thm:q-iso-GS}, $f_1$ can be annihilated via the isomorphism
  \[
  x_0\mapsto x_0-\frac{1}{\,4\,}f_1,\quad x_j\mapsto x_j \quad (\,j=1,\,2,\,3\,).
  \]

  Now we safely assume $f_1=0$. Since $\dim P^0_4\neq 31$, one of $\pl F/\pl x_1$, $\pl F/\pl x_2$, $\pl F/\pl x_3$ is nonzero, say $\pl F/\pl x_1\neq 0$. By comparing the degrees of $\pl F/\pl x_0$ and $\pl F/\pl x_1$ with respect to $x_0$, we obtain
  \[
  (\lam_1 x_1+\lam_2 x_2+\lam_3 x_3)\frac{\pl F}{\pl x_1}\in \sum_{l=0}^3kx_l\frac{\pl F}{\pl x_0}
  \]
  for some $\lam_1$, $\lam_2$, $\lam_3\in k$ only when $\lam_1=\lam_2=\lam_3=0$. Hence
  \[
  \sum_{i,j=0}^3kx_i\frac{\pl F}{\pl x_j}\supseteq\sum_{i=0}^3kx_i\frac{\pl F}{\pl x_0}+\sum_{i=1}^3kx_i\frac{\pl F}{\pl x_1}=\bigoplus_{i=0}^3kx_i\frac{\pl F}{\pl x_0}\oplus \bigoplus_{i=1}^3kx_i\frac{\pl F}{\pl x_1}\cong k^7.
  \]
  It follows that $\dim P^0_4\leq 35-7=28$.
\end{proof}

Therefore, $\dim H^2_{\GS}(\mc A)_1\in\{20,\ldots,29\}\cup\{32\}$. The dimension indeed reaches every number in the set. We list some examples in Table \ref{tabel:dim-GS-1} showing this fact. By Lemma \ref{lem:mult-res}, we are able to check if $t\in P^0_4$ also corresponds to a class in $H^1(X,\mc T_X)$. Accordingly, the dimensions of $H^1(X,\mc T_X)$ for these examples can be computed, as listed in the third column.
\renewcommand\arraystretch{1.3}
\begin{table}[htbp]
\begin{tabular}{|c|c|c|c|}
  \hline
  $F$ & $\dim H^2_\mathrm{GS}(\mc A)_1$ & $\dim H^1(X,\mc T_X)$ & $\dim H^2_\mathrm{GS}(\mc A)_2$ \\
  \hline
  $x_0^4+x_1^4+x_2^4+x_3^4$ & 20 & 20 & 1 \\
  $(x_0^2+x_1^2)^2+x_2^4+x_3^4$ & 21 & 4 & 1 \\
  $(x_0^2+x_1^2)^2+(x_2^2+x_3^2)^2$ & 22 & 2 & 2 \\
  $(x_0^2+x_1^2+x_2^2)^2+x_3^4$ & 23 & 2 & 5 \\
  $x_0^4+x_1^4+x_2^4$ & 24 & 1 & 1 \\
  $(x_0^2+x_1^2)^2+x_2^4$ & 25 & 1 & 5 \\
  $(x_0^2+x_1^2+x_2^2+x_3^2)^2$ & 26 & 1 & 17 \\
  $(x_0^2+x_1^2+x_2^2)^2$ & 27 & 1 & 17 \\
  $x_0^4+x_1^4$ & 28 & 1 & 11 \\
  $(x_0^2+x_1^2)^2$ & 29 & 1 & 11 \\
  $x_0^4$ & 32 & 1 & 31 \\
  \hline
\end{tabular}
\caption{dimensions of several groups}
\label{tabel:dim-GS-1}
\end{table}

For $r=2$, the group $Q^{-2}_2$ comes from the complex
\[
S^6_2/\im\pl_{\bov}\xrightarrow[\quad]{\pl_{\bou}} S^4_5/\im\pl_{\bov}\xrightarrow[\quad]{\pl_{\bou}} S_8
\]
by \eqref{eq:quotient-complex}. It fits into a projection
\[
\xymatrix@C=14mm{
R^6_2 \ar[r]^-{\pl_{\bou}}\ar@{->>}[d] & R^4_5 \ar[r]^-{\pl_{\bou}}\ar@{->>}[d] & R_8 \ar@{->>}[d] \\
S^6_2/\im\pl_{\bov} \ar[r]^-{\pl_{\bou}} & S^4_5/\im\pl_{\bov} \ar[r]^-{\pl_{\bou}} & S_8
}
\]
of complexes. By Euler's formula, the projection turns out to be a quasi-isomorphism. Hence $Q^{-2}_2\cong\ker\{\pl_{\bou}\colon R^6_2\to R^4_5\}$. The dimension of the latter is easier to compute than that of $Q^{-2}_2$. Let elements in $R^6_2$ be expressed by
\[
(a_{01},a_{02},a_{03},a_{12},a_{13},a_{13}).
\]
If $F=x_0^4+(x_1^2+x_2^2)^2$, then
\[
\ker\{\pl_{\bou}\colon R^6_2\to R^4_5\}=\{(0,0,0,0,x_2u,-x_1u)\mid u\in R_1\}
\]
and hence $Q_2^{-2}$ is equal to
\[
\{(0,0,0,0,x_2u,-x_1u)+\im\pl_{\bov}\mid u\in S_1\}
\]
whose dimension is $4$; if $F=(x_0^2+x_1^2+x_2^2+x_3^2)^2$, then $Q_2^{-2}$ is equal to the (direct) sum of
\begin{gather*}
\{(0,x_3u,-x_2u,0,0,x_0u)+\im\pl_{\bov}\mid u\in S_1\},\\
\{(x_3v,0,-x_1v,0,x_0v,0)+\im\pl_{\bov}\mid v\in S_1\},\\
\{(x_2p,-x_1p,0,x_0p,0,0)+\im\pl_{\bov}\mid p\in S_1\},\\
\{(0,0,0,x_3q,-x_2q,x_0q)+\im\pl_{\bov}\mid q\in S_1\},
\end{gather*}
and so $\dim Q_2^{-2}=16$. We omit the computational details and list the dimensions of $H^2_{\GS}(\mc A)_2$ of these examples in the right column of Table \ref{tabel:dim-GS-1}. It is obvious that the lower bound of $\dim H^2_{\GS}(\mc A)_2$ is $1$. However, in the general case, the authors do not know either the upper bound of $\dim H^2_{\GS}(\mc A)_2$, or any gaps between the bound and $1$.

Recall that when $X$ is smooth, the Hodge numbers of $X$ are defined to be $h^{p,q} = \dim H^q(X, \Om_X^p)$. Let $\om_X=\Om_X^2$ be the canonical sheaf of $X$. Then $\om_X\cong \mc O_X$ and by \cite[Cor.\ 3.1.4]{Buchweitz-Flenner:decomp-Atiyah},
\[
H^2_{\GS}(\mc A)\cong HH^2(\om_X)\cong H^2(X,\Om^2_X)\oplus H^1(X,\Om_X)\oplus H^0(X,\mc O_X).
\]
The dimensions of the three summands are $h^{2,2}=1$, $h^{1,1}=20$, $h^{0,0}=1$ respectively. So $\dim H^2_{\GS}(\mc A)_r$ reaches its smallest possible values for $r=0$, $1$, $2$ if $X$ is smooth.

The converse is not true, as there indeed exist non-smooth surfaces with $\dim H^2_{\GS}(\mc A)_1=20$ and $\dim H^2_{\GS}(\mc A)_0=\dim H^2_{\GS}(\mc A)_2=1$. Let us give two examples here.

\begin{example}\label{ex:singu-minimal}
Let $F=x_0^4+x_1^4+x_2^4-4x_2x_3^3+3x_3^4$. We know $\bou=(4x_0^3,4x_1^3,4(x_2^3-x_3^3), -12 x_3^2(x_2-x_3))$. A direct computation shows that $\dim P_4^0=19$ and $\dim Q_2^{-2}=0$. Note that the surface has three isolated singularities $(0:0:1:\zeta^r)$ for $r=0$, 1, 2 where $\zeta$ is a primitive third root of 1. Furthermore, we have $\dim H^1(X,\mc T_X)=11$, in accordance with Theorem \ref{thm:smooth-equiv}.
\end{example}

\begin{example}\label{ex:kummer}
The Kummer surfaces $K_{\mu}$ are a family of quartic surfaces given by
\[
F=(x_0^2+x_1^2+x_2^2-\mu^2x_3^2)^2-\lam pqrs
\]
where
\[
\lam=\frac{3\mu^2-1}{3-\mu^2}
\]
and  $p$, $q$, $r$, $s$ are the tetrahedral coordinates
\begin{alignat*}{2}
p&=x_3-x_2-\sqrt{2}x_0, &\quad q&=x_3-x_2+\sqrt{2}x_0, \\
r&=x_3+x_2+\sqrt{2}x_1, &\quad s&=x_3+x_2-\sqrt{2}x_1.
\end{alignat*}
When $\mu^2\neq 1/3$, $1$, or $3$, $K_{\mu}$ has $16$ isolated singularities which are ordinary double points. In this case, one can check that $\bou$ is a regular sequence in $R$. Thus $\dim P_4^0=19$ and $\dim Q_2^{-2}=0$. We also have $\dim H^1(X,\mc T_X)=1$, in accordance with Theorem \ref{thm:smooth-equiv}.
\end{example}

The examples given above with $\dim H^0(X, \wedge^2\mc T_X)=\dim H^2_{\GS}(\mc A)_2=1$ are all integral, and vice versa. We will give two examples to show this condition is neither necessary nor sufficient for integrality of $X$.

\begin{example}
Let $F=(x_0^2+x_1^2+2x_2^2)(x_0^2+x_1^2+2x_3^2)$. We can easily prove $Q^{-2}_2=0$ and hence $\dim H^0(X, \wedge^2\mc T_X)=1$. However, this is not integral.
\end{example}

\begin{example}
Let $F=x_0^4+x_1^3x_2$. This gives rise to an integral scheme. But $Q_2^{-2}$ is spanned by
\[
(0,0,0,0,x_1u,-x_2u)+\im\pl_{\bov},\quad u\in \{x_0,x_1,x_2,x_3\}
\]
which is $4$-dimensional.
\end{example}

According to our general results, for a smooth K3 surface, we have $P^0_4 = E_{\mathrm{mult}} \subseteq H^2_{\mathrm{GS}}(\aaa)_1 = E_{\mathrm{res}}$ and $\dim P^0_4 = 19$.
To end this section, let us present the resulting two different deformation interpretations of Hochschild $2$-classes in $P^0_4$ for the Fermat quartic surface, i.e. the first example in Table \ref{tabel:dim-GS-1}. Since $\bou=(4x_0^3, 4x_1^3, 4x_2^3, 4x_3^3)$, $P_4^0$ has a basis
\[
\bigl\{x_0^{i_0}x_1^{i_1}x_2^{i_2}x_3^{i_3}\mid i_0+i_1+i_2+i_3=4,\,0\leq i_0,i_1,i_2,i_3\leq 2\bigr\}.
\]
We fix the generators and relations of $\mc A(V)$ for all $V\in\mf V$ as follows:
\begin{align*}
A_1&=k[y_0,y_2,y_3]/(y_0^4+y_2^4+y_3^4+1), & A_2&=k[y_0,y_1,y_3]/(y_0^4+y_1^4+y_3^4+1), \\
A_3&=k[y_0,y_1,y_2]/(y_0^4+y_1^4+y_2^4+1), & A_{12}&=k[y_0,y_1,y_3,y_1^{-1}]/(y_0^4+y_1^4+y_3^4+1), \\
A_{13}&=k[y_0,y_1,y_2,y_1^{-1}]/(y_0^4+y_1^4+y_2^4+1), & A_{23}&=k[y_0,y_1,y_2,y_2^{-1}]/(y_0^4+y_1^4+y_2^4+1), \\
A_{123}&=k[y_0,y_1,y_2,y_1^{-1},y_2^{-1}]/(y_0^4+y_1^4+y_2^4+1).
\end{align*}
For any basis element $x_0^{i_0}x_1^{i_1}x_2^{i_2}x_3^{i_3}\in P_4^0$, there is a deformation $(m,0,0)$ of $\mc A$ given by
\begin{alignat*}{3}
m^{}_{V_1}&=y_0^{i_0}y_2^{i_2}y_3^{i_3}\pmul, & \quad m^{}_{V_2}&=y_0^{i_0}y_1^{i_1}y_3^{i_3}\pmul, & \quad m^{}_{V_3}&=y_0^{i_0}y_1^{i_1}y_2^{i_2}\pmul, \\
m^{}_{V_{12}}&=y_0^{i_0}y_1^{i_1}y_3^{i_3}\pmul, & \quad m^{}_{V_{13}}&=y_0^{i_0}y_1^{i_1}y_2^{i_2}\pmul, & \quad m^{}_{V_{23}}&=y_0^{i_0}y_1^{i_1}y_2^{i_2}\pmul, \\
& & m^{}_{V_{123}}&=y_0^{i_0}y_1^{i_1}y_2^{i_2}\pmul.
\end{alignat*}
We remark that although the same notation $\pmul$ is used, it stands for Hochschild 2-cocycles of individual algebras.

Since in $A_1$ one has
\[
1=4y_0^3\biggl(-\frac{1}{4}y_0\biggr)+4y_2^3\biggl(-\frac{1}{4}y_2\biggr)+4y_3^3\biggl(-\frac{1}{4}y_3\biggr),
\]
it follows that
\[
\pmul=d_{\Hoch}\biggl(-\frac{1}{4}y_0\frac{\ppl}{\ppl y_0}-\frac{1}{4}y_2\frac{\ppl}{\ppl y_2}-\frac{1}{4}y_3\frac{\ppl}{\ppl y_3}\biggr).
\]
Similarly, for $A_2$ and $A_3$, we respectively have
\begin{gather*}
\pmul=d_{\Hoch}\biggl(-\frac{1}{4}y_0\frac{\ppl}{\ppl y_0}-\frac{1}{4}y_1\frac{\ppl}{\ppl y_1}-\frac{1}{4}y_3\frac{\ppl}{\ppl y_3}\biggr),\\
\pmul=d_{\Hoch}\biggl(-\frac{1}{4}y_0\frac{\ppl}{\ppl y_0}-\frac{1}{4}y_1\frac{\ppl}{\ppl y_1}-\frac{1}{4}y_2\frac{\ppl}{\ppl y_2}\biggr).
\end{gather*}
The three preimages are denoted by $s_1$, $s_2$, $s_3$. By abuse of notation, they also denote 1-cochains of the algebras $A_{12}$, $A_{13}$ and so on. Then we have
\begin{align*}
m^{}_{V_1}&=d_{\Hoch}(y_0^{i_0}y_2^{i_2}y_3^{i_3}s_1), & m^{}_{V_2}&=d_{\Hoch}(y_0^{i_0}y_1^{i_1}y_3^{i_3}s_2), \\
m^{}_{V_3}&=d_{\Hoch}(y_0^{i_0}y_1^{i_1}y_2^{i_2}s_3), & m^{}_{V_{12}}&=d_{\Hoch}(y_0^{i_0}y_1^{i_1}y_3^{i_3}s_2), \\
m^{}_{V_{13}}&=d_{\Hoch}(y_0^{i_0}y_1^{i_1}y_2^{i_2}s_3), & m^{}_{V_{23}}&=d_{\Hoch}(y_0^{i_0}y_1^{i_1}y_2^{i_2}s_3), \\
m^{}_{V_{123}}&=d_{\Hoch}(y_0^{i_0}y_1^{i_1}y_2^{i_2}s_3).
\end{align*}
We choose a map $\lam\colon \mc V\to \mc U$ by $\lam(V_{j_1\ldots j_r})=U_{j_r}$ if $j_1<\cdots< j_r$. We thus obtain an equivalent deformation $(0,f,0)$ whose nonzero components of $f$ are
\begin{align*}
f^{}_{V_{12}\subseteq V_1}&=y_0^{i_0}y_1^{i_1}y_3^{i_3}s_2\circ \rho^{V_1}_{V_{12}}-\rho^{V_1}_{V_{12}}\circ y_0^{i_0}y_2^{i_2}y_3^{i_3}s_1,\\
f^{}_{V_{13}\subseteq V_1}&=y_0^{i_0}y_1^{i_1}y_2^{i_2}s_3\circ \rho^{V_1}_{V_{13}}-\rho^{V_1}_{V_{13}}\circ y_0^{i_0}y_2^{i_2}y_3^{i_3}s_1,\\
f^{}_{V_{23}\subseteq V_2}&=y_0^{i_0}y_1^{i_1}y_2^{i_2}s_3\circ \rho^{V_2}_{V_{23}}-\rho^{V_2}_{V_{23}}\circ y_0^{i_0}y_1^{i_1}y_3^{i_3}s_2,\\
f^{}_{V_{123}\subseteq V_1}&=y_0^{i_0}y_1^{i_1}y_2^{i_2}s_3\circ \rho^{V_1}_{V_{123}}-\rho^{V_1}_{V_{123}}\circ y_0^{i_0}y_2^{i_2}y_3^{i_3}s_1,\\
f^{}_{V_{123}\subseteq V_2}&=y_0^{i_0}y_1^{i_1}y_2^{i_2}s_3\circ \rho^{V_2}_{V_{123}}-\rho^{V_2}_{V_{123}}\circ y_0^{i_0}y_1^{i_1}y_3^{i_3}s_2,\\
f^{}_{V_{123}\subseteq V_{12}}&=y_0^{i_0}y_1^{i_1}y_2^{i_2}s_3\circ \rho^{V_{12}}_{V_{123}}-\rho^{V_{12}}_{V_{123}}\circ y_0^{i_0}y_1^{i_1}y_3^{i_3}s_2.
\end{align*}

\providecommand{\bysame}{\leavevmode\hbox to3em{\hrulefill}\thinspace}
\providecommand{\MR}{\relax\ifhmode\unskip\space\fi MR }
\providecommand{\MRhref}[2]{%
  \href{http://www.ams.org/mathscinet-getitem?mr=#1}{#2}
}
\providecommand{\href}[2]{#2}

\end{document}